\documentclass[12pt]{amsart}
\usepackage{amsmath, amsthm, amssymb}
\usepackage{amsfonts}
\usepackage{fullpage}
\usepackage{color}
\usepackage{hyperref}
\usepackage{enumitem}
\usepackage{bm}
\usepackage{verbatim}

\usepackage[dvipsnames]{xcolor}
\usepackage{graphicx}  
\DeclareGraphicsExtensions{.pstex,.eps}

\usepackage{xcolor}

\renewcommand{\ell}{{l}}  

\newcommand{\bY}{\mathbf{Y}}
\newcommand{\bH}{\mathbf{H}}
\newcommand{\bh}{\mathbf{h}}

\newcommand{\R}{{\mathbb{R}}}
\newcommand{\Z}{{\mathbb Z}}

\renewcommand{\S}{{\mathbb S}}

\newcommand{\N}{{\mathbb N}}

\renewcommand{\AA}{{\mathcal A}} 
\newcommand{\BB}{{\mathcal B}}

\newcommand{\EE}{{\mathcal E}}

\newcommand{\bEE}{{\bm{ \mathcal E}}}
\newcommand{\MM}{{\mathcal M}}
\newcommand{\FF}{{\mathcal F}}

\newcommand{\GG}{{\mathcal G}}
\newcommand{\HH}{{\mathcal H}}
\newcommand{\PP}{{\mathcal P}}
\newcommand{\QQ}{{\mathcal Q}}

\renewcommand{\SS}{{\mathcal S}}

\newcommand{\tri}{|\!|\!|}

\newcommand{\tF}{{\tilde{F}}}
\newcommand{\tia}{{\tilde a}}

\newcommand{\tc}{{\tilde c}}
\newcommand{\tih}{{\tilde h}}
\newcommand{\ts}{{\tilde s}}
\newcommand{\tR}{{\tilde{R}}}
\newcommand{\tg}{{\tilde g}}

\newcommand{\tp}{{\tilde p}}

\newcommand{\tr}{\operatorname{tr}}

\renewcommand{\Re}{\mathop{\rm Re}\nolimits}
\renewcommand{\Im}{\mathop{\rm Im}\nolimits}

\DeclareMathOperator{\Ric}{Ric}
\DeclareMathOperator{\tRic}{{\widetilde{Ric}}}

\theoremstyle{plain}

\newtheorem{thm}{Theorem}[section]
\newtheorem{prop}[thm]{Proposition}

\newtheorem{lemma}[thm]{Lemma}

\theoremstyle{definition}
\newtheorem{ex}{Example}[thm]

\newtheorem{rem}{Remark}[thm]

\numberwithin{equation}{section}

\newcommand{\propref}[1]{Proposition~\ref{#1}}

\def\squarebox#1{\hbox to #1{\hfill\vbox to #1{\vfill}}}

\newcommand{\<}{\langle}
\renewcommand{\>}{\rangle}

\renewcommand{\d}{\partial}
\newcommand{\ep}{\epsilon}
\newcommand{\lV}{\lVert}
\newcommand{\rV}{\rVert}

\def\be{{\beta}}
\def\ga{\gamma}
\def\Ga{\Gamma}
\def\de{\delta}
\def\De{\Delta}
\def\ep{\epsilon}

\def\la{\lambda}
\def\La{\Lambda}
\def\si{\sigma}
\def\Si{\Sigma}
\def\om{\omega}

\def\nab{\nabla}

\def\al{\alpha}

\renewcommand{\dh}{{\bm\delta} h}
\newcommand{\dA}{{\bm\delta} A}

\newcommand{\dla}{{\bm\delta} \lambda}
\newcommand{\dSS}{{\bm\delta} \SS}

\newcommand{\bmF}{\mathbf{F}}
\newcommand{\tbmF}{\tilde{\mathbf{F}}}

\title[Skew Mean Curvature Flow]
{Local well-posedness of the  Skew mean curvature flow for small data in $d\geq 2$ dimensions}

\author[J. Huang]
{Jiaxi Huang}

\author[D. Tataru]
{Daniel Tataru}

\address{Beijing International Center for Mathematical Research, Peking University, Beijing 100871, P. R. China}
\email{huangjiaxi@bicmr.pku.edu.cn}

\address{Department of Mathematics, University of California, Berkeley \\
 Berkeley, CA 94720, USA}
\email{tataru@math.berkeley.edu}

\subjclass[2010]{Primary: 35Q55; Secondary: 53E10.}

\keywords{Skew mean curvature flow, local well-posedness, low regularity, small data}

\thanks{J. Huang was supported by the NSFC Grant No. 12071010}

\thanks{D. Tataru was supported by the NSF grant DMS-2054975 as well as by a Simons Investigator grant from the Simons Foundation.   }

\begin{document}

\begin{abstract}
The skew mean curvature flow is an evolution equation for $d$ dimensional ma\-nifolds
embedded in $\mathbb{R}^{d+2}$ (or more generally, in a Riemannian manifold). It can be viewed as a Schr\"odinger analogue of the mean curvature 
flow, or alternatively  as a quasilinear version of the Schr\"odinger Map equation.  In an earlier paper, 
the authors introduced a harmonic/Coulomb gauge formulation of 
the problem, and used it to prove small data local well-posedness in dimensions $d \geq 4$.
In this article, we prove small data local well-posedness in low-regularity Sobolev spaces for the skew mean curvature flow in dimension $d\geq 2$. This is achieved by introducing a new, 
heat gauge formulation of the equations, which turns out to be more robust in low dimensions.
\end{abstract}

\date{\today}
\maketitle


\setcounter{tocdepth}{1}
\pagenumbering{roman} \tableofcontents \newpage \pagenumbering{arabic}

\section{Introduction}

In this article we continue our study of the 
local well-posedness for the skew mean curvature flow (SMCF).
This  is a nonlinear Schr\"odinger type flow 
modeling the evolution of a $d$ dimensional oriented manifold  embedded into a fixed oriented $d+2$ dimensional manifold. It can be seen as a Schr\"odinger analogue of the well studied mean curvature flow. 

In earlier work \cite{HT21}, we have considered 
the (SMCF) flow in higher dimension $d \geq 4$,
and proved local well-posedness for small initial data
in low regularity Sobolev spaces.
This was achieved by developing a suitable harmonic/Coulomb gauge formulation of the  equations, which allowed us to reformulate the problem as a quasilinear Schr\"odinger evolution.

In this article, we consider the small data local well-posedness for the skew mean curvature flow in low dimensions $d \geq 2$, also for low regularity initial data.
As the earlier harmonic/Coulomb gauge formulation has issues 
in low dimensions, here we introduce an alternative 
heat gauge, which resolves these difficulties.

\subsection{The (SMCF) equations}
Let $\Sigma^d$ be a $d$-dimensional oriented manifold, and $(\mathcal{N}^{d+2},g_{\mathcal{N}})$ be a $d+2$-dimensional oriented Riemannian manifold. Let $I=[0,T]$ be an interval and $F:I\times \Sigma^d \rightarrow \mathcal{N}$ be a 
one parameter family of immersions. This induces a time dependent Riemannian structure 
on $\Sigma^d$. For each $t\in I$, we denote the submanifold by $\Sigma_t=F(t,\Sigma)$,  its tangent bundle by $T\Sigma_t$, and its normal bundle by $N\Sigma_t$ respectively. For an arbitrary vector $Z$ at $F$ we denote by $Z^\perp$
its orthogonal projection onto $N\Sigma_t$.
The mean curvature  $\mathbf{H}(F)$ of $\Sigma_t$ can be identified naturally with a section of the normal bundle $N\Sigma_t$.

  The normal bundle $N\Sigma_t$ is a rank two vector bundle with a naturally induced complex structure $J(F)$ which simply rotates a vector in the normal space by $\pi/2$ positively. Namely, for any point $y=F(t,x)\in \Sigma_t$ and any normal vector $\nu\in N_{y}\Sigma_t$, we define $J(F)\in N_{y}\Sigma_t$ as the unique vector with the same length
  so that 
  \[
  J(F)\nu\bot \nu, \qquad \omega(F_{\ast}(e_1), F_{\ast}(e_2),\cdots F_{\ast}(e_d), \nu,  J(F)\nu)>0,
  \]
  where $\om$ is the volume form of $\mathcal{N}$ and $\{e_1,\cdots,e_d\}$ is an oriented basis of $\Sigma^d$. The skew mean curvature flow (SMCF) is defined by the initial value problem
\begin{equation}           \label{Main-Sys}
\left\{\begin{aligned}
&(\d_t F)^{\perp}=J(F)\mathbf{H}(F),\\
&F(\cdot,0)=F_0,
\end{aligned}\right.
\end{equation}
which evolves a codimension two submanifold along its binormal direction with a speed given by its mean curvature.

The (SMCF) was derived both in physics and mathematics. 
The one-dimensional (SMCF) in the Euclidean space $\R^3$ is the well-known vortex filament equation (VFE)
\begin{align*}
\d_t \ga=\d_s \ga\times \d_s^2 \ga,
\end{align*}
where $\ga$ is a time-dependent space curve, $s$ is its arc-length parameter and $\times$ denotes the cross product in $\R^3$. The (VFE) was first discovered by Da Rios \cite{DaRios1906} in 1906 in the study of the free motion of a vortex filament.

The (SMCF) also arises in the study of asymptotic dynamics of vortices in the context of superfluidity and superconductivity. For the Gross-Pitaevskii equation, which models the wave function associated with a Bose-Einstein condensate, physics evidence indicates that the vortices would evolve along the (SMCF). An incomplete verification was attempted by Lin \cite{LinT00} for the vortex filaments in three space dimensions. For higher dimensions, Jerrard \cite{Je02} proved this conjecture when the initial singular set is a codimension two sphere with multiplicity one.

The other motivation is that the (SMCF) naturally arises in the study of the hydrodynamical Euler equation. A singular vortex in a fluid is called a vortex membrane in higher dimensions if it is supported on a codimension two subset. The law of locally induced motion of a vortex membrane can be deduced from the Euler equation by applying the Biot-Savart formula. Shashikanth \cite{Sh12} first investigated the motion of a vortex membrane in $\R^4$ and showed that it is governed by the two dimensional (SMCF), while Khesin \cite{Kh12} then generalized this conclusion to any dimensional vortex membranes in Euclidean spaces.

From a mathematical standpoint, the (SMCF) equation is a canonical geometric flow for codimension two submanifolds which can be viewed as the Schr\"odinger analogue of the well studied mean curvature flow. In fact, the infinite-dimensional space of codimension two immersions of a Riemannian manifold admits a generalized Marsden-Weinstein sympletic structure, and hence the Hamiltonian flow of the volume functional on this space is verified to be the (SMCF). Haller-Vizman \cite{HaVi04} noted this fact where they studied the nonlinear Grassmannians.
For a detailed mathematical derivation of these equations we refer the reader  to the article \cite[Section 2.1]{SoSun17}. 

The one dimensional case of this problem has been extensively studied. This is because the one dimensional (SMCF) flow
agrees the classical  Schr\"odinger Map type equation, provided that one chooses suitable coordinates, i.e. the arclength
parametrization. As such, it exhibits many special properties 
(e.g. complete integrability) which are absent in higher 
dimensions.   For more details we refer the reader to the survey article of Vega~\cite{Ve14}.

The study of higher dimensional (SMCF), on the other hand, 
is far less developed.  Song-Sun \cite{SoSun17} proved the local existence of (SMCF) with a smooth, compact oriented surface as the initial data in two dimensions, then Song \cite{So19} generalized this result to compact oriented manifolds for all $d\geq 2$ and also  proved a corresponding uniqueness result.  Song \cite{So17} also proved that the Gauss map of a $d$ dimensional (SMCF) in $\R^{d+2}$ satisfies a Schr\"odinger Map type equation but relative to the varying metric. More recently, Li \cite{Li1,Li2} considered a class of transversal small pertubations of Euclidean planes under the (SMCF) and proved a global regularity result for small initial data.

This article is instead concerned with the case when $\Sigma^d = \R^d$, i.e. where $\Sigma_t$ has a trivial topology. We will further restrict to the case 
when $\mathcal{N}^{d+2}$ is the Euclidean space $\R^{d+2}$. Thus, the reader should visualize $\Sigma_t$ as an asymptotically flat codimension two submanifold 
of $\R^{d+2}$. 

Such manifolds $\Si=\R^d$ with $d\geq 4$ were already considered in our earlier work \cite{HT21}, where we proved the local well-posedness for small data in low-regularity Sobolev spaces. Here we consider instead the lower dimensional case, 
namely the dimensions $d = 2,3$.
A key role in both \cite{HT21} and in this article  was played by our gauge choices,  which are discussed next.

\subsection{ Gauge choices for (SMCF)}

There are two components for the gauge choice, which are briefly discussed here and in full detail in Section~\ref{Sec-gauge}:

\begin{enumerate}
    \item The choice of coordinates on $I \times \Sigma$.
    \item The choice of an orthonormal frame on $I \times N\Sigma$.
\end{enumerate}

Indeed, as written above in \eqref{Main-Sys}, the (SMCF) equations  are independent of the choice of coordinates in $I \times \Sigma$; here we include 
the time interval $I$ to emphasize that coordinates may be chosen in a time dependent fashion. The manifold $\Sigma^d$ simply serves to provide a parametrization for the moving manifold $\Sigma_t$; it  determines the topology of $\Sigma_t$, but nothing else. Thus, the (SMCF) system written in the form \eqref{Main-Sys} should be seen as a geometric evolution, with a large gauge group, namely the group of time dependent changes of coordinates in $I \times \Sigma$.  One may think of the 
gauge choice here as having two components, (i) the choice of coordinates at the initial time,  and (ii) the time evolution 
of the coordinates. One way to describe the latter choice
is to rewrite the equations in the form 
\begin{equation*}       
\left\{\begin{aligned}
&(\d_t - V \partial_x) F =J(F)\mathbf{H}(F),\\
&F(\cdot,0)=F_0,
\end{aligned}\right.
\end{equation*}
where the vector field $V$ can be freely chosen, and captures the time evolution of the  coordinates. Indeed, some of the earlier papers \cite{SoSun17} and \cite{So19} on (SMCF) use this formulation with $V = 0$. This would seem to simplify the equations, however it introduces difficulties at the level of comparing solutions.This is because the regularity of the map $F$ is no longer determined by the regularity of the second fundamental form, and instead there is a loss 
of derivatives which may only be avoided if the initial data is assumed to have extra regularity. This loss is what prevents a complete low regularity theory in that approach.

In our earlier work \cite{HT21} in dimension $d \geq 4$, we choose \emph{harmonic coordinates} on $\Sigma$, separately at each time. This implicitely fixes $V$, which may be obtained as the solution of an appropriate  elliptic equation. The same approach could be made to work in dimension $d = 3$, if one uses a more careful study of the linearized equation as in the present paper. Unfortunately this does not seem to work well in two dimensions, essentially due to a lack of  sufficient control on the metric at low regularity, which is caused by a lack of decay of the fundamental solution for the Laplacian. 

To rectify this issue, in the present paper we use instead a \emph{heat gauge}, where the coordinates and implicitely the metric are determined dynamically via a heat flow. This in particular requires also a good choice of coordinates at the initial time; there, we fall back to the 
harmonic coordinates. In dimension three and higher, this is all that is needed, and in effect both gauge choices, i.e. the heat gauge and the harmonic gauge, work equally well.  However, in two dimensions the harmonic coordinates fail to yield the needed low frequency decay 
of the metric. We rectify this by adding an a-priori low frequency assumption on the metric in suitable coordinates, and then propagate this in time via the heat gauge.

We now discuss the second component of the gauge choice, namely the orthonormal frame in the normal bundle. Such a choice is needed in order to fix the second fundamental form for $\Sigma$; indeed, the (SMCF) is most naturally interpreted as a nonlinear Schr\"odinger evolution for the second fundamental form of $\Sigma$.  In our earlier paper \cite{HT21} we use the Coulomb gauge. But that seems to no longer be well behaved  
in two dimensions, so we replace it again with a heat flow.
In this context, this strategy is reminiscent of the work of 
the second author and collaborators for the Chern-Simons-Schr\"odinger flow.

\subsection{Scaling and function spaces}
To understand what are the natural thresholds for local well-posedness, it is interesting  
to to consider the scaling properties of the solutions. As one might expect, a clean scaling law is obtained when $\Sigma^d = \R^d$ and $\mathcal{N}^{d+2} = \R^{d+2}$. Then we have the following

\begin{prop}[Scale invariance for (SMCF)]
	Assume that $F$ is a solution of \eqref{Main-Sys} with initial data $F(0)=F_0$, then ${F}_\mu(t,x):=\mu^{-1}F(\mu^2 t,\mu x)$ is a solution of \eqref{Main-Sys} with initial data ${F}_\mu(0)=\mu^{-1}F_0(\mu x)$.
\end{prop}

The above scaling would suggest the critical Sobolev space for our moving surfaces 
$\Sigma_t$ to be $\dot H^{\frac{d}{2}+1}$. However, instead of working directly with the 
surfaces, it is far more convenient to track the regularity at the level of the curvature
$\mathbf{H}(\Sigma_t)$, which scales at the level of $\dot H^{\frac{d}2-1}$. 

For our main result we will use instead 
inhomogeneous Sobolev spaces, and it will suffice to go one derivative above scaling.   There is also a low frequency issue, precisely in two space dimensions where the $L^2$ norm is critical. There we will need to make a slightly stronger assumption on the low frequency part of the initial data.

\subsection{The main result}
Our objective in this paper is to establish the local well-posedness of skew mean 
curvature flow for small data at low regularity. A key observation is that providing a rigorous description  of fractional Sobolev spaces for functions (tensors) on a rough manifold is a delicate matter, which a-priori requires both a good choice of coordinates on the manifold and a good 
frame on the vector bundle (the normal bundle in our case). This is done in the next section, where we fix the gauge and write the equation as a quasilinear Schr\"odinger evolution in a good 
gauge. At this point, we contend ourselves with a less precise formulation of the main result:

\begin{thm}[Small data local well-posedness in dimensions $d\geq 3$]   \label{LWP-thm}
	Let $d\geq 3$, $s>\frac{d}{2}$ and $\si_d=\frac{d}{2}-\de$. Then there exists $\ep_0>0$ sufficiently small such that, for all initial data $\Sigma_0$ with metric $g_0$  and mean curvature $ \mathbf{H}_0$ satisfying 
\begin{equation} \label{small-datad3}
	\||D|^{\si_d} (g_0-I_d)\|_{H^{s+1-\si_d}}\leq \ep_0, \qquad \lV \mathbf{H}_0 \rV_{H^s(\Sigma_0)}\leq \ep_0,
\end{equation}	
relative to some parametrization of $\Sigma_0$,	the skew mean curvature flow \eqref{Main-Sys} for maps from $\R^d$ to the Euclidean space $(\R^{d+2},g_{\R^{d+2}})$ is locally well-posed on the time interval $I=[0,1]$ in a suitable gauge. 
\end{thm}
With a slight adjustment, a similar result holds in dimension $d=2$:

\begin{thm}[Small data local well-posedness in dimension $d=2$]   \label{LWP-thmd=2}
	Let $d = 2$, $s>\frac{d}{2}$ and $\si_d=\frac{d}{2}-\de$. Then there exists $\ep_0>0$ sufficiently small such that, for all initial data $\Sigma_0$ with metric $g_0$ and mean curvature $ \mathbf{H}_0$ satisfying 
	\begin{equation*}
	\||D|^{\si_d} (g_0-I_d)\|_{H^{s+1-\si_d}}\leq \ep_0, \qquad \lV \mathbf{H}_0 \rV_{H^s(\Sigma_0)}\leq \ep_0,
\end{equation*}	
	as well as a low frequency bound for $g_0$ 
	\begin{equation}\label{extra-2d}
	  \|g_0-I_d\|_{Y_{0}^{lo}}<\ep_0,  
	\end{equation}
relative to some parametrization of $\Sigma_0$,	
the skew mean curvature flow \eqref{Main-Sys} for maps from $\R^d$ to the Euclidean space $(\R^{d+2},g_{\R^{d+2}})$ is locally well-posed on the time interval $I=[0,1]$ in a suitable gauge. 
\end{thm}

We continue with some comments on the function spaces in the above theorems:

\begin{itemize}
\item For the metric $g_0$, we use the difference $g_0-I_d$ in the 
above statements in order to emphasize the normalization $g_0 \to I_d$
at infinity.

\item In dimension $d \geq 3$, the $g_0-I_d$ norm in \eqref{small-datad3} only plays a qualitative role, namely to place us in a regime where, in harmonic coordinates, $g_0$ is uniquely determined by the mean curvature $\mathbf H_0$.

\item The $Y_0^{lo}$ norm in \eqref{extra-2d}, defined
in Section~\ref{Sec3}, captures low frequency $\ell^1$
summability properties for $g_0$ with respect 
to cube lattice partitions of $\R^d$. Similar norm appear in our analysis in dimensions $d \geq 3$. The main difference is that in higher dimension, the $Y$
norms of $g_0-I_d$ can be estimated in terms of the 
$H^s$ norm of $ \mathbf H$ in harmonic coordinates. In two dimensions, this estimate borderline fails, so we instead include the $Y_0^{lo}$ bound in the hypothesis.
\end{itemize}

Following the spirit of our earlier work \cite{HT21}, in these results
we consider rough data and provide a full, Hadamard style well-posedness result based on a more modern, frequency envelope approach and  using a paradifferential form for both the full and the linearized equations. For an overview 
of these ideas we refer the reader to the expository paper \cite{IT-primer}. This is unlike any of the prior results, which prove only existence and uniqueness for smooth data.

The favourable gauge mentioned in the theorem is defined in the next section in two steps:

\bigskip

a) at the initial time, where we proceed as in 
\cite{HT21}, and use
\begin{itemize}
    \item Harmonic coordinates on the manifold $\Sigma_0$.
    \item The Coulomb gauge for the orthonormal frame on the normal bundle $N\Sigma_0$.
\end{itemize}

\bigskip

b) dynamically for $t > 0$, where we use

\begin{itemize}
    \item The heat coordinates on the manifolds $\Sigma_t$.
    \item The heat gauge for the orthonormal frame on the normal bundle $N\Sigma$.
\end{itemize}

One simple example of initial data allowed by our theorem consists of graph submanifolds with defining functions $u_1$, $u_2$, of the form
\[
\Sigma_0 = \{ x, u_1(x), u_2(x); x \in \R^d\}
\]
Here one may simply take $u_1$ and $u_2$ to be small in $H^{s+2}$,
with the added low frequency control in the $Y^{lo}_0$ space in dimension two. However, the $H^{s+2}$ control is only needed at high frequency, while at low frequency it suffices to have 
control only in homogeneous norms $\dot H^{\frac{d}{2}+1-\delta}$
with $\delta > 0$. This allows for perturbations which are not small in any uniform norm:

\begin{ex}[Bump-like sub-manifolds]
     Let $\phi_i,\ i=1,2$ be Schwartz functions.
     Then for small $\epsilon > 0$ and $\delta > 0$, the manifold 
     $\Sigma_0$ given by the defining functions
     \[
     u_j = \epsilon^{\frac{d}2 -2 +\delta} \phi_j(\epsilon x)
     \]
     satisfies the hypotheses of our theorem.
     with $\ep>0$ sufficiently small. This manifold is not a small  perturbation of the Euclidean plane in low dimension.
\end{ex}

\begin{ex}[ Sub-manifolds with nontrivial asymptotics]
     For small $\epsilon_j > 0$ and $\delta > 0$, the manifold 
     $\Sigma_0$ given by the defining functions
     \[
     u_j = \epsilon_j (1+x^2)^{\frac{1-\delta}2}
     \]
     satisfies the hypotheses of our theorem.
     with $\ep_j>0$ sufficiently small. This manifold is also not a small  perturbation of the Euclidean plane.
\end{ex}

In the next section we reformulate the (SMCF) equations as a quasilinear Schr\"odinger 
evolution for good scalar complex variable $\la$, which is exactly the second fundamental form
but represented in the good gauge.  There we provide an alternate formulation of the above result, 
as a well-posedness result for the $\la$ equation. In the final section of the paper we 
close the circle and show that one can reconstruct the full (SMCF) flow starting from 
the good variable $\la$.

Once our problem is rephrased as a nonlinear Schr\"odinger evolution,  one may compare its 
study with earlier results on general quasilinear Schr\"odinger evolutions. This story begins 
with the classical work of Kenig-Ponce-Vega \cite{KPV2,KPV3,KPV}, where local well-posedness is established for more regular and localized data. Lower regularity results in translation invariant Sobolev spaces were later established by Marzuola-Metcalfe-Tataru~\cite{MMT3,MMT4,MMT5}. 
The local energy decay properties of the Schr\"odinger equation, as developed earlier in \cite{CS,CKS,Doi,Doi1} play a key role in these results. While here we are using some of the ideas in the above papers, the present problem is both more complex and exhibits  additional structure. Because of this, new ideas and more work are required in order to close the estimates required for both the full problem and for its linearization.

\subsection{An overview of the paper}

Our first objective in this article will be to provide a self-contained formulation
of the (SMCF) flow, interpreted as a nonlinear Schr\"odinger equation for a well chosen variable. This variable, denoted by $\la$, represents the second fundamental form on $\Sigma_t$, in complex notation. We remark that 
in our earlier paper \cite{HT21} we have used instead the complex representation $\psi$ of the mean curvature $\mathbf H$ as the good variable, and $\lambda$ was uniquely determined by $\psi$ via an elliptic  div-curl system. However, solving this system in two dimensions is a delicate matter, which is why here we switch to $\lambda$. The slight downside  of this strategy is that the components of $\lambda$ are not independent, and instead satisfy a set of compatibility conditions which need to be propagated along the flow.

In addition to the main variable $\lambda$,
we will use several dependent variables, as follows:
\begin{itemize}
    \item The Riemannian metric $g$ on $\Sigma_t$.
    \item The magnetic potential $A$, associated to the natural connection on the 
    normal bundle $N \Sigma_t$.
\end{itemize}

These additional variables will be viewed as uniquely determined by our main variable $\la$ and initial metric $g_0$ in a dynamical fashion. This is first done at the initial time by choosing harmonic coordinates on $\Sigma_0$, respectively the Coulomb gauge 
on $N \Sigma_0$. Finally, our dynamical gauge choice also has two
components:

\begin{enumerate}
    \item[(i)] The choice of coordinates on $\Sigma_t$; here we use heat coordinates, with 
    suitable boundary conditions at infinity.
    \item[(ii)] The choice of the orthonormal frame on $N\Sigma_t$; here we use the heat gauge, again assuming flatness at infinity.
\end{enumerate}

To begin this analysis, in the next section we describe the gauge choices, so that by the end we obtain 
\begin{enumerate}
    \item[(a)] A nonlinear Schr\"odinger equation for $\la$, see \eqref{mdf-Shr-sys-2}.  
    \item[(b)] A parabolic system \eqref{par-syst} for the dependent variables
    $\SS=(g,A)$, together with suitable compatibility conditions (constraints).
\end{enumerate}
 
Setting the stage to solve these equations, in Section~\ref{Sec3} we describe the function spaces for both $\la$ and $\SS$. This is done at two levels, first at fixed time, which is needed in order to track data sets,
and then in the space-time setting, which is needed in order to solve both the heat flows \eqref{par-syst} and 
the Schr\"odinger evolution \eqref{mdf-Shr-sys-2}. The fixed time spaces are classical Sobolev spaces,
with matched regularities for all the components. The main space-time norms are the so called local energy spaces associated to the Schr\"odinger evolution, as developed in \cite{MMT3,MMT4,MMT5}. In addition, we 
also use parabolic mixed norm spaces, which capture the regularity gain  in the heat flows.

We begin our analysis in Section~\ref{s:elliptic}, where we 
place the initial data in the harmonic/Coulomb gauge. In higher dimension this analysis was already carried out in our earlier paper
\cite{HT21}. Thus our emphasis here is on the two dimensional case,
where some additional low frequency issues arise in connection with the 
$Y$ norms for the metric $g$. Compared to our earlier article \cite{HT21},
here we are able to improve the analysis and  relax the low frequency
component of the $Y$ norm. This suffices in dimension three, but is only borderline in dimension two, which is why we add the low frequency $Y$ 
bound to the hypothesis of Theorem~\ref{LWP-thmd=2}.

Next, in Section~\ref{Sec-Para}, we consider the solvability of the parabolic system
\eqref{par-syst}. We will do this in two steps. First  we prove that this system is solvable in the space $\EE^s$. Then we prove  space-time bounds for the metric $h$ in local energy spaces; the latter will be needed in the study of the Schr\"odinger evolution \eqref{mdf-Shr-sys-2}.

Finally, we turn our attention to the Schr\"odinger system \eqref{mdf-Shr-sys-2}, whose 
study may be compared with earlier results on general quasilinear Schr\"odinger evolutions. This begins 
with the classical work of Kenig-Ponce-Vega \cite{KPV2,KPV3,KPV}, where local well-posedness is established for more regular and localized data. Lower regularity results in translation invariant Sobolev spaces were later established by Marzuola-Metcalfe-Tataru~\cite{MMT3,MMT4,MMT5}. 
The local energy decay properties of the Schr\"odinger equation, as developed earlier in \cite{CS,CKS,Doi,Doi1} play a key role in these results. Here we are following a similar track, though the present problem is both more complex and exhibits  additional structure. Because of this, new ideas and more work are required in order to close the estimates required for both the full problem and for its linearization.

We divide our approach in several steps.
In Section~\ref{Sec-mutilinear} we establish several multilinear and nonlinear estimates in our space-time
function spaces. These are then used in Section~\ref{Sec-LED} in order to prove local energy decay bounds
first for the linear paradifferential Schr\"odinger flow, and then for a full linear 
Schr\"odinger flow associated to the linearization of our main evolution.

The analysis is completed in Section~\ref{Sec-LWP}, where we combine the linear heat flow bounds and the linear Schr\"odinger bounds 
in order to (i) construct solutions for the full nonlinear Schr\"odinger flow, and (ii) to prove the uniqueness and continuous dependence of the solutions. 
The solutions are initially constructed without reference to the constraint equations, but then we prove that the constraints are indeed satisfied, by propagating them from the initial time.

Last but not least, in the last section we prove that the full set of variables $(\lambda,g,A)$
suffice in order to uniquely reconstruct the defining function $F$ for the evolving surfaces $\Sigma_t$,
as $H^{s+2}_{loc}$ manifolds. More precisely, with respect to the parametrization provided by  our chosen gauge, $F$ has regularity 
\[
\partial_t F,\ \partial_x^2 F \in C[0,1;H^s]. 
\]

\bigskip

\section{The differentiated equations and the gauge choice}
\label{Sec-gauge}

The goal of this section is to introduce our main independent variable $\la$, which represents 
the second fundamental form in complex notation, as well as the 
following auxiliary variables: the metric $g$,
the connection coefficients $A$ for the  normal bundle. 
For $\la$ we start with \eqref{Main-Sys} and derive a nonlinear Sch\"odinger type system \eqref{mdf-Shr-sys-2}, with coefficients depending on  $\SS = (g,A)$. Under suitable gauge conditions, the auxiliary variables $\SS$ are  shown to satisfy a parabolic system \eqref{par-syst}, as well as a natural set of constraints. We conclude the section with a gauge formulation of our main result, see Theorem~\ref{LWP-MSS-thm}.  
Here we will introduce the heat coordinates and heat gauge in detail.
For some of the detailed derivations, we refer to section 2 in \cite{HT21}.

\subsection{The Riemannian metric \texorpdfstring{$g$}{} and the second fundamental form.} Let $(\Sigma^d,g)$ be a $d$-dimensional oriented manifold and let $(\R^{d+2},g_{\R^{d+2}})$ be $(d+2)$-dimensional Euclidean space. Let $\al,\be,\ga,\cdots \in\{1,2,\cdots,d\}$. Considering the immersion $F:\Sigma\rightarrow (\R^{d+2},g_{\R^{d+2}})$, we obtain the induced metric $g$ in $\Sigma$,
\begin{equation}        \label{g_metric}
	g_{\al\be}=\d_{x_{\al}} F\cdot \d_{x_{\be}} F.
\end{equation}
We denote the inverse of the matrix $g_{\al\be}$ by $g^{\al\be}$, i.e.
\begin{equation*}
g^{\al\be}:=(g_{\al\be})^{-1},\qquad g_{\al\ga}g^{\ga\be}=\delta_{\al}^{\be}.
\end{equation*}

Let $\nab$ be the cannonical Levi-Civita connection on $\Si$ associated with the induced metric $g$.
A direct computation shows that on the Riemannian manifold $(\Si,g)$ we have the Christoffel symbols
\begin{align*}
	\Gamma^{\ga}_{\al\be}=\ g^{\ga\si}\Ga_{\al\be,\si}=\ g^{\ga\si}\d^2_{\al\be}F\cdot\d_\si F.
\end{align*}
For any tensor $T^{\al_1\cdots \al_r}_{\be_1\cdots\be_s}dx^{\beta_1}\otimes...dx^{\beta_s}\otimes\frac{\partial }{\partial x^{\alpha_1}}\otimes...\otimes\frac{ \partial }{\partial x^{\alpha_r}}$, we define its \emph{covariant derivative} as follows
\begin{equation}     \label{co_d}
\nab_\ga T^{\al_1\cdots \al_r}_{\be_1\cdots\be_s}=\d_\ga T^{\al_1\cdots \al_r}_{\be_1\cdots\be_s}-\sum_{i=1}^s\Ga_{\ga\be_i}^\si T^{\al_1\cdots \al_r}_{\be_1\cdots \be_{i-1}\si\be_{i+1}\cdots\be_s}+\sum_{j=1}^r\Ga_{\ga\de}^{\al_j} T^{\al_1\cdots\al_{j-1}\de\al_{j+1}\cdots \al_r}_{\be_1\cdots\be_s}.
\end{equation}
Hence, the Laplace-Beltrami operator $\Delta_g$ can be written in the form
\begin{align*}
	\Delta_g f=&\ \tr\nab^2 f=g^{\al\be}(\d_{\al\be}^2f-\Gamma^{\ga}_{\al\be}\d_{\ga} f),
\end{align*}
for any twice differentiable function $f:\Sigma\rightarrow \R$. The curvature $R$ on the Riemannian manifold $(\Sigma,g)$ is given by 
\begin{align*}           
R_{\ga\al\be}^{\si}=\d_{\al} \Gamma_{\be\ga}^{\si}-\d_{\be} \Gamma_{\al\ga}^{\si} +\Gamma_{\be\ga}^m\Gamma_{\al m}^{\si}  -\Gamma_{\al\ga}^m\Gamma_{\be m}^{\si}.
\end{align*}
We also have
\begin{align}             \label{R-2}
R_{\si\ga\al\be}=\d_{\al} \Gamma_{\be\ga,\si}-\d_{\be} \Gamma_{\al\ga,\si} +\Gamma_{\be\si}^m\Gamma_{\al \ga,m}  -\Gamma_{\al\si}^m\Gamma_{\be \ga,m},
\end{align}
and the Ricci curvature
\begin{equation*}
\Ric_{\al\be}=R^{\si}_{\al\si\be}=g^{\si\ga}R_{\ga\al\si\be}.
\end{equation*}

\medskip 

Next, we derive the second fundamental form for $\Sigma$.
Let $\bar{\nab}$ be the Levi-Civita connection in $(\R^{d+2},g_{\R^{d+2}})$ and let $\bh$ be the second fundamental form for $\Sigma$ as an embedded manifold. For any vector fields $u,v\in T_{\ast}\Sigma$, the Gauss relation is 
\begin{equation*}
	\bar{\nab}_u F_{\ast}v=F_{\ast}(\nab_u v)+ \bh(u,v).
\end{equation*}
Then we have
\begin{align*}
	\bh_{\al\be}=\bh(\d_{\al},\d_{\be})=\bar{\nab}_{\d_{\al}} \d_{\be} F-F_{\ast}(\nab_{\d_{\al}}\d_{\be})
	=\d_{\al\be}^2 F-\Gamma_{\al\be}^{\ga} \d_{\ga} F.
\end{align*}
This gives the mean curvature $\bH$ at $F(x)$,
\begin{equation*}
	\bH=\tr_g \bh=g^{\al\be}\bh_{\al\be}=g^{\al\be}(\d^2_{\al\be}F-\Gamma^{\ga}_{\al\be}\d_{\ga} F)=\Delta_g F.
\end{equation*}
Hence, the $F$-equation in \eqref{Main-Sys} is rewritten as
\begin{equation*}  
	(\d_t F)^{\perp}=J(F)\Delta_g F=J(F)g^{\al\be}(\d^2_{\al\be}F-\Gamma^{\ga}_{\al\be}\d_{\ga} F).
\end{equation*}
This equation is still independent of the choice of coordinates 
in $\Sigma^d$.

\subsection{The complex structure equations} \label{complex-mc}
Here we introduce a complex structure on the normal bundle $N\Sigma_t$. This is achieved 
by choosing $\{\nu_1,\nu_2\}$ to be an orthonormal basis of $N\Si_t$ such that
\[
J\nu_1=\nu_2,\quad J\nu_2=-\nu_1.  
\]
Such a choice is not unique; in making it we introduce a second component to our gauge group,  
namely the group of sections of an $SU(1)$ bundle over $I \times \R^d$.  

The vectors $\{ F_1,\cdots,F_d,\nu_1,\nu_2\}$ form a frame at each point on the manifold $(\Sigma,g)$, where $F_{\al}$ for $\al\in\{1,\cdots,d\}$ are defined as
\[
F_{\al}=\d_{\al}F.
\]  
We define the tensors $\kappa_{\al\be},\ \tau_{\al\be}$, the connection coefficients $A_{\al}$ and the temporal component $B$ of the connection in the normal bundle by
\[
\kappa_{\al\be}:=\d_{\al} F_{\be}\cdot\nu_1,\quad \tau_{\al\be}:=\d_{\al} F_{\be}\cdot\nu_2,\quad A_{\al}=\d_{\al}\nu_1\cdot\nu_2,\quad B=\d_t \nu_1\cdot \nu_2.
\]
Then we complexify the normal frame $\{\nu_1,\nu_2\}$ and second fundamental form as 
\begin{equation*}
	m=\nu_1+i\nu_2,\quad \lambda_{\al\be}=\kappa_{\al\be}+i\tau_{\al\be}.
\end{equation*}
Here we can define the \emph{complex scalar mean curvature} $\psi$ to be
\begin{equation}\label{csmc}
\psi:=\tr\la=g^{\al\be}\la_{\al\be}.
\end{equation}
Our objective for the rest of this section will be to interpret the (SMCF) equation as 
a nonlinear Schr\"odinger evolution for $\la$, by making suitable gauge choices.
We remark that the action of sections of the $SU(1)$ bundle is given by 
\begin{equation}     \label{gauge-A}
\psi \to e^{i \theta \psi }, \quad \lambda \to e^{i \theta \lambda}, \quad m \to e^{i\theta} m,
\quad A_\alpha \to A_\alpha - \partial_\alpha \theta.
\end{equation}
for a real valued function $\theta$.

If we differentiate the frame, we obtain a set of structure equations of the following type 
\begin{equation}               \label{strsys-cpf}
\left\{\begin{aligned}
&\d_{\al}F_{\be}=\Gamma^{\ga}_{\al\be}F_{\ga}+\Re(\lambda_{\al\be}\bar{m}),\\
&\d_{\al}^A m=-\lambda^{\ga}_{\al} F_{\ga},
\end{aligned}\right.
\end{equation}
where 
$\d_{\al}^A=\d_{\al}+iA_{\al}$.

\subsection{The Gauss and Codazzi relations} 
\label{s:GC}
The Gauss and Codazzi equations are derived from the equality of second derivatives $\d_{\al}\d_{\be}F_{\ga}=\d_{\be}\d_{\al}F_{\ga}$ for the tangent vectors on the submanifold $\Sigma$ and for the normal vectors respectively.  Here we use the Gauss and Codazzi relations to derive the Riemannian curvature, the first compatibility condition and a symmetry. 

By the structure equations \eqref{strsys-cpf}, we get
\begin{equation}\label{d^2 F_ga}
\begin{aligned}           
\d_{\al}\d_\be F_{\ga}
=&(\d_{\al}\Ga_{\be\ga}^{\si}+\Ga_{\be\ga}^{\mu}\Gamma^{\si}_{\al\mu}-\Re(\la_{\be\ga}\bar{\la}_{\al}^{\si}))F_{\si}+\Re[(\d_{\al}^A\la_{\be\ga}+\Ga^{\si}_{\be\ga}\la_{\al\si})\bar{m}].
\end{aligned}
\end{equation}
Then in view of $\d_{\al}\d_{\be} F_{\ga}=\d_{\be}\d_{\al} F_{\ga}$ and equating the coefficients of the tangent vectors, we obtain
\begin{equation*}
\d_{\al}\Ga_{\be\ga}^{\si}+\Ga_{\be\ga}^{\mu}\Gamma^{\si}_{\al\mu}-\d_{\be}\Ga_{\al\ga}^{\si}-\Ga_{\al\ga}^{\mu}\Gamma^{\si}_{\be\mu}=\Re(\la_{\be\ga}\bar{\la}_{\al}^{\si}-\la_{\al\ga}\bar{\la}_{\be}^{\si}).
\end{equation*}
This gives the Riemannian curvature
\begin{align}          \label{R-la}
R_{\si\ga\al\be}=\Re(\lambda_{\be\ga}\bar{\la}_{\al\si}-\la_{\al\ga}\bar{\la}_{\be\si}),
\end{align}
which is a complex formulation of the Gauss equation.
Correspondingly we obtain the the Ricci curvature
\begin{equation}\label{Ric}
\Ric_{\ga\be}=\Re (\la_{\ga\be}\bar{\psi}-\la_{\ga\al}\bar{\la}_{\be}^{\al}).
\end{equation}
After equating the coefficients of the vector $m$ in \eqref{d^2 F_ga}, we obtain 
\begin{equation*}
\d^A_{\al}\la_{\be\ga}+\Ga_{\be\ga}^{\si}\la_{\al\si}=\d^A_{\be}\la_{\al\ga}+\Ga_{\al\ga}^{\si}\la_{\be\si},
\end{equation*}
By the definition of covariant derivatives \eqref{co_d}, 
we obtain the complex formulation of the Codazzi equation, namely
\begin{equation}         \label{la-commu}
\nab^A_{\al} \la_{\be\ga}=\nab^A_{\be}\la_{\al\ga}.
\end{equation}

\medskip

Next, we use the relation $\d_{\al}\d_{\be}m=\d_{\be}\d_{\al}m$ in order to derive a compatibility condition between the connection $A$ in the normal bundle and the second fundamental form.
Indeed, from $\d_{\al}\d_{\be}m=\d_{\be}\d_{\al}m$ we obtain the commutation relation
\begin{equation} \label{comp-cond_m-pre}
[\d^A_{\al},\d^A_{\be}]m=i(\d_{\al} A_{\be}-\d_{\be} A_{\al})m.
\end{equation}
By \eqref{strsys-cpf}  we have
\begin{align*}
\d^A_{\al}\d^A_{\be} m=&-\d^A_{\al}(\la^{\ga}_{\be} F_{\ga})=-(\d^A_{\al}\la^{\si}_{\be}+\la^{\ga}_{\be}\Ga^{\si}_{\al\ga})F_{\si}-\la^{\ga}_{\be}\Re(\la_{\al\ga}\bar{m})).
\end{align*}
Then multiplying \eqref{comp-cond_m-pre}  by $m$ yields
\begin{align*}
\d_{\al} A_{\be}-\d_{\be} A_{\al}=\Im(\la^{\ga}_{\al}\bar{\la}_{\be\ga}).
\end{align*}
This gives the compatibility condition for the curvature $A$,
\begin{equation}          \label{cpt-AiAj-2}
\nab_{\al} A_{\be}-\nab_{\be} A_{\al}=\Im(\la^{\ga}_{\al}\bar{\la}_{\be\ga}),
\end{equation}
which can be seen as the complex form of the Ricci equations.
We remark that, by equating the coefficients of the tangent vectors in \eqref{comp-cond_m-pre} , we also obtain the relation \eqref{la-commu} again.

\subsection{The motion of the frame \texorpdfstring{$\{F_1,\cdots,F_d,m\}$}{} under (SMCF)} Here we derive the equations of motion for the frame, assuming that the immersion $F$ satisfying \eqref{Main-Sys}. 

We begin by rewriting  the SMCF equations in the form 
\begin{equation*}   
\d_t F=J(F)\bH(F)+V^{\ga} F_{\ga},
\end{equation*}  
where $V^{\ga}$ is a vector field on the manifold $\Sigma$, which 
in general depends on the choice of coordinates. 
By the definition of $m$ and $\la_{\al\be}$, the above $F$-equation is rewritten as
\begin{equation}          \label{sys-cpf}
\d_t F=-\Im (\psi\bar{m})+V^{\ga} F_{\ga}.
\end{equation}

Then by \eqref{sys-cpf}, the structure equations \eqref{strsys-cpf} and the orthogonality relation $m\bot F_{\al}=0$ we obtain the following equations of motion for the frame
\begin{equation}              \label{mo-frame}
\left\{\begin{aligned}
&\d_t F_{\al}=-\Im (\d^A_{\al} \psi \bar{m}-i\la_{\al\ga}V^{\ga} \bar{m})+[\Im(\psi\bar{\la}^{\ga}_{\al})+\nab_{\al} V^{\ga}]F_{\ga},\\
&\d^{B}_t m=-i(\d^{A,\al} \psi -i\la^{\al}_{\ga}V^{\ga} )F_{\al}.
\end{aligned}\right.
\end{equation}
where covariant derivative $\d_t^B=\d_t +iB$ and $B=\<\d_t \nu_1,\nu_2\>$ is the temporal
component of the connection in the normal bundle. 

From this we obtain the evolution equation for the metric $g$. By the definition of the induced metric $g$ \eqref{g_metric}  and \eqref{mo-frame} , we have
\begin{align}         \label{g_dt}
\d_t g_{\al\be}
=&\ 2\Im(\psi\bar{\la}_{\al\be})+\nab_{\al}V_{\be}+\nab_{\be}V_{\al}.
\end{align}
So far, the choice of $V$ has been unspecified; it depends on the choice of coordinates 
on our manifold as the time varies.

\subsection{The motion of  \texorpdfstring{$A$ and $\lambda$}{} under (SMCF)} Here we use the equations of motion for the frame  in \eqref{mo-frame}  in order to repeat the computations of  
Section~\ref{s:GC} with respect to time differentiation, with the aim of computing the time derivative of both $\lambda$ and $A$.
We start from the commutation relation
\begin{equation*}
[\d^{B}_t,\d^A_{\al}]m=i(\d_t A_{\al}-\d_{\al} B)m.
\end{equation*}
In order, for the left-hand side, by \eqref{strsys-cpf}  and \eqref{mo-frame}  we have
\begin{align*}
\d^{B}_t\d^A_{\al} m
=&-[\d^{B}_t\la^{\si}_{\al}+\la^{\ga}_{\al}(\Im(\psi\bar{\la}^{\si}_{\ga})+\nab_{\ga}V^{\si})]F_{\si}+\la^{\ga}_{\al}\Im(\d^A_{\ga}\psi \bar{m}-i\la_{\ga\si}V^{\si}\bar{m}),
\end{align*}
and 
\begin{align*}
\d^A_{\al}\d^{B}_t m
=&-i\nab^A_{\al}(\d^{A,\si} \psi -i\la^{\si}_{\ga}V^{\ga} )F_{\si}-i(\d^{A,\si} \psi -i\la^{\si}_{\ga}V^{\ga} )\Re(\la_{\al\ga}\bar{m}).
\end{align*}
Then by the above three equalities, equating the coefficients of the tangent vectors and the normal vector $m$, we obtain the evolution equation for $\la$
\begin{equation}\label{main-eq-abst}
\d^{B}_t\la^{\si}_{\al}+\la^{\ga}_{\al}(\Im(\psi\bar{\la}^{\si}_{\ga})+\nab_{\ga} V^{\si})=i\nab^A_{\al}(\d^{A,\si} \psi -i\la^{\si}_{\ga}V^{\ga} ),
\end{equation}
as well as the compatibility condition (curvature relation)
\begin{equation}\label{Cpt-A&B}
\d_t A_{\al}-\d_{\al} B = \Re(\la_{\al}^{\ga}\bar{\d}^A_{\ga}\bar{\psi})-\Im (\la^\ga_\al \bar{\la}_{\ga\si})V^\si.
\end{equation}

\subsection{ The equations for the connection \texorpdfstring{$A$}{} in the Coulomb gauge and the heat gauge} 

Here we take the first step towards fixing the gauge, and consider 
the choice of the orthonormal frame in $N\Sigma$. Our starting point 
consists of the curvature relations \eqref{cpt-AiAj-2} at fixed time, 
respectively \eqref{Cpt-A&B} dynamically, together with the gauge group \eqref{gauge-A}. We will fix the gauge in two steps, first in a static, elliptic fashion at the initial time, and then dynamically, using a heat flow, for later times.  

At the initial time $t=0$ we obtain an elliptic system for $A$ by imposing the Coulomb gauge condition
\begin{equation}\label{Coulomb}
    \nab^\al A_\al=0.
\end{equation}
As in \cite{HT21}, this yields
\begin{lemma}[Div-curl system for $A$]\label{Ell-A}
Under the Coulomb gauge condition \eqref{Coulomb}, the connection $A$ solves
	\begin{equation} \label{Ellp-A-pre}
	\nab^\al A_\al=0,\quad
	\nab_\al A_\be-\nab_\be A_\al=\Im(\la^\ga_\al \bar{\la}_{\be\ga}).
	\end{equation}
\end{lemma}

In our previous work \cite{HT21}, the connection coefficients $A$ and $B$ were  determined via the Coulomb gauge condition \eqref{Coulomb} at all times. Instead, in this article we only enforce the Coulomb gauge condition at the initial time $t = 0$, while for $t > 0$ we adopt from \cite{LST} a different gauge condition called the \emph{parabolic gauge} or \emph{heat gauge}. This is defined by the relation
\begin{equation}   \label{heat-gauge}
    \nab^\al A_\al=B,
\end{equation}
which in turn yields a parabolic equation for $A$:

\begin{lemma}[Parabolic equations for $A$] Under the heat gauge condition~\eqref{heat-gauge}, the connection $A$ solves
    \begin{equation}\label{Heat-A-pre}
        (\d_t -\nab_\si\nab^\si) A_\al=\nab^\si\Im(\la^\ga_\al\bar{\la}_{\si\ga})-\Ric_{\al\de}A^\de+\Re(\la^\ga_\al\overline{\nab^A_\ga\psi})-\Im(\la^\ga_\al\bar{\la}_{\ga\si})V^\si.
    \end{equation}
\end{lemma}
\begin{proof}
    Since by \eqref{cpt-AiAj-2} we have
    \begin{align*}
        \nab_\al \nab^\si A_\si=&[\nab_\al,\nab^\si ]A_\si+\nab^\si(\nab_\al A_\si-\nab_\si A_\al)+\nab^\si\nab_\si A_\al\\
        =&-\Ric_{\al\si}A^\si+\nab^\si\Im(\la^\ga_\al\bar{\la}_{\si\ga})+\nab^\si\nab_\si A_\al
    \end{align*}
    Then the equations \eqref{Heat-A-pre} is obtained from \eqref{Cpt-A&B} and the heat gauge \eqref{heat-gauge}.
\end{proof}

\subsection{The equations for the metric \texorpdfstring{$g$}{} in harmonic coordinates and heat coordinates}    
Here we take the next step towards fixing the gauge, by choosing to work in harmonic coordinates at $t=0$ and heat coordinates for $t>0$. Precisely,
at the initial time $t=0$ we will require the coordinate functions $\{x_{\al},\al=1,\cdots,d\}$ to be globally Lipschitz solutions of the elliptic equations 
\begin{equation} \label{h-gauge}
\Delta_g x_{\al}=0.
\end{equation}
This determines the coordinates uniquely modulo time dependent affine transformations. This remaining ambiguity will be removed later on by imposing suitable boundary conditions at infinity. After this, the only remaining degrees of freedom in the choice of coordinates at $t = 0$ will be  given by  translations and
rigid rotations.

Here we interpret the above harmonic coordinate condition at fixed time as an elliptic equation for the metric $g$.
The equations \eqref{h-gauge} may be expressed in terms of the Christoffel symbols $\Ga$, which  must satisfy the condition
\begin{equation}           \label{hm-coord}
	g^{\al\be}\Ga^{\ga}_{\al\be}=0,\quad {\rm for}\ \ga=1,\cdots,d.
\end{equation}
This leads to an equation for the metric $g$:

\begin{lemma} [Elliptic equations of $g$, Lemma 2.4 \cite{HT21}]\label{Ell-g}
In harmonic coordinates, the metric $g$ satisfies
	\begin{equation}\label{h-ell}
	\begin{aligned}      
		g^{\al\be}\d^2_{\al\be}g_{\ga\si}=&\ [-\d_{\ga} g^{\al\be}\d_{\be} g_{\al\si}-\d_{\si} g^{\al\be}\d_{\be} g_{\al\ga}+\d_{\ga} g_{\al\be}\d_{\si} g^{\al\be}]\\
		&+2g^{\al\be}\Ga_{\si\al,\nu}\Ga^{\nu}_{\be\ga}-2\Re (\la_{\ga\si}\bar{\psi}-\la_{\al\ga}\bar{\la}_{\si}^{\al}).
	\end{aligned}
	\end{equation}
\end{lemma}

For latter times $t>0$ we will introduce the \emph{heat gauge}, where we
require the coordinate functions $\{x^\al,\al=1,\cdots,d\}$ to be global Lipschitz solutions of the  heat equations
\[
(\partial_t - \Delta_g - V^\gamma \partial_\gamma) x_\alpha = 0.
\]
This can be rewritten as 
\[ 
\De_g x^\ga=- V^\ga,  
\]
and can also be expressed in terms of the Christoffel symbols $\Ga$, namely,
\begin{equation} \label{Heat-coordinate}
g^{\al\be}\Ga^\ga_{\al\be}=V^\ga.
\end{equation}
Once a choice of coordinates is made 
at the initial time, the coordinates will be 
uniquely determined later on by this gauge condition.

With the advection field $V$ fixed via the heat coordinate condition \eqref{Heat-coordinate}, we can derive a parabolic equation for the metric $g$:
\begin{lemma}[Parabolic equations for metric $g$]
     Under the condition~\eqref{Heat-coordinate}, the metric $g$ solves
     \begin{equation} \label{g-Heat-Eq}
          \begin{aligned}  
         \d_t g_{\mu\nu}-g^{\al\be}\d^2_{\al\be}g_{\mu\nu}=&2\Ric_{\mu\nu}+2\Im(\psi\bar{\la}_{\mu\nu})-2g^{\al\be}\Ga_{\mu\be,\si}\Ga^\si_{\al\nu}\\
         &+\d_{\mu}g^{\al\be}\Ga_{\al\be,\nu}+\d_{\nu}g^{\al\be}\Ga_{\al\be,\mu}.
     \end{aligned}
     \end{equation}
\end{lemma}
\begin{proof}
    By the relation~\eqref{Heat-coordinate} we have
    \begin{align*}
        \nab_\mu V_\nu= g^{\al\be}\d_\mu\Ga_{\al\be,\nu}+\d_\mu g^{\al\be}\Ga_{\al\be,\nu}-g^{\al\be}\Ga_{\mu\nu}^\si \Ga_{\al\be,\si}
    \end{align*}
    Using the expression for $\Ga$ and for the Riemannian curvature \eqref{R-2} we have
    \begin{align*}
        g^{\al\be}(\d_\mu\Ga_{\al\be,\nu}+\d_\nu\Ga_{\al\be,\mu})=&g^{\al\be}[\d_\mu(\d_\al g_{\be\nu}-\frac{1}{2}\d_\nu g_{\al\be})+\d_\nu(\d_\al g_{\be\mu}-\frac{1}{2}\d_\mu g_{\al\be})]\\
        =& g^{\al\be}[\d_\al(\d_\mu g_{\be\nu}+\d_\nu g_{\be\mu}-\d_\be g_{\mu\nu})-\d_\mu(\d_\nu g_{\al\be}+\d_\al g_{\nu\be}-\d_\be g_{\al\nu})\\
        &+\d^2_{\al\be}g_{\mu\nu}]\\
        =&g^{\al\be}[2\d_\al \Ga_{\mu\nu,\be}-2\d_{\mu}\Ga_{\al\nu,\be}
        +\d^2_{\al\be}g_{\mu\nu}]\\
        =& 2g^{\al\be}(R_{\be\nu\al\mu}-\Ga_{\nu\be,\si}\Ga^\si_{\al\mu}+\Ga_{\al\be,\si}\Ga^\si_{\mu\nu})+g^{\al\be}\d^2_{\al\be}g_{\mu\nu}.
    \end{align*}
    We then obtain
    \begin{align*}
        \nab_\mu V_\nu+\nab_\nu V_\mu=g^{\al\be}\d^2_{\al\be}g_{\mu\nu}+2\Ric_{\mu\nu}+\d_\mu g^{\al\be}\Ga_{\al\be,\nu}+\d_\nu g^{\al\be}\Ga_{\al\be,\mu}-2g^{\al\be}\Ga_{\nu\be,\si}\Ga^\si_{\al\mu}.
    \end{align*}
    Combined with \eqref{g_dt}, this implies \eqref{g-Heat-Eq}.
\end{proof}

\subsection{Derivation of the modified Schr\"{o}dinger system from SMCF}
Here we carry out the last step in our analysis of the equations, and 
obtain the main Schr\"odinger equation which governs the time evolution of $\lambda$.

Our starting point is the equations \eqref{main-eq-abst}, which are rewritten as
\begin{equation*}
    i\d^{B}_t\la_{\al\be}+\nab^A_{\al}\nab^{A}_\be \psi -i\la^{\ga}_{\al}\Im(\psi\bar{\la}_{\ga\be})-i\la^{\ga}_{\al}\nab_{\be} V_{\ga}-i\la_\be^\ga\nab_\al V_\ga-iV^\ga\nab^A_\ga\la_{\al\be}=0,
\end{equation*}
We use the compatibility conditions \eqref{csmc}, \eqref{cpt-AiAj-2} and \eqref{R-la} to write the second term as
\begin{align*}
    \nab^A_{\al}\nab^{A}_\be \psi=& \ \nab^A_\al \nab^A_{\si}\la^\si_\be=[\nab^A_\al,\nab^A_\si]\la^\si_\be+\nab^A_\si\nab^{A,\si}\la_{\al\be}\\
    =&\ R_{\al\si\si\de}\la^\de_\be+R_{\al\si\be\de}\la^{\si\de}+i\Im(\la_{\al\mu}\bar{\la}^\mu_\si)\la^\si_\be+\nab^A_\si\nab^{A,\si}\la_{\al\be}\\
    =&-\Ric_{\al\de}\la^\de_\be+R_{\al\si\be\de}\la^{\si\de}+i\Im(\la_{\al\mu}\bar{\la}^\mu_\si)\la^\si_\be+\nab^A_\si\nab^{A,\si}\la_{\al\be}\\
    =& -\Re(\la_{\al\de}\bar{\psi})\la^\de_\be+\Re(\la_{\si\de}\bar{\la}_{\al\be}-\la_{\si\be}\bar{\la}_{\al\de})\la^{\si\de}  +\la_{\al\mu}\bar{\la}^\mu_\si\la^\si_\be+\nab^A_\si\nab^{A,\si}\la_{\al\be}
\end{align*}
Since 
\begin{align*}
    \frac{1}{2}[-\Re(\la_{\al\de}\bar{\psi})\la^\de_\be-\Re(\la_{\be\de}\bar{\psi})\la^\de_\al -i\la^{\ga}_{\al}\Im(\psi\bar{\la}_{\ga\be}) -i\la^{\ga}_{\be}\Im(\psi\bar{\la}_{\ga\al})]
    =-\psi\Re(\la_{\al\de}\bar{\la}^\de_\be),
\end{align*}
we obtain the $\la$-equations
\begin{equation}     \label{Sch-la}
    \begin{aligned}
i\d^{B}_t\la_{\al\be}
+\nab^A_\si\nab^{A,\si}\la_{\al\be}
=& \ iV^\ga\nab^A_\ga\la_{\al\be}
+i\la^{\ga}_{\al}\nab_{\be} V_{\ga}
+i\la_\be^\ga\nab_\al V_\ga
+\psi\Re(\la_{\al\de}\bar{\la}^\de_\be)\\
&-\Re(\la_{\si\de}\bar{\la}_{\al\be}-\la_{\si\be}\bar{\la}_{\al\de})\la^{\si\de} 
-\la_{\al\mu}\bar{\la}^\mu_\si\la^\si_\be
.
\end{aligned}
\end{equation}

In conclusion, under the heat coordinate condition $g^{\al\be}\Ga^{\ga}_{\al\be}=V^\ga$ and heat gauge condition $\nab^{\al}A_{\al}=B$, by \eqref{Sch-la}, \eqref{g-Heat-Eq} and \eqref{Heat-A-pre}, we obtain the covariant Schr\"{o}dinger equation for the complex second fundamental form tensor $\la$
\begin{equation}        \label{mdf-Shr-sys-2}
\left\{
\begin{aligned}
    &\begin{aligned}
    (i\d_t+\nab_\si\nab^\si)\la_{\al\be}
    =&i(V-2A)^\si\nab_\si \la_{\al\be}-i\nab_\si A^\si \la_{\al\be}+i\la^{\ga}_{\al}\nab_{\be} V_{\ga}
    +i\la_\be^\ga\nab_\al V_\ga\\
    &+(B+A_\si A^\si-V_\si A^\si)\la_{\al\be}
    +\psi\Re(\la_{\al\de}\bar{\la}^\de_\be)
    \\
    &-\Re(\la_{\si\de}\bar{\la}_{\al\be}-\la_{\si\be}\bar{\la}_{\al\de})\la^{\si\de} 
    -\la_{\al\mu}\bar{\la}^\mu_\si\la^\si_\be,
    \end{aligned}\\
    & \la(0,x) = \la_0(x).
    \end{aligned}
\right.    
\end{equation}
These equations are fully covariant, and do not depend on the gauge choices
made earlier. On the other hand, our gauge choices imply that 
the advection field $V$ and the connection coefficient $B$ are determined by the metric $g$ and connection $A$ via \eqref{Heat-coordinate}, respectively, \eqref{heat-gauge}. In turn,
the metric $g$ and the connection coefficients $A$  are determined in an parabolic fashion via the following equations
\begin{equation}           \label{par-syst}
	\left\{\begin{aligned}
		&\begin{aligned}  
         (\d_t-g^{\al\be}\d^2_{\al\be})g_{\mu\nu}=&2\Ric_{\mu\nu}+2\Im(\psi\bar{\la}_{\mu\nu})-2g^{\al\be}\Ga_{\mu\be,\si}\Ga^\si_{\al\nu}\\
         &+\d_{\mu}g^{\al\be}\Ga_{\al\be,\nu}+\d_{\nu}g^{\al\be}\Ga_{\al\be,\mu}.
        \end{aligned}\\
	    &(\d_t -\nab_\si\nab^\si) A_\al=-\nab^\si\Im(\la^\ga_\al\bar{\la}_{\si\ga})-\Ric_{\al\de}A^\de+\Re(\la^\ga_\al\overline{\nab^A_\ga\psi})-\Im(\la^\ga_\al\bar{\la}_{\ga\si})V^\si,\\
	    &V^\ga=g^{\al\be}\Ga_{\al\be}^\ga,\quad B=\nab^\al A_\al,
	\end{aligned}\right.
\end{equation}
with initial data 
\begin{equation}    \label{ini-gA}
    g(0,x)=g_0(x),\quad A(0,x) = A_0(x).
\end{equation}
These are determined at the initial time by choosing harmonic coordinates on $\Sigma_0$,  respectively the Coulomb gauge for $A$.

Fixing the remaining degrees of freedom (i.e. the affine group for the choice of the 
coordinates as well as the time dependence of the $SU(1)$ connection) 
 we can assume that the  following conditions hold at infinity in an averaged sense:
 \begin{equation*}  
g(\infty) = I_d,\quad A(\infty) = 0.     
 \end{equation*}
These are needed to insure  the unique solvability of the above parabolic equations
in a suitable class of functions. For the metric $g$ it will be useful to use the representation 
\begin{equation*}
g = I_d + h    
\end{equation*}
so that $h$ vanishes at infinity.

We have arrived at the main Schr\"odinger-Parabolic system \eqref{mdf-Shr-sys-2}-\eqref{par-syst}, whose solvability is the primary 
objective of the rest of the paper. This system is accompanied by
a family of compatibility conditions as follows:
\begin{enumerate}[label=(\roman*)]
    \item The Gauss equations \eqref{R-la} connecting the curvature $R$ of $g$ and $\lambda$.
    \item The Codazzi equations  \eqref{la-commu} for $\lambda$.
    \item The Ricci equations \eqref{cpt-AiAj-2} for the curvature of $A$.
   \item The compatibility condition \eqref{Cpt-A&B} for the $B$.
\end{enumerate}
We will solve the system irrespective of these compatibility conditions, 
but then show them be satisfied for small solutions to the nonlinear system \eqref{mdf-Shr-sys-2}-\eqref{par-syst}, by propagating them from the initial time $t=0$.

Now we can restate here the small data local well-posedness result for the (SMCF) system in 
Theorem~\ref{LWP-thm} in terms of the above system:

\begin{thm}[Small data local well-posedness in the good gauge]   \label{LWP-MSS-thm}
	Let $d\geq 2$ and $s>\frac{d}{2}$. Then there exists $\ep_0>0$ sufficiently small such that, for all initial data $(\la_0, h_0,A_0)$
	satisfying the constraints \eqref{la-commu}, \eqref{R-la} and \eqref{cpt-AiAj-2} and with 
\begin{equation} \label{small-data-MS}
	\lV \la_0\rV_{H^s}+\lV h_0\rV_{\bY_0^{s+2}} 
	+ \lV  A_0\rV_{H^{s+1}} \leq \ep_0,
\end{equation}
the modified Schr\"odinger system \eqref{mdf-Shr-sys-2}, coupled with the
parabolic system \eqref{par-syst} for $(h,A)$
is locally well-posed in $l^2X^s\times\bEE^s$ on the time interval $I=[0,1]$. Moreover, the second fundamental form $\la$, the metric $g$ and the connection coefficients $A$ satisfy
the bounds
	\begin{equation}\label{psi-full-reg}
	\lV \la\rV_{l^2 X^s} +  \lV (h,A)\rV_{\bEE^s}\lesssim \lV \la_0\rV_{H^s} + \lV h_0\rV_{\bY_0^{s+2}} 
	+ \lV  A_0\rV_{H^{s+1}}.
	\end{equation}
In addition, the functions $(\la, g,A)$ satisfy 
	the constraints \eqref{R-la}, \eqref{la-commu}, \eqref{cpt-AiAj-2} and \eqref{Cpt-A&B}.
\end{thm}

Here the solution $\la$ satisfies in particular the expected bounds
\[
\| \la \|_{C[0,1;H^s]} \lesssim \|\la_0\|_{H^s}.
\]
The spaces $l^2 X^s$ and  $\bEE^s$, defined in the next section, contain a more complete description of the  full set of variables $\la,h,A$, which includes both Sobolev regularity and local energy bounds.

In the above theorem,  by well-posedness we mean a full Hadamard-type well-posedness, including the following properties:
\begin{enumerate}[label=\roman*)]
    \item Existence of solutions $\la \in C[0,1;H^s]$, with the additional regularity 
    properties \eqref{psi-full-reg}.
    \item Uniqueness in the same class.
    \item Continuous dependence of solutions with respect to the initial data 
    in the strong $H^s$ topology.
    \item Weak Lipschitz dependence  of solutions with respect to the initial data 
    in the weaker $L^2$ topology.
    \item Energy bounds and propagation of higher regularity.
\end{enumerate}

We conclude this section with several remarks concerning 
the result in Theorem~\ref{LWP-MSS-thm}:

\begin{rem}[The variable $\lambda$ vs $\psi$]
In our earlier paper \cite{HT21} we have worked with $\psi$ as the main dynamic variable for the Schr\"odinger flow, and the full second fundamental form $\lambda$ was obtained from $\psi$ by solving an elliptic div-curl system derived from the Codazzi relations \eqref{la-commu}. Here we work directly with $\lambda$, because solving this elliptic system 
has issues at the $L^2$ level in two\footnote{ However, in three and higher dimensions one could still work with $\psi$ if desired.} space dimensions. The downside is that the components of $\lambda$ are not independent, and are instead connected via the compatibility relations \eqref{la-commu}. Thus, these relations will have to be propagated dynamically. 
\end{rem}

\begin{rem}[Initial data sets]
The harmonic/Coulomb gauge condition at the initial time plays no role 
in Theorem~\ref{LWP-MSS-thm}, where smallness is assumed for both $\lambda_0$, $h_0$ and $A_0$. However, it is useful in order to connect 
Theorem~\ref{LWP-MSS-thm} with the earlier statement in Theorems~\ref{LWP-thm}, \ref{LWP-thmd=2}. 
\end{rem}

\bigskip

\section{Function spaces and notations}   \label{Sec3}
The goal of this section is to define the function spaces where we aim to solve 
the (SMCF) system  in the good gauge, given by \eqref{mdf-Shr-sys-2}.
Both the spaces and the notation presented in this section are similar to those introduced in \cite{MMT3,MMT4,MMT5}. 

We begin with some constants. 
Let regularity index $s>d/2$ and $\de>0$ be a small\footnote{ Ideally here one would like to set $\delta = 0$, but this is only possible in dimensions three and higher.} constant satisfying
\begin{equation*}
    0<\de\ll s-s_d.
\end{equation*}
We then define the constant $\si_d$ depending on dimensions $d$ as 
\begin{equation}       \label{si_d}
    \si_d=d/2-\de.
\end{equation}

For a function $u(t,x)$ or $u(x)$, let $\hat{u}=\FF u$ and $\check{u}=\FF^{-1}u$ denote the Fourier transform and inverse Fourier transform in the spatial variable $x$, respectively. Fix a smooth radial function $\varphi:\R^d \rightarrow [0,1] $ supported in $[-2,2]$ and equal to 1 in $[-1,1]$, and for any $i\in \Z$, let
\begin{equation*}
	\varphi_i(x):=\varphi(x/2^i)-\varphi(x/2^{i-1}).
\end{equation*}
We then have the spatial Littlewood-Paley decomposition,
\begin{equation*}
	\sum_{i=-\infty}^{\infty}P_i (D)=1, \quad \sum_{i=0}^{\infty}S_i (D)=1,
\end{equation*}
where $P_i$ localizes to frequency $2^i$ for $i\in \Z$, i.e,
\begin{equation*}
	 \FF(P_i u)=\varphi_i(\xi)\hat{u}(\xi),
\end{equation*}
and 
\[S_0(D)=\sum_{i\leq 0}P_i(D),\quad S_i(D)=P_i(D),\ {\rm for}\ i>0.\]
For simplicity of notation, we set
\[
u_j=S_j u,\quad u_{\leq j}=\sum_{i=0}^j S_i u,\quad u_{\geq j}=\sum_{i=j}^{\infty} S_i u.
\]

For each $j\in\N$, let $\QQ_j$ denote a partition of $\R^d$ into cubes of side length $2^j$, and let $\{\chi_Q\}$ denote an associated partition of unity. For a translation-invariant Sobolev-type space $U$, set $l^p_j U$ to be the Banach space with associated norm 
\begin{equation*}
	\lV u\rV_{l^p_j U}^p=\sum_{Q\in\QQ_j}\lV \chi_Q u\rV_U^p
\end{equation*}
with the obvious modification for $p=\infty$.

Next we define the $l^2X^s$ and $l^2N^s$ spaces, which will be used for the primary variable 
$\la$, respectively for the source term in the Schr\"odinger equation for $\la$.
Following \cite{MMT3,MMT4,MMT5}, we first define the $X$-norm as
\begin{equation*}
	\lV u\rV_{X}=\sup_{l \in \N} \sup_{Q\in\QQ_l} 2^{-\frac{l}{2}}\lV u\rV_{L^2L^2([0,1]\times Q)}.
\end{equation*}
Here and throughout, $L^pL^q$ represents $L^p_tL^q_x$. To measure the source term, we use an atomic space $N$ satisfying $X=N^{\ast}$. A function $a$ is an atom in $N$ if there is a $j\geq 0$ and a $Q\in \QQ_j$ such that $a$ is supported in $[0,1]\times Q$ and 
\begin{equation*}
	\lV a\rV_{L^2([0,1]\times Q)}\lesssim 2^{-\frac{j}{2}}.
\end{equation*}
Then we define $N$ as linear combinations of the form 
\begin{equation*}
	f=\sum_k c_k a_k,\ \ \sum_k|c_k|<\infty,\ \ a_k\ {\rm atom},
\end{equation*}
with norm 
\begin{equation*}
	\lV f\rV_N=\inf\big\{\sum_k |c_k|: f=\sum_k c_k a_k,\ a_k\ {\rm atoms}\big\}.
\end{equation*}

For solutions which are localized to frequency $2^j$ with $j \geq 0$, we will work in the space
\begin{equation*}
	X_j=2^{-\frac{j}{2}}X\cap L^{\infty}L^2,
\end{equation*} 
with norm 
\begin{equation*}
	\lV u\rV_{X_j}=2^{\frac{j}{2}}\lV u\rV_X+\lV u\rV_{L^{\infty}L^2}.
\end{equation*}
One way to assemble the $X_j$ norms is via the $X^s$ space
\begin{equation*}
\lV u\rV_{X^s}^2=\sum_{j\geq 0} 2^{2js}\lV S_j u\rV_{X_j}^2.
\end{equation*}
But we will also add the $l^p$ spatial summation on the $2^j$ scale to $X_j$,
in order to obtain the space $l^p_j X_j$ with norm
\[
\lV u\rV_{l^p_j X_j} =(\sum_{Q\in \QQ_j} \lV \chi_Q u\rV_{X_j}^p)^{1/p}.
\] 
We then define the space $l^p X^s$ by 
\begin{equation*}
	\lV u\rV_{l^p X^s}^2=\sum_{j\geq 0}2^{2js}\lV S_j u\rV_{l^p_j X_j}^2.
\end{equation*}
For the solutions of Schr\"{o}dinger equation in \eqref{mdf-Shr-sys-2}, we will be working primarily in $l^2 X^s$.

We analogously define 
\begin{equation*}
	N_j=2^{\frac{j}{2}}N+L^1L^2,
\end{equation*}
which has norm 
\begin{equation*}
	\lV f\rV_{N_j}=\inf_{f=2^{\frac{j}{2}}f_1+f_2} \big(\lV f_1\rV_N+\lV f_2\rV_{L^1L^2}\big),
\end{equation*}
and 
\begin{equation*}
	\lV f\rV_{l^pN^s}^2=\sum_{j\geq 0}2^{2js}\lV S_j f\rV_{l^p_j N_j}^2.
\end{equation*}
Here we shall be working primarily with $l^2N^s$.

We also note that for any $j$, we have
\begin{equation*}
	\sup_{Q\in\QQ_j} 2^{-\frac{j}{2}}\lV u\rV_{L^2L^2([0,1]\times Q)}\leq \lV u\rV_{X},
\end{equation*}
hence
\begin{equation*}
	\lV u\rV_N\lesssim 2^{j/2}\lV u\rV_{l^1_jL^2L^2}.
\end{equation*}
This bound will come in handy at several places later on.

For the parabolic system \eqref{par-syst}, it is natural to work in spaces of the form
$L^\infty H^s$.  However, in order to obtain frequency envelope bounds it is more convenient to slightly 
strengthen this norm. Precisely, we  define the $Z^{\si,s}$ norm as
\begin{equation*}
\| h\|_{Z^{\si,s}}^2
=\||D|^\si S_0 h\|_{L^\infty L^2}^2+\sum_{j\geq 1} 2^{2sj} \| S_j h\|_{L^\infty L^2}^2.
\end{equation*}
Compared to $L^\infty H^s$, here we just commute the $L^\infty_t$ and $l^2$ frequency summation. For simplicity of notation, we denote $Z^s:=Z^{0,s}$. In particular we have  
\[
Z^s \subset L^\infty H^s.
\]

With these notations, we will seek  the solution $(h,A)$ to the parabolic system \eqref{par-syst} in the space $\EE^s$ defined by 
\begin{equation*}
\lV (h,A)\rV_{\mathcal E^s}=\lV h\rV_{Z^{\si_d,s+2}}+\lV  A\rV_{Z^{s+1}}.
\end{equation*}

Correspondingly, at fixed time we define the space $\HH^s$ as
\begin{equation*}
\lV (h,A)\rV_{\HH^s}=\lV |D|^{\si_d}h\rV_{H^{s+2-\si_d}}+\lV  A\rV_{H^{s+1}}.
\end{equation*}

In addition to the above standard norms, for the study of the Schr\"odinger equation for $\la$ we will also need to control a stronger norm $\bY^{s+2}$ for the metric $h=g-I_d$; this will be defined in what follows.

First, similarly to the $l^p_j X_j$ norms above, we also add the $l^p$ spatial summation on the $2^j$ scale to $Z_j$, in order to obtain the space $l^p Z^{\si,s}$ with norm
\begin{equation*}
\| h\|_{l^pZ^{\si,s}}^2=\sum_{j\in \mathbb Z} 2^{2\si j^-+2sj^+} \| P_j h\|_{l^p_{|j|}Z_j}^2=\sum_{j\in \mathbb Z} 2^{2\si j^-+2sj^+} \| P_j h\|_{l^p_{|j|}L^\infty L^2}^2.
\end{equation*}
Here we need to decompose the low frequency part, this allows us to obtain a estimate of $h$ in $l^2 Z^{\si_d,s+2}$ in Proposition \ref{cor-4.3}.
Correspondingly, we will strengthen the $Z^{\si_d,s+2}$ norm of $h$ to $l^2 Z^{\si_d,s+2}$.

More importantly, we will also  introduce some additional structure which is associated to spatial scales larger than the frequency.
Precisely, to measure the portion of $h$ which is localized to frequency $2^j$, this time with $j\in\Z$, we decompose $P_j h$ as an atomic summation of components $h_{j,l}$ associated to spatial scales $2^l$
with $l\geq |j|$, i.e.
\begin{equation*}
P_jh=\sum_{l\geq |j|}h_{j,l}.
\end{equation*}  
Then we define the $Y_j$-norm by
\begin{equation*}
\lV P_j h \rV_{Y_j} =\inf_{P_jh=\sum_{l\geq |j|}h_{j,l}} \sum_{l\geq |j|} 2^{l-|j|} \lV h_{j,l}\rV_{l^1_lL^{\infty}L^2}.
\end{equation*}
In the decomposition of $P_j h$ we may project and assume that all terms are
also localized at frequency $2^j$. However in the definition of the $Y_j$ 
norms we make no such assumption.

Assembling together the dyadic pieces in an $l^2$ Besov fashion, we obtain  the $Y^{s}$ 
space with norm given by
\begin{equation*}
\lV h\rV_{Y^{s}}^2=\sum_{j\in \Z} 2^{2(\frac{d}{2}-\de) j^-+2sj^+}\lV P_jh\rV_{Y_j}^2.
\end{equation*}
Then for $h$-equation in \eqref{mdf-Shr-sys-2}, we will be working primarily in 
$\bY^{s+2}$,
whose norm is defined by
\begin{align*} 
\lV h\rV_{\bY^{s+2}}=\lV  h\rV_{l^2Z^{\si_d,s+2}}+\lV h\rV_{Y^{s+2}}.
\end{align*}

Collecting all the components defined above, for the parabolic system \eqref{par-syst} we define the final $\bEE^s$ norm as
\begin{equation*}
    \lV (h,A)\rV_{\bEE^s}=\lV h\rV_{\bY^{s+2}} +\lV  A\rV_{Z^{s+1}}.
\end{equation*}

At fixed time, we can remove the $L^\infty_t$ in $\bY^{s+2}$ and $\bEE^s$, and obtain the function spaces $\bY_0^{s+2}$ and $\bEE^s_0$ respectively. Precisely, we define the $Y_{0j}$ norm corresponding to $Y_j$ as
\begin{equation*}
\lV P_j h \rV_{Y_{0j}} =\inf_{P_jh=\sum_{l\geq |j|}h_{j,l}} \sum_{l\geq |j|} 2^{l-|j|} \lV h_{j,l}\rV_{l^1_lL^2}.
\end{equation*}
and obtain the $Y^s_0$ space with norm given by
\begin{equation*}
\lV h\rV_{Y_0^{s}}^2=\sum_{j\in \Z} 2^{2(\frac{d}{2}-\de) j^-+2sj^+}\lV P_jh\rV_{Y_{0j}}^2.
\end{equation*}
Then we obtain the space $\bY^{s+2}_0$ with norm defined by
\begin{equation*}
    \|h\|_{\bY^{s+2}_0}=\||D|^{\si_d}h\|_{H^{s+2-\si_d}}+\| h\|_{Y^{s+2}_0},
\end{equation*}
and the space $\bEE^s_0$ defined by
\begin{equation*}
\lV (h,A)\rV_{\bEE^s_0}=\lV h\rV_{\bY^{s+2}_0} +\lV  A\rV_{H^{s+1}}.
\end{equation*}

Finally, to capture only the low frequency information in the $Y_0^s$ spaces, 
we introduce the $Y_0^{lo}$ norm, which is used in our main two dimensional result in Theorem~\ref{LWP-thmd=2}:
\begin{equation*}
\lV h\rV_{Y_0^{lo}}^2=\| P_{\geq 0} h\|_{Y_{00} \cap L^\infty}+ \sum_{j < 0} 2^{2(\frac{d}{2}-\de) j^-}\lV P_jh\rV_{Y_{0j}}^2.    
\end{equation*}

\medskip

Next, we define the frequency envelopes as in \cite{MMT3,MMT4,MMT5} which will be used in multilinear estimates. Consider a Sobolev-type space $U$ for which we have
\begin{equation*}
\lV u\rV_U^2=\sum_{k=0}^{\infty} \lV S_k u\rV_U^2.
\end{equation*}
A frequency envelope for a function $u\in U$ is a positive $l^2$-sequence, $\{ a_j\}$, with 
\begin{equation*}
\lV S_j u\rV_U\leq a_j.
\end{equation*}
We shall only permit slowly varying frequency envelopes. Thus, we require $a_0\approx \lV u\rV_U$ and 
\begin{equation}      \label{FreEve-relation}
a_j\leq 2^{\delta|j-k|} a_k,\quad j,k\geq 0,\ 0<\delta\ll s-d/2.
\end{equation}
The constant $\delta$ shall be chosen later and only depends on $s$ and the dimension $d$. Such frequency envelopes always exist. For example, one may choose
\begin{equation}         \label{Freq-envelope}
a_j=2^{-\de j}\lV u\rV_U+\max_k 2^{-\de|j-k|} \lV S_k u\rV_U.
\end{equation}

Since we often use Littlewood-Paley decompositions, the next lemma is a convenient tool 
to see that our function spaces are invariant under the action of some standard classes of multipliers:

\begin{lemma}
	For any Schwartz function $f\in\SS$, multiplier $m(D)$ with $\lV \FF^{-1}(m(\xi))\rV_{L^1}<\infty$, and translation-invariant Sobolev-type space $U$, we have
	\begin{equation*}
		\lV m(D)f\rV_{U}\lesssim \lV \FF^{-1}(m(\xi))\rV_{L^1}\lV f\rV_U.
	\end{equation*}
\end{lemma}

Finally, we state a Bernstein-type inequality and two estimates.

\begin{lemma}[Bernstein-type inequality, Lemma 3.2 \cite{HT21}]
	For any $j,k\in\Z$ with $j+k\geq 0$, $1\leq r<\infty$ and $1\leq q\leq p\leq \infty$, we have
	\begin{align*}   
		&\lV P_k f\rV_{l^r_jL^p}\lesssim 2^{kd(\frac{1}{q}-\frac{1}{p})}\lV P_k f\rV_{l^r_jL^q}.
	\end{align*}
\end{lemma}

\begin{prop}[Algebra property]
    For any $f,\ g\in Y^{lo}_0$ we have 
    \begin{equation}\label{algebra-Ylo}
        \| fg\|_{Y^{lo}_0}\lesssim \| f\|_{Y^{lo}_0}\| g\|_{Y^{lo}_0}.
    \end{equation}
\end{prop}
\begin{proof}
We first note that by Bernstein's inequality we have 
$Y^{lo}_0 \subset L^\infty$. Then for the high-low and low-high interactions we 
can estimate
    \begin{align*}
        \|P_j (P_j f P_{<j} g)\|_{Y_{0j}}+\|P_j (P_{<j} f P_{j} g)\|_{Y_{0j}}\lesssim \| P_j f\|_{Y_{0j}} \|P_{<j}g\|_{L^\infty}+\| P_{<j} f\|_{L^\infty} \|P_{j}g\|_{Y_{0j}}.
    \end{align*}
    For the high-high interactions, we have
    \begin{align*}
        &2^{(\frac{d}{2}-\de)j^-}\|P_j (\sum_{0>l>j}P_l f P_l g+P_{\geq 0}f P_{\geq 0}g)\|_{Y_{0j}}\\
        \lesssim&  2^{(d-\de)j^-}(\sum_{0>l>j}\|P_j (P_l f P_l g)\|_{l^1_{|j|}L^1}+\|P_{\geq 0}f P_{\geq 0}g\|_{l^1_{|j|}L^1})\\
        \lesssim & \sum_{0>l>j} 2^{(d-\de)(j^--l)} 2^{(d-\de)l}\|P_l f\|_{L^2} \|P_l g\|_{L^2}+2^{(d-\de)j^-}\|P_{\geq 0}f\|_{L^2}\|P_{\geq 0}g\|_{L^2}\\
        \lesssim & \sum_{0>l>j} 2^{(d-\de)(j^--l)} 2^{(d-\de)l}\|P_l f\|_{Y_{0j}} \|P_l g\|_{Y_{0j}}+2^{(d-\de)j^-}\|P_{\geq 0}f\|_{Y_{00}}\|P_{\geq 0}g\|_{Y_{00}}.
    \end{align*}
    These two bounds imply that
    \begin{equation*}
        \|P_{<0}(fg)\|_{Y^{lo}_0}\lesssim \|f\|_{Y^{lo}_0}\|g\|_{Y^{lo}_0}.
    \end{equation*}
    
    For the high-frequency part $P_{\geq 0}(fg)$, we bound its $L^\infty$ norm by
    \begin{equation*}
        \|P_{\geq 0}(fg)\|_{L^\infty} \lesssim \|f\|_{L^\infty}\|g\|_{L^\infty}\lesssim \|f\|_{Y^{lo}_0}\|g\|_{Y^{lo}_0}.
    \end{equation*}
    To bound its $Y_{00}$ norm, we further decompose it as 
    \begin{equation*}
        P_{\geq 0}(fg)=P_{\geq 0}(P_{\geq 0}f \cdot g)+P_{\geq 0}(P_{<0}f\cdot P_{\geq 0}g)+P_{\geq 0}(P_{<0}f\cdot P_{[-3,-1]}g).
    \end{equation*}
    The first term is bounded by
    \begin{equation*}
        \|P_{\geq 0}(P_{\geq 0}f \cdot g)\|_{Y_{00}} \lesssim \|P_{\geq 0}f\|_{Y_{00}} \|g\|_{L^\infty}\lesssim \|f\|_{Y^{lo}_0}\|g\|_{Y^{lo}_0}.
    \end{equation*}
    The second term is bounded similarly. We bound the last term by
    \begin{equation*}
        \|P_{\geq 0}(P_{< 0}f \cdot P_{[-3,-1]} g)\|_{Y_{00}} \lesssim \|P_{< 0}f \|_{L^\infty}\|P_{[-3,-1]}g\|_{Y_{00}} \lesssim  \|f\|_{Y^{lo}_0}\|g\|_{Y^{lo}_0}.
    \end{equation*}
    This completes the bound for high frequency part, and thus the proof of the proposition.
\end{proof}

\begin{lemma}
    For any Schwartz function $f$, $j\in \N$ and $1\leq r\leq \infty$, we have
    \begin{gather}\label{lin-heat}
    \| e^{t\De}f\|_{l^r_jL^\infty_tL^2}\lesssim \| f\|_{l^r_j L^2},\\    \label{NLin-heat}
    \|\int_0^t e^{(t-s)\De} S_j f\ ds\|_{l^r_j L^\infty L^2}\lesssim 2^{-2j^+}\| S_j f\|_{l^r_jL^\infty L^2}.
    \end{gather}
\end{lemma}
\begin{proof}
We use the heat kernel 
\[
K(t,x) = (4\pi t)^{-\frac{d}2} e^{-\frac{x^2}{4t}} 
\]
which we decompose with respect to cubes $Q \in \QQ_j$. Then from the 
corresponding decomposition
\[
e^{t\Delta} f = \sum_{Q \in \QQ_j} (\chi_Q(x) K(t,x)) \ast_x f 
\]
we obtain 
\[
\|  e^{t\De}f\|_{l^r_jL^\infty_tL^2} \lesssim \|K\|_{l^1_j L^\infty_t L^1_x} 
\| f\|_{L^2}
\]
Since $t \in [0,1]$ and $r \geq 0$, we can use the exponential 
off-diagonal decay for $K$ on the unit scale to conclude that 
\[
\|K\|_{l^1_j L^\infty_t L^1_x} \lesssim 1,
\]
and thus \eqref{lin-heat} follows.

For the  second bound, we separate the low frequencies and use the kernel $K_0$
for $S_0e^{(t-s)\De}$ with a similar cube decomposition to estimate
\[
\|\int_0^t e^{(t-s)\De} S_0 f\ ds\|_{l^r_0 L^\infty L^2}
\lesssim \| K_0 \|_{l^1_j L^1_{t,x}} \| f\|_{l^r_0 L^\infty L^2}
\]
where the $K_0$ norm is easily estimates using the rapid kernel decay on the unit scale.

Similarly, for high frequencies $j > 0$ we use the kernel $K_j$
for $S_j e^{(t-s)\De}$ with a similar cube decomposition to estimate
\[
\|\int_0^t e^{(t-s)\De} S_j f\ ds\|_{l^r_0 L^\infty L^2}
\lesssim \| K_j \|_{l^1_j L^1_{t,x}} \| f\|_{l^r_0 L^\infty L^2}
\]
For fixed $t$ we use the exponential symbol decay to obtain
\[
\|  K_j \|_{l^1_j L^1_{x}} \lesssim  e^{-c 2^{2j} t},
\]
and now the time integration yields the desired $2^{-2j}$ decay.
This concludes the proof of \eqref{NLin-heat}.
\end{proof}

\bigskip

\section{The initial data}\label{s:elliptic}

Our evolution begins at time $t = 0$, where we need to make a good gauge choice for the initial submanifold $\Sigma_0$. This has two components,

(i) a good set of coordinates on $\Sigma_0$, namely the global harmonic coordinates, represented via the map $F : \R^d \rightarrow \R^{d+2}$.

(ii) a good orthonormal frame in $N \Sigma_0$, where we will use the Coulomb gauge.

Once this is done, we have the frame  in the tangent space and the frame $m$ in the normal bundle. In turn, as described in Section \ref{Sec-gauge}, these generate the metric $g$, the second fundamental form $\la$ with trace $\psi$ and the connection $A$, all at the initial time $t = 0$.

We will first carry out the construction of the 
global harmonic coordinates, and use them to prove 
bounds for the parametrization $F$ and for the metric $g_0 = I_d + h_0$. Then we introduce the Coulomb gauge, which in turn determines $\lambda_0$ and $A_0$. 

The final objective of this section will be  to describe the regularity and size of $(\lambda_0,g_0,A_0)$, and thus justify the smallness condition \eqref{small-data-MS}
for the Schr\"odinger-Parabolic system\eqref{mdf-Shr-sys-2}-\eqref{par-syst}. 
The main result of this section is stated below in  Proposition~\ref{Global-harmonic} for dimensions $d\geq 3$ and \propref{Global-harmonic-d2} for dimension $2$, respectively.

In order to state the following propositions, we define some notations. Let $F:\R^d_x \rightarrow(\R^{d+2},g_{\R^{d+2}})$ be an immersion with induced metric $g(x)$. For any change of coordinate $y=x+\phi(x)$, we denote 
\begin{equation*}
    \tilde{F}(y)=F(x(y)),
\end{equation*}
and its induced metric by $\tilde g_{\al\be}(y)=\<\d_{y_\al}\tF,\d_{y_\be}\tF\>$. We also denote its Christoffel symbol as $\tilde \Ga$ and $\tilde h(y)=\tilde g(y)-I_d$. The main results are summarized as follows:

\begin{prop}[Harmonic coordinates and initial data in dimensions $d\geq 3$]\label{Global-harmonic}
Let $d\geq 3$, $s>\frac{d}{2}$. Let
$
F:(\R^d_x,g)\rightarrow (\R^{d+2},g_{\R^{d+2}})
$
be an immersion with induced metric $g=I_d+h$. 
Assume that the metric $h$ and the mean curvature $\mathbf H$ satisfy the smallness conditions
\begin{equation}       \label{h0}
\| |D|^{\si_d} h\|_{H^{s+1-{\si_d}}}\leq \ep_0,\quad  \|\mathbf H\|_{H^s}\leq \ep_0.
\end{equation}
Then there exists a unique change of coordinates 
$y=x+\phi(x)$ with $\lim_{x\rightarrow\infty}\phi(x)=0$ and $\nab\phi$ uniformly small,
such that the new coordinates $\{y_1,\cdots,y_d\}$ are global harmonic coordinates.  Moreover, we have the bound 
\begin{align}        \label{bound-phi}
\| |D|^{\si_d} \nabla \phi\|_{H^{s+1-{\si_d}}} \lesssim \| |D|^{\si_d} h\|_{H^{s+1-{\si_d}}}
\end{align}
and, in the new coordinates $\{y_1,\cdots,y_d\}$,
for the metric and mean curvature we have
\begin{equation}   \label{bound-tildeg}
    \| |D_y|^{\si_d}\tilde h\|_{H^{s+1-\sigma_d}(dy)}+\| \mathbf H\|_{H^{s}(dy)}\lesssim \ep_0.
\end{equation}
In addition, under the  harmonic coordinate condition \eqref{hm-coord} for $g$, respectively the Coulomb gauge \eqref{Coulomb} for $A$, we have the following bounds for complex second fundamental form $\la$, metric $h = g - I_d$ and $A$:
\begin{equation}   \label{ini}
\|\la\|_{H^s}+\| h\|_{\bY_0^{s+2}} +\|  A \|_{H^{s+1}} \lesssim \epsilon_0.    
\end{equation}
\end{prop}

\bigskip

Compared to the above higher dimensions cases, in dimensions $2$ we would work in a smaller function space.

\begin{prop}[Harmonic coordinates and initial data in dimension 2]\label{Global-harmonic-d2}
Let $d= 2$, $s>\frac{d}{2}$, and $\si_d$ be as in \eqref{si_d}. Let
$
F:(\R^d_x,g)\rightarrow (\R^{d+2},g_{\R^{d+2}})
$
be an immersion with induced metric $g=I_d+h$. 
Assume that the metric $h$ and mean curvature $\mathbf H$ 
satisfy the smallness conditions
\begin{equation}      \label{h0-d2}
\| |D|^{\si_d} h\|_{H^{s+1-{\si_d}}}\leq \ep_0,\quad \|h\|_{Y^{lo}_0}\leq \ep_0, \quad  \|\mathbf H\|_{H^s}\leq \ep_0.  
\end{equation}
Then there exists a change of coordinates 
$y=x+\phi(x)$, with  $\nab\phi$ uniformly small
and with $\lim_{x\rightarrow\infty}\nabla \phi(x)=0$, unique modulo constants,
such that the new coordinates $\{y_1,\cdots,y_d\}$ are global harmonic coordinates.  Moreover, we have the bound
\begin{align}        \label{bound-phi-d2}
    \||D|^{\si_d}\nabla \phi\|_{H^{s+1-\si_d}}\lesssim \||D|^{\si_d} h\|_{H^{s+1-\si_d}},
\end{align}
and, in the new coordinates $\{y_1,\cdots,y_d\}$,
for the metric and mean curvature we have
\begin{equation}   \label{bound-tildeg-d2}
    \| |D_y|^{\si_d}\tilde h\|_{H^{s+1-\si_d}(dy)}+\|\tilde h\|_{Y^{lo}_0}+\| \mathbf H\|_{H^{s}(dy)}\lesssim \ep_0\,.
\end{equation}
In addition, under the  harmonic coordinate condition \eqref{hm-coord} for $g$, respectively the Coulomb gauge \eqref{Coulomb} for $A$, we have the following bounds for complex second fundamental form $\la$, metric $h = g - I_d$ and $A$:
\begin{equation}   \label{ini-d2}
\|\la\|_{H^s}+\| h\|_{\bY_0^{s+2}} +\| A \|_{H^{s+1}} \lesssim \epsilon_0.    
\end{equation}
\end{prop}

We remark that the bounds \eqref{ini} respectively 
\eqref{ini-d2} are the only way the harmonic/Coulomb gauge condition  at $t = 0$ enters this paper. Later, in the study of the parabolic system 
\eqref{par-syst}, we simply assume that the initial data $(\lambda_0,h_0,A_0)$
satisfies the above smallness condition.

Of the three components of the initial data, $\la_0$ 
may be thought of as the fundamental one.
 Indeed, the initial data $(g_0,A_0)$ for the heat flow \eqref{par-syst} is determined by $\lambda_0$ via the harmonic coordinate condition \eqref{hm-coord} for $g$, respectively the Coulomb gauge \eqref{Coulomb} for $A$, which yield the elliptic equations in Lemmas~\ref{Ell-g} and \ref{Ell-A}. This was the point of view adopted in our previous paper \cite{HT21} in high dimension, and it largely applies here as well. 
  The only exception to this is in two space dimensions, where we a-priori make an additional low frequency assumption on the metric $g$, namely the $Y_0^{lo}$ bound, which cannot be recovered from the $\lambda_0$ bounds.

\subsection{Global Harmonic coordinates} 
    Here we make a change of coordinates to gain the harmonic coordinates, and then prove that in the new coordinates, the metric $h$ and mean curvature $\mathbf H$ are also small.
  \medskip
  
    \emph{Step 1: Solve the $\phi$ equation and prove the bounds \eqref{bound-phi} and \eqref{bound-phi-d2}.} To obtain harmonic coordinates, we start with the bound for metric 
    \begin{equation}   \label{bound-metric-re}
        \||D|^{\si_d}h\|_{H^{s+1-\si_d}}\lesssim \ep_0.
    \end{equation}
    We make a change of coordinates $x+\phi(x)=y$ with $\nab\phi$ small such that the new coordinates are harmonic. Since the operator $\De_g$ does not depend on the coordinates, by \eqref{hm-coord} we have
    \begin{equation*}
        \De_g (x+\phi(x))=0,
    \end{equation*}
    which implies
    \begin{equation}  \label{varphi-eq}
        \De_g \phi_\ga= g^{\al\be}\Ga_{\al\be}^\ga. 
    \end{equation}
    which we write schematically in the form
    \begin{equation*}
        \De \phi=h\nab^2 \phi+g\nab h \nab \phi+g\nab h.
    \end{equation*}
    Since the leading order term in the right hand side is $\nab h$, by the assumption on the metric $|D|^{\si_d}h \in H^{s+1-\si_d}$ we will work in the space 
    \begin{equation*}
        \{ \phi: \||D|^{1+\si_d}\phi\|_{H^{s+1-\si_d}}<\infty   \}
    \end{equation*}
    Then by Sobolev embeddings and the smallness of $h$ we can uniquely solve 
    the equation \eqref{varphi-eq} in this space using 
    the  contraction principle, obtaining a solution $\phi$ which satisfies 
    the  bound
    \begin{equation}    \label{bound-phi-re}
        \||D|^{\si_d} \nabla \phi\|_{H^{s+1-\si_d}}\lesssim \||D|^{\si_d}h\|_{H^{s+1-\si_d}}\lesssim \ep_0.
    \end{equation}
which is exactly \eqref{bound-phi} and \eqref{bound-phi-d2} in Theorem~\ref{Global-harmonic}, respectively Theorem~\ref{Global-harmonic-d2}.

    \medskip  
    
    \emph{Step 2: Prove the bounds \eqref{bound-tildeg} and \eqref{bound-tildeg-d2} for $\tilde h$ and $\mathbf H$ in Sobolev spaces.}
    First we prove that the desired $\tilde h$ bound holds in the $x$-coordinates,  
    \begin{equation}    \label{bound-tildeg-key1}
    \||D_x|^{\si_d}\tih(y(x))\|_{H^{s+1-\si_d}(dx)}\lesssim \||D|^{\si_d} h\|_{H^{s+1-\si_d}(dx)}.
    \end{equation}
    By the above change of coordinate and \eqref{bound-phi} we have $\frac{\d x}{\d y}=I_d+\PP(x)$ where $\PP$ is an algebraic function of $\nabla \phi$. 
    Hence by algebra and Moser estimates we have
    \begin{equation}    \label{estimate-P}
        \||D|^{\si_d} \PP\|_{H^{s+1-\si_d}}\lesssim \||D|^{\si_d} \nabla \phi\|_{H^{s+1-\si_d}}\lesssim \ep_0.
    \end{equation}
    Then the desired bound \eqref{bound-tildeg-key1} follows from the relation 
    \[
    \tilde{g}_{\al\be}(y(x))=g_{\mu\nu}(x)(\de^\mu_\al+\PP^{\mu}_\al)(\de^\nu_\be+\PP^{\nu}_\be),
    \]
    again by using algebra bounds in the same space.

    In order to complete the proof of \eqref{bound-tildeg} and \eqref{bound-tildeg-d2}, we need to be able to transfer the Sobolev norms 
    from the $x$ to the $y$ coordinates. For this we will apply the following lemma:
    \begin{lemma}\label{l:Sobolev-equiv}
    	Let the change of coordinates $x+\phi(x)=y$ be as in Proposition \ref{Global-harmonic}. 
        Define the linear operator $T$ as $T(f)(y)=f(x(y))$	
        for any function $f\in L^2(dx)$. Then we have
    	\begin{align}    \label{Equ-f(y)&F(x)}
    	&\| T(f)(y)\|_{H^\sigma(dy)}\lesssim \|f(x)\|_{H^\sigma(dx)}, \qquad
    	\sigma \in [0,s+1],\\    \label{frac-T}
    	&\|T(f)(y)\|_{\dot H^{\al}(dy)}\lesssim \|f(x)\|_{\dot H^{\al}(dx)},\qquad \al\in [0,\frac{d}2).
    	\end{align}
    \end{lemma} 
    \begin{proof}
        The first bound is obtained from \eqref{estimate-P} and \eqref{bound-phi} using the same argument as in Lemma 8.5 in \cite{HT21},  It remains to prove the second bound \eqref{frac-T}. 
        
        By the smallness of $\phi$ \eqref{bound-phi} we have
        \begin{align*}
            \|T(f)(y)\|_{L^2(dy)}\lesssim &\| f(x) \sqrt{I+\d_x\phi} \|_{L^2(dx)}\\
            \lesssim & (1+\||D|^{1+\si_d}\|_{H^{s-\si_d}})^N\|f(x)\|_{L^2(dx)}\lesssim \|f(x)\|_{L^2(dx)}.
        \end{align*}
        Similarly, by \eqref{bound-phi} and \eqref{estimate-P} we also have
        \begin{align*}
            \|\d_y T(f)(y)\|_{L^2(dy)}\lesssim &\| (1+\PP) \d_x f(x) \sqrt{I+\d_x\phi} \|_{L^2(dx)}
            \lesssim  \|\d_x f(x)\|_{L^2(dx)}.
        \end{align*}
        Then  by interpolation we obtain \eqref{frac-T} for $\alpha \in [0,1]$. This suffices in dimension $d = 2$. In higher dimension, 
        we inductively increase the range of $\alpha$ by differentiating.
        Precisely, for $\alpha > 1$ we have
        \[
        \|T(f)(y)\|_{\dot H^{\al}(dy)} = \|\partial_y T(f)(y)\|_{\dot H^{\al-1}(dy)}
        \]
        Here 
        \[
        \partial_y T(f)(y) = T( (I+\PP) \partial_x f)
        \]
        and, by \eqref{estimate-P},  
        \[
        \| (I+\PP) \partial_x f\|_{\dot H^{\alpha-1}(dx)} \lesssim 
        \| \partial_x f\|_{\dot H^{\alpha-1}(dx)}
        \]
        Hence we have reduced the $\dot H^\alpha$ bound to the $\dot H^{\alpha-1}$ bound.
    \end{proof}
    
    \medskip
    
    Given this lemma, by \eqref{bound-tildeg-key1}, \eqref{Equ-f(y)&F(x)} with $\sigma = s+1$ and \eqref{frac-T} we obtain
    \begin{equation*}
        \||D|^{\si_d} \tih\|_{H^{s+1-\si_d}(dy)}\lesssim 
        \||D|^{\si_d} h\|_{H^{s+1-\si_d}(dx)}
    \end{equation*}
    Hence the $\tilde h$ bounds in  \eqref{bound-tildeg} and \eqref{bound-tildeg-d2} follow. Similarly, the $\mathbf H$ bound is also directly transferred to the 
$y$ coordinates by Lemma~\ref{l:Sobolev-equiv}. 
   
 \medskip

 \emph{Step 3: Prove bounds for $\partial^2 F$ in 
 harmonic coordinates.}
While this bound was not explicitely stated in 
Propositions~\ref{Global-harmonic}, \ref{Global-harmonic-d2}, it will play an important role later in the proof of the bounds \eqref{ini} 
and \eqref{ini-d2}.

\begin{lemma}[] \label{d2F-lem}
Let $d\geq 2$, $s>\frac{d}{2}$, and 
$
F:(\R^d,g)\rightarrow (\R^{d+2},g_{\R^{d+2}})
$
be an immersion with metric $\||D|^{\si_d}h\|_{H^{s+1-\si_d}}\lesssim \ep_0$ and mean curvature $\|\mathbf H\|_{H^s}\lesssim \ep_0$ in some coordinates. Then we have
\begin{equation}     \label{d2F}
    \|\d^2 F\|_{H^s}\lesssim \ep_0.
\end{equation}
\end{lemma}
We note that, as a corollary, it follows that we also have the bound
\begin{equation}\label{th-bd}
\| \nabla \tilde h\|_{H^s} \lesssim \ep_0.
\end{equation}
This bound in effect superseeds the $\tilde h$
bound in \eqref{bound-tildeg}, \eqref{bound-tildeg-d2}, with one exception, namely 
in two dimensions at low frequency.  

Another corollary of this is the corresponding bound 
for the second fundamental form $\mathbf h$, namely
\begin{equation}\label{bfh-bd}
\| \mathbf h\|_{H^s} \lesssim \ep_0.
\end{equation}

\begin{proof}[Proof of \propref{d2F-lem}]
By the smallness of $|D|^{\si_d}(g-I_d)$ and Sobolev embedding, we have
\begin{align*}
    \| g^{\al\be}\Ga_{\al\be}^\ga \d_\ga F\|_{H^s}
    \lesssim & (1+\| |D|^{\si_d} h\|_{H^{s+1-\si_d}})\| |D|^{\si_d} h\|_{H^{s+1-\si_d}}(\|\d_\ga F\|_{L^\infty \cap \dot{H}^s})\\
    \lesssim &\ep_0 (\|g\|_{L^\infty}^{1/2}+\|\d^2 F\|_{H^s})
    \lesssim  \ep_0 (1+\|\d^2 F\|_{H^s}).
\end{align*}
Then we can bound $\d^2 F$ by
\begin{align*}
    \|\d^2 F\|_{H^s}= &\|\mathcal R \De F\|_{H^s}\lesssim \|\De F\|_{H^s}\\
    \lesssim & \| \De_g F\|_{H^s}+\|h^{\al\be}\d^2_{\al\be}F\|_{H^s}+\| g^{\al\be} \Ga_{\al\be}^\ga \d_\ga F \|_{H^s}\\
    \lesssim & \| \mathbf{H}\|_{H^s}+\ep_0(1+ \|\d^2 F\|_{H^s})\\
    \lesssim & \ep_0(1+ \|\d^2 F\|_{H^s}),
\end{align*}
which implies \eqref{d2F}, and thus completes the proof of lemma.
\end{proof}

\medskip

\emph{Step 4: Prove the $Y_{0}^{lo}$ bound for the metric $h$ in \eqref{bound-tildeg-d2} in two dimensions.}
To transfer the $Y_0^{lo}$ bounds to $\tilde h$, our starting point is the estimate
\[
\| h\|_{Y_0^{lo}}  \lesssim \epsilon_0
\]
Next we show that $\nabla \phi$ satisfies a similar bound,
\begin{equation}    \label{phi-Y}
\| \nabla \phi\|_{\tilde Y_0^{lo}}  \lesssim \epsilon_0
\end{equation}

\medskip

\begin{proof}[Proof of \eqref{phi-Y}]
We use the $\phi$-equations \eqref{varphi-eq}, which have the form
\begin{equation*}
    \De \phi=\nab h+h\nab^2 \phi+g\nab h \nab \phi+h\nab h.
    \end{equation*}
To get \eqref{phi-Y} via the contraction principle it suffices 
to estimate the right hand side above in order to prove that
\begin{align*}
    \|\nab \phi\|_{Y^{lo}_0}\lesssim \| h\|_{Y^{lo}_0}+\ep_0(\| h\|_{Y^{lo}_0}+\| \nab \phi\|_{Y^{lo}_0})+\ep_0^2
\end{align*}

First, we bound the $Y_j$ norm of $\nab\phi$. For the $\nab h$, we easily have 
\begin{align*}
    \|P_j\nab^{-1} \nab h\|_{Y_j}\lesssim \|P_j h\|_{Y_j},
\end{align*}
which is acceptable. We will next show how to bound the most umbalanced term $h\nab^2 \phi$; the rest of the terms are estimated similarly. For the high-low interactions: $P_j h \nab^2 P_{<j} \phi$, by \eqref{bound-phi-re} we have
\begin{align*}
    \|P_j \nab^{-1}(P_j h  \nab^2 P_{<j}\phi)\|_{Y_{0j}} \lesssim &  \|P_j h\|_{Y_{0j}} \|\nab P_{<j}\phi\|_{L^\infty}\lesssim \ep_0\|P_j h\|_{Y_{0j}}.
\end{align*}
Similarly, for the low-high interactions: $P_{<j} h \nab^2 P_j\phi$, by \eqref{bound-metric-re} we have
\begin{align*}
    \|P_{j} \nab^{-1}(P_{<j} h  \nab^2 P_j\phi)\|_{Y_{0j}} \lesssim &  \|P_{<j} h\|_{L^\infty} \|\nab P_{j}\phi\|_{Y_{0j}}\lesssim \ep_0 \|\nab P_{j}\phi\|_{Y_{0j}}.
\end{align*}
Finally we consider the high-high interactions, 
$\sum_{l> j} P_j( P_{l} h \nab^2 P_l\phi)$. Here we use Bernstein's inequality to obtain
\begin{align*}
    2^{(\frac{d}{2}-\de)j}\|\sum_{l> j} \nab^{-1}  P_j( P_{l} h \nab^2 P_l\phi)\|_{Y_{0j}}
    \lesssim &\ 2^{(\frac{d}{2}-1-\de)j}\sum_{l>j}\| P_j( P_{l} h \nab^2 P_l\phi)\|_{l^1_{|j|}L^2}\\
    \lesssim & \ 2^{(d-1-\de)j}\sum_{l>j}\|  P_{l} h\|_{L^2}\| \nab^2 P_l\phi\|_{L^2}\\
    \lesssim &\ \sum_{l>j} 2^{(d-1-\de)(j-l)} \| |D|^{\frac{d}{2}-\de} P_l h\|_{L^2}\||D|^\frac{d}{2} P_l \nabla \phi\|_{L^2},
\end{align*}
which in view of the 
bound \eqref{bound-phi-re}
gives
\begin{align*}
    \left(\sum_{j<0}2^{2(\frac{d}{2}-\de)j^-}\|\sum_{l> j} \nab^{-1}  P_j( P_{l} h \nab^2 P_l\phi)\|_{Y_{0j}}^2\right)^{1/2}\lesssim \||D|^{\si_d} h\|_{L^2} \| |D|^{\frac{d}{2}} \nab\phi\|_{L^2}\lesssim \ep_0^2.
\end{align*}

Secondly, we bound the $Y_{00}\cap L^\infty$ norm for the high frequency part $P_{\geq 0} \nab \phi$. The $L^\infty$ bound follows from the $H^s$ bound for $\phi$ and Sobolev embeddings. It remains to estimate its $Y_{00}$ norm.
Since the operator $P_{>0} \nabla^{-1} \nabla$ has the kernel localized to the unit spatial scale, we have
\begin{equation*}
    \|P_{\geq 0}\nab^{-1}\nab h\|_{Y_{00}}\lesssim \|P_{\geq 0}\|_{Y_{00}}.
\end{equation*}
Here we also only discuss the term $h\nab^2\phi$; the contributions 
of the other terms are estimated similarly. We first divide this term as
\begin{equation*}
    P_{\geq 0}\nab^{-1}(P_{\geq 0 }h \nab^2 \phi)+P_{\geq 0}\nab^{-1}(P_{< 0 }h \nab^2 P_{\geq 0} \phi)+P_{\geq 0}\nab^{-1}(P_{< 0 }h \nab^2 P_{[-3,-1]}\phi).
\end{equation*}
For the first term, we directly have
\[ 
\|P_{\geq 0}\nab^{-1}(P_{\geq 0 }h \nab^2 \phi)\|_{Y_{00}}\lesssim \|P_{\geq 0}h\|_{Y_{00}}\|\nab^2 \phi\|_{L^\infty}\lesssim \ep_0\|P_{\geq 0}h\|_{Y_{00}}.  \]
The second term, we further divide it as
\[
P_{\geq 0}\nab^{-1}(P_{< 0 }h \nab^2 P_{\geq 0} \phi)=P_{\geq 0}\mathcal R (P_{< 0 }h \nab P_{\geq 0} \phi)+P_{\geq 0}\nab^{-1}(\nab P_{< 0 }h \nab P_{\geq 0} \phi),
\]
where $\mathcal R$ is Riesz transform.
Then we bound this by 
\[ \|P_{\geq 0}\nab^{-1}(P_{< 0 }h \nab^2 P_{\geq 0} \phi)\|_{Y_{00}}\lesssim \|P_{<0}h\|_{L^\infty}\|P_{\geq 0}\nab\phi\|_{Y_{00}}\lesssim \ep_0 \|P_{\geq 0}\nab\phi\|_{Y_{00}}. \]
Finally, we bound the last term by
\[ \|P_{\geq 0}\nab^{-1}(P_{< 0 }h \nab^2 P_{[-3,-1]} \phi)\|_{Y_{00}}\lesssim \|P_{<0}h\|_{L^\infty}\|P_{[-3,-1]}\nab\phi\|_{Y_{00}}\lesssim \ep_0 \|P_{[-3,-1]}\nab\phi\|_{Y^{lo}_0}. \]
This concludes the proof of the $Y_{00}\cap L^\infty$ norm for $P_{\geq 0}\nab\phi$.
\end{proof}

\medskip

The new metric $\tilde h$ expressed in the $x$ 
coordinates has the cubic polynomial form
\[
\tilde h = P(h,\nabla \phi).
\]
Using the algebra property \eqref{algebra-Ylo} for $Y^{lo}_0 $ and \eqref{phi-Y}, 
we conclude that 
\[
\| \tilde h\|_{\tilde Y^x_0}  \lesssim \epsilon_0
\]
It remains to switch this  bound to the $y$ coordinates,
i.e. show that 
\begin{equation}\label{switch-y}
 \| \tilde h\|_{\tilde Y^x_0} \approx  \| \tilde h\|_{\tilde Y^y_0}  
\end{equation}
where the difficulty is that we need to use a Littlewood-Paley decomposition.
We will circumvent this by using the following 
representation of $Y_0^{lo}$ functions:

\begin{lemma}
A function $f$ is in $Y_0^{lo}$ iff it admits a representation
\[
f = \sum_{j \leq 0} f_j, \qquad f_j \in Y_{0j}
\]
so that the following norm is finite:
\[
\tri(f_j)\tri^2 = \| f_0\|_{Y_{00} \cap L^\infty} + \sum_{j < 0} 2^{2(1-\delta)j} \left(\| f_j\|_{Y_{0j}}^2 + 2^{-2j} \|\nabla  f_j\|_{Y_{0j}}^2+ 2^{-4 j} \|\nabla^2  f_j\|_{Y_{0j}}^2\right)
\]
Further, we have 
\[
\| f\|_{Y_0^{lo}} \approx \inf \{ \tri(f_j)|\tri; f =\sum_{j \leq 0} f_j \}
\]
\end{lemma}    
Since by Sobolev embeddings $\phi$ is small in $C^2$, the  triple norms are easily seen to be equivalent in the $x$ and the $y$ coordinates, therefore
the relation \eqref{switch-y} follows. It remains to prove the Lemma. 

\begin{proof}
In one direction, we directly see that the 
decomposition 
\[
f_0 = P_{\geq 0} f, \qquad
f_j = P_j f, \quad j < 0
\]
yields
\[
 \| f\|_{\tilde Y_0} \approx  \tri(f_j)\tri
\]

Conversely, if $f = \sum f_j$,
then we need to show that
\begin{equation} \label{est-sum-f}
\| f\|_{\tilde Y_0} \lesssim \tri (f_j)\tri
\end{equation}
For this we estimate for $k < 0$
\[
\begin{aligned}
\| P_k f\|_{Y_{0k}} \lesssim & \ 
\| P_k f_0\|_{Y_{0k}} +  \sum_{j< 0}  \| P_k f_j\|_{Y_{0k}}
\\
\lesssim  & \
2^k \| f_0\|_{Y_{00}}+
\sum_{j < 0} 2^{-|j-k|} (\|f_j\|_{Y_j}
+ 2^{-2j} \|\nabla^2 f_j\|_{Y_j})
\end{aligned}
\]
Due to the off-diagonal decay, this implies \eqref{est-sum-f}.
\end{proof}

\medskip

\subsection{The initial data $(\la_0,h_0,A_0)$}
These are determined by the initial manifold $\Sigma_0$ given a gauge choice, which consists 
of choosing (i) a good set of coordinates
on $\Sigma_0$, namely the harmonic coordinates,
and (ii) a good orthonormal frame in $N \Sigma_0$,
where we will use the Coulomb gauge.

In the previous subsection we have discussed the construction of harmonic coordinates and proved the Sobolev bound \eqref{th-bd}
for $h_0$. Here we begin by constructing a Coulomb frame in the normal bundle. Then we can define $\lambda_0$ and $A_0$ and directly prove 
$H^s$ bounds for them.

However, it turns out that the $H^s$ bounds
tell only part of the story for $h_0$ and $A_0$, by treating
them as linear objects. Instead, in our chosen gauge both 
$h_0$ and $A_0$ should be seen as quadratic objects, 
via the equations \eqref{hm-coord}, respectively \eqref{Coulomb}.
In the last part of the section we use these equations 
to improve the bounds for both $h_0$ and $A_0$.

\medskip

\emph{Step 1: The Coulomb frame in $N \Sigma_0$ and the $H^s$ bound for $\lambda$ and $A$.} 
 To obtain the Coulomb gauge, we choose $\tilde{\nu}$ constant uniformly transversal to $T \Sigma_0$; such a $\tilde \nu$ exists because, by Sobolev embeddings, $\partial_x F$
has a small variation in $L^\infty$.
Projecting $\tilde{\nu}$ on the normal bundle $N \Sigma_0$ and normalizing we obtain a normalized section $\tilde{\nu}_1$ of the normal bundle
with the same regularity as $\partial F$. Then we choose $\tilde{\nu}_2$ in $N \Sigma_0$ perpendicular to $\tilde{\nu}_1$. We obtain the orthonormal frame $(\tilde{\nu}_1,\tilde{\nu}_2)$ in $N \Sigma_0$, which again has the same regularity and bounds as $\partial_x F$, namely (see Lemma~\ref{d2F-lem})
\begin{equation}\label{reg-tnu}
\| \partial \tilde \nu_j\|_{H^s} \lesssim \epsilon_0.
\end{equation}
This in particular implies that the associated connection $\tilde A$ also satisfies
\begin{equation}\label{reg-tA}
\|\tilde A\|_{H^s} \lesssim \epsilon_0.
\end{equation}
Then we rotate the frame to get a Coulomb frame $(\nu_1,\nu_2)$, i.e. where the Coulomb gauge condition $\nabla^\al A_\al=0$ is satisfied.
In our complex notation, this corresponds to 
\[
\nu_1 + i \nu_2 = e^{ib} (\tilde \nu_1+ i\tilde \nu_2), \qquad A_j = \tilde A_j -\partial_j b,
\]
where the rotation angle $b$ must solve 
\[
\Delta_g b = \nabla^\alpha \tilde A_\alpha.  
\]
This is an elliptic equation, where the metric $g_0= I_d+h_0$
satisfies \eqref{th-bd}. Using the variational formulation at the $H^1$
level and then perturbative analysis at higher regularity, the solution
is easily seen to satisfy 
\[
\| \partial b\|_{H^s} \lesssim \|A\|_{H^s}
\]
It directly follows that $\nu_1, \nu_2$ and $A$
also satisfy the bounds in \eqref{reg-tnu}, \eqref{reg-tA},
\begin{equation}\label{reg-A}
\| \partial \nu_j\|_{H^s} + \| A\|_{H^s} \lesssim \epsilon_0.    
\end{equation}

Projecting the second fundamental form $\mathbf h$ and the mean curvature $\mathbf H$ on the Coulomb frame as in Section~\ref{complex-mc} we obtain 
the complex second fundamental form $\lambda$ and the
complex mean curvature $\psi$. In view of \eqref{h0}, \eqref{h0-d2} and \eqref{bfh-bd} both of them have
the same regularity,
\begin{equation*}
\| \lambda \|_{H^s} + \|\psi \|_{H^s} \lesssim \epsilon_0. 
\end{equation*}

\bigskip

\emph{Step 2: Prove the bounds in \eqref{ini} and \eqref{ini-d2} for the metric $h$.} For this we rely on the equation \eqref{h-ell}.
The main result is as follows:

\begin{lemma}
     Let $d\geq 2$, $s>\frac{d}{2}$ and $\si_d$ be as in \eqref{si_d}. Assume that $h$ is a solution of \eqref{h-ell} satisfying
     \begin{equation*} 
     \||D|^{\si_d}h\|_{H^{s+1-\si_d}}\leq \ep_0, \qquad \|\la\|_{H^s}\leq \ep_0.
     \end{equation*}
     Then for $d\geq 3$ we have
     \begin{equation}    \label{h-t0}
         \|\partial h\|_{H^{s+1}}+\|h\|_{Y_0^{s+2}}\lesssim \ep_0.
     \end{equation}
    Under the additional assumption 
    \begin{equation}\label{h-start-extra}
    \|h\|_{Y^{lo}_0}\leq \ep_0,
    \end{equation}
    in dimension $d=2$ we have
     \begin{equation}    \label{h-t02}
         \||D|^{\si_d}h\|_{H^{s+2-\si_d}}+\|P_{\geq 0} h\|_{Y_0^{s+2}}\lesssim \ep_0.     \end{equation}
\end{lemma}
Here we remark on the key difference between dimensions 
two and higher. In higher dimensions $d \geq 3$, $h$ may be seen as the unique small solution  for the equation \eqref{h-ell}. But in two dimensions, we merely use \eqref{h-ell} to improve the high frequency bound for $h$. At low frequency this no longer works, and instead 
we use the low frequency bounds on the initial metric $h_0$ as an assumption in our main result. We note that the assumption \eqref{h-start-extra} in the two dimensional case could be avoided,
at the expense of a considerably longer proof.

\begin{proof}
     By \eqref{h-ell}, it suffices to write the equation for $h$ in the shorter form
     \begin{align*}
         \De h=h\nab^2 h+\nab h\nab h+\la^2.
     \end{align*}
     From this and Sobolev embedding, we easily have
     \begin{align*}
         \|P_{\geq 0}\De h\|_{H^s}\lesssim & \ \||D|^{\si_d}h\|_{H^{s+1-\si_d}}\||D|^{\si_d}h\|_{H^{s+2-\si_d}}+\|\d_x h\|_{H^{s}}^2+\|\la\|_{H^s}^2\\
         \lesssim & \ \ep_0 \|P_{\geq 0}|D|^{\si_d}h\|_{H^{s+2-\si_d}}+\ep_0^2.
     \end{align*}
     This implies that 
     \begin{equation}    \label{h-s+2}
         \||D|^{\si_d}h\|_{H^{s+2-\si_d}}\lesssim \ep_0.
     \end{equation}
     In dimension three and higher, a similar argument applies
     in order to improve the low frequency bound. This argument is already in \cite{HT21}, and we do not repeat it here.

     Next, we bound the $Y^{s+2}_0$ norm of $h$. For the low-frequency part, we only need to consider the higher dimensional case $d \geq 3$, as in the case $d = 2$ the low 
     frequency bound is assumed in Theorem~\ref{LWP-thmd=2}. We bound the high-low or low-high interactions by
     \begin{equation*}     
         \begin{aligned}
         \|\De^{-1}P_j (P_{\leq j-3}h \nab^2 h)\|_{Y_j}\lesssim & \ 2^{-2j}\sum_{l\geq |j|}2^{l-|j|}\|P_j (P_{\leq j-3}h \nab^2 h_{j,l})\|_{l^1_{l}L^2}\\
         \lesssim & \ \|h\|_{L^\infty}\|P_j h\|_{Y_j}\\
         \lesssim & \ \||D|^{\si_d} h\|_{H^{s-\si_d}}\|P_j h\|_{Y_j},
     \end{aligned}
     \end{equation*}
     For the high-high interactions $P_j (\nab P_l h \nab P_l h)$, we have
     \begin{align*}
         &\|\sum_{l\geq j}\De^{-1}P_j (\nab P_{l}h \nab P_l h)\|_{Y_j}\\
         \lesssim & \ 2^{-2j}(\sum_{|j|\geq l\geq j}\|P_j (\nab P_{l}h \nab P_l h)\|_{l^1_{|j|}L^2}
         +\sum_{l> |j|}2^{l-|j|}\|P_j (\nab P_{l}h \nab P_l h)\|_{l^1_{l}L^2})\\
         \lesssim &\ \sum_{|j|\geq l\geq j}2^{(d/2-2)j}\|\nab P_{l}h \|_{L^2} \|\nab P_l h\|_{L^2}
         +\sum_{l> |j|}2^{(d/2-2)j}2^{l-|j|}\|\nab P_{l}h\|_{L^2}\| \nab P_l h\|_{L^2},
     \end{align*}
     From these two bounds, for $d\geq 3$ we obtain 
     \begin{align}   \label{h-d3}
         \|\De^{-1}(h\nab^2 h+\nab h\nab h)\|_{Y^{s+2}_0}\lesssim \||D|^{\si_d}h\|_{H^{s+2-\si_d}}\| h\|_{Y^{s+2}_0}+\||D|^{\si_d}h\|_{H^{s+2-\si_d}}^2,
     \end{align}
     and for $d=2$ we obtain 
     \begin{align}    \label{h-d2}
         \|P_{\geq 0}\De^{-1}(h\nab^2 h+\nab h\nab h)\|_{Y^{s+2}_0}\lesssim \||D|^{\si_d}h\|_{H^{s+2-\si_d}}\| h\|_{Y^{s+2}_0}+\||D|^{\si_d}h\|_{H^{s+2-\si_d}}^2.
     \end{align}
    
    We then bound the contribution of the $\la^2$ source term. For the high-low or low-high interactions we have
    \begin{align*}
         \| \De^{-1} P_j(P_{<j}\la P_j \la)\|_{Y_j}
         \lesssim &\ 2^{-2j}\|P_j ( P_{<j}\la P_j \la)\|_{l^1_{|j|}L^2}\\
         \lesssim &\ 2^{dj^-/2-2j} \|P_{<j}\la\|_{H^s}\|P_j \la\|_{L^2},
     \end{align*}
     For the high-high interactions we have
     \begin{align*}
         \|\sum_{l\geq j} \De^{-1} P_j(P_{l}\la P_l \la)\|_{Y_j}
         \lesssim \ \sum_{|j|\geq l\geq j}2^{(d/2-2)j}\|P_{l}\la \|_{L^2}^2
         +\sum_{l> |j|}2^{(d/2-2)j}2^{l-|j|}\|P_{l}\la\|_{L^2}^2.
     \end{align*}
     These two bounds also imply
     for $d\geq 3$
     \begin{align}\label{la-d3}
         \|\De^{-1}(\la^2)\|_{Y^{s+2}_0}\lesssim \|\la\|_{H^{s}}^2,
     \end{align}
     and for $d=2$
     \begin{align} \label{la-d2}
         \|P_{\geq 0}\De^{-1}(\la^2)\|_{Y^{s+2}_0}\lesssim \|\la\|_{H^{s}}^2.
     \end{align}
     Using \eqref{h-s+2}, $\|h\|_{Y^{lo}_0}\leq \ep_0$ and $\|\la\|_{H^s}\leq \ep_0$, by $h$-equation, \eqref{h-d3} and \eqref{la-d3} we obtain for $d\geq 3$
     \begin{equation*}
         \|h\|_{Y^{s+2}_0}\lesssim \ep_0 \|h\|_{Y^{s+2}_0}+\ep_0^2.
     \end{equation*}
     and by \eqref{h-d2} and \eqref{la-d2} we obtain for $d=2$
     \begin{equation*}
         \|P_{\geq 0}h\|_{Y^{s+2}_0}\lesssim \ep_0 \|h\|_{Y^{s+2}_0}+\ep_0^2\lesssim \ep_0 \|P_{\geq 0}h\|_{Y^{s+2}_0}+\ep_0^2.
     \end{equation*}
     This concludes the proof of the  bounds \eqref{h-t0} and  \eqref{h-t02}.
\end{proof}

\medskip

\emph{Step 3: Prove the bound \eqref{ini-d2} for $A$.} This is obtained by \eqref{reg-A} and the following proposition. Here we solve the initial data $A_0$ from the elliptic div-curl system \eqref{Ellp-A-pre}.

\begin{prop}[Initial data $A_0$]     \label{Ini-A}
Let $d\geq 2$, $s>d/2$ and $\de_d=\de$ if $d=2$ and $\de_d=0$ if $d\geq 3$. Assume that 
\begin{equation} \label{A-start}
\|\la\|_{H^s}\leq \ep_0, \qquad \||D|^{\si_d} h\|_{H^{s+1-\si_d}}\leq\ep_0
\end{equation}
Then the elliptic system \eqref{Ellp-A-pre} for $A$  admits a unique small solution with
\begin{equation}\label{Ellip-A0}
    \| |D|^{\de_d} A\|_{H^{s+1-\de_d}}\lesssim  \|\la\|_{H^s}.
\end{equation}
Moreover, assume that $p_{0k}$ is an admissible frequency envelope for $\la\in H^s$. Then we have the frequency envelope bounds
\begin{equation}      \label{A0-envelope}
\|S_k |D|^{\de_d} A\|_{H^{s+1-\de_d}}\lesssim \ep_0  p_{0k}.
\end{equation}
In addition, for the linearization of the solution map above we also have the bound:
\begin{equation}               \label{A0-linearize}
    \| |D|^{\de_d} \dA\|_{H^{\si+1-\de_d}}\lesssim \ep_0(\||D|^{\si_d}\dh\|_{H^{\si+1-\si_d}}+ \|\dla\|_{H^\si}),\qquad \si\in(d/2-2,s].
\end{equation}
\end{prop}

\begin{proof}
    Using the definition of covariant derivatives and the harmonic coordinate condition \eqref{hm-coord} we can rewrite the div-curl
    system  \eqref{Ellp-A-pre} for $A$ as
 \[
 \d^\al A_\al= 0, \qquad
\d_\al A_\be-\d_\be A_\al=\Im(\la_{\al\ga}\bar{\la}^\ga_\be).
\]   
Using these equations we derive a  second order elliptic equation for $A$, namely
\[
\begin{aligned}
\partial_\gamma \partial^\gamma A_\alpha = & \  
\partial_\gamma g^{\gamma \beta} \partial_\alpha A_\beta  
- \partial_\gamma g^{\gamma \beta} (\partial_\alpha A_\beta -\partial_\beta A_\alpha) 
\\ 
= & \ (\partial_\gamma g^{\gamma \beta} \partial_\alpha
- \partial_\alpha g^{\gamma \beta} \partial_\gamma)
A_\beta - \partial_\gamma \Im(\la_{\al\sigma}\bar{\la}^{\ga\sigma})
\end{aligned}
\]
Here we have a leading order cancellation in the first term on the 
right, but we prefer to keep the divergence structure and rewrite
this equation schematically in the form
\[
\Delta A = \partial (\lambda^2) + \partial ( h \partial A).
\]
This will be well suited in order to solve this equation via the contraction principle. 
    
Precisely,  we define the map $A\rightarrow \mathcal T(A)$ with $\mathcal T(A)$ satisfying
\begin{equation*}
\De \mathcal T(A) = \partial (\lambda^2) +\d( h\partial A).
\end{equation*}
so that the solution $A$ may be seen as a fixed point for $\mathcal T$.  To use the contraction principle, it suffices to show that, under the assumption \eqref{A-start}, this map  is Lipschitz in the ball $\{ A:\||D|^{\de_d} A\|_{H^{s+1-\de_d}}\leq C\|\la\|_{H^s} \}$
with a small Lipschitz constant. 
This would yield  the existence and uniqueness of solutions for $A$-equations and the bound \eqref{Ellip-A0}.
    
To establish the contraction property,
we consider the linearization of  $\mathcal T$,
\begin{align*}
\De {\bf \delta}\mathcal T(A)=\nab(\la \dla)+\nab(\dh \nab A+h\nab\dA),
\end{align*}
under the assumptions 
\[
\||D|^{\de_d} A\|_{H^{s+1-\de_d}} + \||D|^{\si_d}h\|_{H^{s+1-\si_d}}+ \|\la\|_{H^s}\lesssim \ep_0.
\]
Here we denote by $\tia_{0k}$, $\tc_{0k}$ and $\tp_{0k}$ are admissible frequency envelopes for $|D|^{\de_d} \dA\in H^{\si+1-\de_d}$, $|D|^{\si_d}\dh\in H^{\si+1-\si_d}$ and $\dla\in H^\si$ respectively. Under these assumptions  we will prove that
the above linearization  satisfies the bound
    \begin{equation}      \label{dA0-delta}
    \| S_k |D|^{\de_d} {\bf \delta}\mathcal T(A)\|_{H^{\si+1-\de_d}}\lesssim \ep_0 ( \tia_{0k}+\tc_{0k}+\tp_{0k}),
    \end{equation}
If the bound \eqref{dA0-delta} is true, then by the contraction principle we immediately get a unique small solution $A$ to our 
equations,  as well as the  linearized bound \eqref{A0-linearize}.

    We can also use \eqref{dA0-delta} in order to prove
    the frequency envelope bounds \eqref{A0-envelope}. Indeed, by \eqref{Freq-envelope} and \eqref{Ellip-A0} we have
    \begin{align*}
    a_{0k}= &\ 2^{-\de k}\||D|^{\de_d} A\|_{H^{s+1-\de_d}}+\max_{j}2^{-\de|j-k|}\|S_j|D|^{\de_d} A\|_{H^{s+1-\de_d}}\\
    \lesssim &\  2^{-\de k} (\ep_0\||D|^{\de_d} A\|_{H^{s+1-\de_d}}+\ep_0\|\la\|_{H^s})+\max_{j}2^{-\de|j-k|}(\ep_0 a_{0j}+\ep_0 p_{0j})\\
    \lesssim &\ \ep_0  a_{0k}+\ep_0 p_{0k}.
    \end{align*}
    This implies \eqref{A0-envelope} for $\ep_0$ sufficiently small.
    
    \medskip

    It remains to prove the bound \eqref{dA0-delta}.
    We have
    \begin{align*}
        \|S_k |D|^{\de_d} {\bf \delta}\mathcal T(A) \|_{H^{\si+1-\de_d}}\lesssim \| S_k |D|^{-1+\de_d}(\dh\nab A+h\nab\dA
        +\la\dla)\|_{H^{\si+1-\de_d}}.
    \end{align*}
    Here we only estimate the term $\la\dla$; the others are similar.  Precisely, when $k=0$, using a Littlewood-Paley decomposition, Bernstein's inequality and \eqref{FreEve-relation} we obtain
    \begin{align*}
    \| |D|^{-1+\de_d} S_0(\la\dla)\|_{L^2}\lesssim&\ \|\la\|_{L^2}\|S_0\dla\|_{L^2}+\sum_{j\geq 0} 2^{-\si j}\|\la_j\|_{L^2} 2^{\si j}\|\dla_j\|_{L^2}\\
    \lesssim & \ \|\la\|_{L^2}\tp_{00}+ \sum_{j\geq 0} 2^{(\de-\si)j}\|\la_j\|_{L^2} \tp_{00}\\
    \lesssim & \  \|\la\|_{H^s}\tp_{00}.
    \end{align*}
    When $k>0$, we use Bernstein's inequality to bound the high-low and low-high interactions by $\|\la\|_{H^s}\tp_{0k}$. For the high-high interaction we have
    \begin{align*}
    \| |D|^{-1+\de_d} S_k(\sum_{l>k}\la_l \dla_l)\|_{H^{\si+1-\de_d}}\lesssim & \ \sum_{l>k} 2^{(\si+\frac{d}{2})k}\|\la_l\|_{L^2}\|\dla_l\|_{L^2}\\
    \lesssim & \ \textbf{1}_{\leq -d/2}(\si)\sum_{l>k} 2^{(\si+\frac{d}{2}-\de)k}2^{(\de-\si)l}\|\la_l\|_{L^2}2^{\si l}\|\dla_l\|_{L^2}\\
    &+\textbf{1}_{> -d/2}(\si)\sum_{l>k} 2^{(\si+\frac{d}{2}+\de)(k-l)}2^{(\frac{d}{2}+\de)l}\|\la_l\|_{L^2}2^{\si l}\|\dla_l\|_{L^2}\\
    \lesssim & \ \|\la\|_{H^s}\tp_{0k}.
    \end{align*}    
    This concludes the proof of bound \eqref{dA0-delta}, and completes the proof of the lemma.
\end{proof}

\bigskip

\section{Estimates of parabolic equations}\label{Sec-Para}

Here we consider the solvability of the parabolic system \eqref{par-syst}.
For this purpose we view $\lambda \in l^2X^s$ as a parameter, and 
show that the solution $(h,A) \in \bEE^s$ exists, it is small and has 
a Lipschitz dependence on both the initial data and on $\lambda$.

\begin{thm}\label{para-thm}
a) Let $d \geq 2$, $s > d/2$. Assume that $\|h_0\|_{\bY_0^{s+2}}\leq  \ep$, $\||D|^{\de_d} A_0\|_{H^{s+1-\de_d}}\leq \ep$ and $\|\la\|_{l^2Z^s}\leq \ep$. Then the parabolic system \eqref{par-syst}-\eqref{ini-gA} admits a unique small solution $\SS= (h,A)$ in $\bEE^s$,
with
\begin{equation}   \label{para-bd0}
\|\SS\|_{\bEE^s} \lesssim 
\|\SS_0\|_{\bEE^s_0}
+ \| \la\|_{l^2Z^s}.    
\end{equation}
In addition 
this solution has a Lipschitz dependence on both 
$\SS_0$ in $\bEE^s_0$ and 
$\la$ in $l^2Z^s$.
Moreover, assume that $s_{0k}$ and $p_k$ are admissible frequency envelopes for $(h_0, A_0)\in \bEE_0^{s}$, $\la\in l^2Z^s$ respectively, we have the frequency envelope version 
\begin{equation}   \label{para-env0}
	  \| \SS_k \|_{\bEE^s} \lesssim s_{0k} +\ep p_k.
\end{equation}

b) In addition, for the linearization of the  parabolic system \eqref{par-syst} we have 
the bounds
\begin{equation}         \label{para-s}
\|   \dSS \|_{\bEE^s} \lesssim \|\dSS_0\|_{\bEE_0^{s}}+ \ep \|\dla\|_{l^2Z^s},
\end{equation}
and
\begin{equation}         \label{para-delta0}
\|   \dSS \|_{\EE^\si} \lesssim \|\dSS_0\|_{\HH^{\si}}+ \ep \|\dla\|_{Z^\sigma},
\end{equation}
for $\sigma \in (\frac{d}{2}-2,s]$. 
\end{thm}

We will do this in two steps. First  we prove that this system is solvable in the larger space $\EE^s$. Then we improve the space-time bounds for the metric $h$ to the stronger norm $\bY^{s+2}$; the latter will be needed in the study of the Schr\"odinger evolution \eqref{mdf-Shr-sys-2}.

\begin{lemma}    \label{eq-Lem1}
	Let $g=I_d+h$. Assume that $\| h\|_{ Z^{\si_d,s+1}}\leq\ep$ for $s>d/2$ and $d\geq 2$. Let $c_k$ and $a_k$ be admissible frequency envelopes for $h\in  Z^{\si_d,s+1}$, respectively $A\in Z^{s}$. Then for any $d/2-2< \si\leq s$ and a linearization operator $\bm \de$ we have
	\begin{gather}  \label{eq-hsi}
	\| \bm \de(h g)\|_{Z^{\si_d,\si+1}}\lesssim  \|\dh\|_{Z^{\si_d,\si+1}},\\
	  \label{eq-Asi}
	\| \bm \de(A g)\|_{Z^{\si}}\lesssim  \|\dA\|_{Z^{\si}}+ \| A\|_{ Z^{s}} \|\dh\|_{Z^{\si_d,\si}},
	\end{gather}
	and hence we have
	\begin{align}   \label{eq-hs} 
	\|S_k(hg)\|_{Z^{\si_d, s+1}}\lesssim  c_k,\\	\label{eq-As}
	\| S_k (A g)\|_{Z^{s}}\lesssim a_k,
	\end{align}
\end{lemma}
\begin{proof}
	Assume that $\tc_k(\si)$ and $\tia_k$ are admissible frequency envelopes for $\dh \in  Z^{\si_d,\si}$ and $\dA\in  Z^{\si}$. Using a Littlewood-Paley decomposition, Bernstein's inequality and the smallness of $h$ we obtain
	\begin{align*}
	\| S_k(\dh h)\|_{ Z^{\si_d,\si+1}}\lesssim \ep \tc_k(\si+1)+\|\dh\|_{ Z^{\si_d,\si+1}}c_k.
	\end{align*}
	This implies  \eqref{eq-hsi} and \eqref{eq-hs} immediately.
    For $A$ we have
	\begin{align*}
	\| S_k(\dA h)\|_{ Z^{\si}}\lesssim \ep \tia_k,
	\end{align*}
	and
	\begin{align*}
	\| S_k(A \dh)\|_{Z^{\si}}\lesssim a_k \|\dh\|_{ Z^{\si_d,\si}} +\| A\|_{ Z^{s}} \tc_k(\si).
	\end{align*}
	These give \eqref{eq-Asi} and \eqref{eq-As}.
\end{proof}

Now, we begin to solve the parabolic system \eqref{par-syst} with initial data \eqref{ini-gA} as follows:

\begin{prop}\label{para-prop}
a) Assume that $\|(h_0,A_0)\|_{\HH^s}\leq \ep$ and $\|\la\|_{Z^s}\leq \ep$ for $s > d/2$ and $d \geq 2$. Then the parabolic system \eqref{par-syst}-\eqref{ini-gA} admits a unique small solution $\SS= (h,A)$ in $\mathcal E^s$,
with
\begin{equation}   \label{para-bd}
\|\SS\|_{\mathcal E^s} \lesssim \|\SS_0\|_{\HH^{s}}+ \| \la\|_{Z^s}.    
\end{equation}
In addition 
this solution has a Lipschitz dependence on
$\SS_0$ in $\HH^{s}$ and
$\la$ in $Z^s$.
Moreover, assume that $s_{0k}$ and $p_k$ are admissible frequency envelopes for $\SS_0\in \HH^{s}$, $\la\in Z^s$ respectively, then we have the frequency envelope version 
\begin{equation}   \label{para-envelope}
	  \| \SS_k \|_{\mathcal E^s} \lesssim s_{0k} +\ep p_k.
\end{equation}

b) In addition, for the linearization of the  parabolic system \eqref{par-syst} we have 
the bounds
\begin{equation}         \label{para-delta}
\|   \dSS \|_{\mathcal E^\si} \lesssim \|\dSS(0)\|_{\HH^{\si}}+ \ep \|\dla\|_{Z^\sigma},
\end{equation}
for $\sigma \in (\frac{d}{2}-2,s]$. 
\end{prop}

\begin{proof}
    First, we consider a linear equation and prove a linear estimate. Precisely, assume that the frequency localized function $u_k$ is solution of the linear equation
	\begin{equation*}
	\d_t u_k-\De u_k=f_k,\qquad
	u_k(0)=u_{0k}.
	\end{equation*}
	Then by Bernstein's inequality we have the linear estimates
	\begin{align*}
	\frac{1}{2}\frac{d}{dt}\| u_k\|_{L^2}^2
	\leq &- c 2^{2k} \| u_k\|_{L^2}^2+\|u_k\|_{L^2}\|f_k\|_{L^2}.
	\end{align*}
	We cancel one $\|u_k\|_{L^2}$, then multiply both sides by $e^{c 2^{2k}t}$ and integrate in time to obtain
	\begin{equation}     \label{lin-u}
	\begin{aligned}
	\| u_k(t)\|_{L^2}
	\lesssim &e^{-c 2^{2k}t}\|u_{0k}\|_{L^2}+2^{-2k} \|f_k(s)\|_{L^\infty L^2}.
	\end{aligned}
	\end{equation}

    In order to solve \eqref{par-syst} with small initial data, it suffices to consider the following linearized equations
    \begin{align*}
        &\d_t \dh_k -\De \dh_k=\mathcal N_1,\qquad
        \d_t \dA_k -\De\dA_k=\mathcal N_2,
    \end{align*}
    where the nonlinearities $\mathcal N_1$ and $\mathcal N_2$ are
    \begin{align*}
	\mathcal N_1=&S_k(\dh\nab^2 h+h\nab^2\dh+\dh\nab h\nab h+g\nab h \nab\dh+\la\dla),\\
    \mathcal N_2=&S_k(h\nab^2 \dA+\dh \nab^2 A+\nab h\nab \dA+\nab\dh\nab A\\
    &+\nab^2 h\ \dA+\nab^2\dh\  A
    +\nab h\nab h \dA+\nab h\nab\dh A
    +\la\nab\dla\\
    &+\dla\nab\la
    +\la\dla(\nab h+A)+\la^2(\nab\dh+\dA)),
    \end{align*}
    with $h$, $A$ and $\la$ satisfying $\|(h,A)\|_{\EE^s}\lesssim \ep$, $\|\la\|_{Z^s}\leq \ep$.
    Then we will prove the bound
    \begin{equation}\label{par-h_keybd}
    \begin{aligned}   
    \| S_k\dSS\|_{\EE^{\si}}
    \lesssim  \ts_{0k}+\ep(\ts_{k}+\tp_k)+(s_k+p_k)(\| \dSS\|_{\EE^{\si}}+\|\dla\|_{Z^\si}),
    \end{aligned}
    \end{equation}
    where $\ts_{0k}$, $\ts_k$, $\tp_k$ and $s_k$ are admissible frequency envelopes for $\dSS(0)\in \HH^\si$, $\dSS\in \EE^\si$, $\dla\in Z^\si$ and $\SS\in \EE^s$ respectively.
    
    \medskip 
    
Assuming the bound \eqref{par-h_keybd} is true, then we can use the contraction mapping principle to solve the parabolic system \eqref{par-syst} in the space
    \[ \big\{ \SS=(h,A)\in \EE^s:\|\SS\|_{\EE^{s}}\leq C(\|\SS_0\|_{\HH^{s}}+ \| \la\|_{Z^s})\leq 2C\ep \big\}, \]
    which also implies the bound \eqref{para-bd}.
    
    By the definition of frequency envelopes \eqref{Freq-envelope} and \eqref{para-bd}, the bound \eqref{par-h_keybd} with $\si=s$ and ${\bm\de}=Id$ implies
    \begin{align*}
    s_k    \lesssim  s_{0k}+\ep (s_k+p_k).
    \end{align*}
    Thus the bound \eqref{para-envelope} follows.
    By \eqref{para-bd}, the bound \eqref{par-h_keybd} also gives \eqref{para-delta}.
    
    \medskip
    
    We now return to the proof of \eqref{par-h_keybd}. By the energy estimates in \eqref{lin-u} we have
    \begin{align*}
    \| S_k\dSS(t)\|_{\EE^{\si}}
    \lesssim \|S_k\dSS_0\|_{\HH^{\si}}
    +\| |D|^{\si_d} S_k\mathcal N_1\|_{L^\infty H^{\si-\si_d}}
    +\| S_k\mathcal N_2\|_{L^\infty H^{\si-1}}.
    \end{align*}
    The estimates for the nonlinearities are similar, here we only estimate the following terms. 
    
    \medskip 
    \emph{A. The estimate for the terms  $h\nab^2\dh$ and $\la\dla$ in $\mathcal N_1$.} Using a Littlewood-Paley decomposition we have
    \begin{align*}
    \|S_k(h\nab^2\dh)\|_{L^\infty H^\si}
    \lesssim &\ 2^{\si k} \| h_{\leq k}\|_{L^\infty L^\infty}\|\nab^2\dh_k\|_{L^\infty L^2}+\sum_{l\leq k} 2^{\si k+(d/2+2)l}\| h_{k}\|_{L^\infty L^2}\|\dh_l\|_{L^\infty L^2}\\
    &+\sum_{l> k} 2^{(\si+d/2) (k-l)+(\si+d/2+2)l}\| h_{l}\|_{L^\infty L^2}\|\dh_l\|_{L^\infty L^2}\\
    \lesssim &\ \||D|^{\si_d} h\|_{L^\infty H^{s-\si_d}}\tc_k+2^{sk+2\de k} \|h_k\|_{L^\infty L^2}\tc_k
    +\tc_k \||D|^{\si_d}h\|_{L^\infty H^{s+1-\si_d}}\\
    \lesssim &\ \tc_k \||D|^{\si_d}h\|_{L^\infty H^{s+1-\si_d}},
    \end{align*}
    and
    \begin{align*}
    \|S_k(\la\dla)\|_{L^\infty H^\si}
    \lesssim &\ 2^{\si k} \| \la_{\leq k}\|_{L^\infty L^\infty}\|\dla_k\|_{L^\infty L^2}+\sum_{l\leq k} 2^{\si k+dl/2}\| \la_{k}\|_{L^\infty L^2}\|\dla_l\|_{L^\infty L^2}\\
    &+\sum_{l> k} 2^{(\si+d/2) (k-l)+(\si+d/2)l}\| \la_{l}\|_{L^\infty L^2}\|\dla_l\|_{L^\infty L^2}\\
    \lesssim &\ \| \la\|_{L^\infty H^s}\tp_k+p_k \|\dla\|_{L^\infty H^\si}+\sum_{l> k} 2^{(\si+d/2-\de) (k-l)}\| \la_{l}\|_{L^\infty H^{d/2}}\tp_k\\
    \lesssim &\ \tp_k\| \la\|_{L^\infty H^s}+p_k \|\dla\|_{L^\infty H^\si}.
    \end{align*}
    
    \medskip 
    \emph{B. The estimate for the terms $\nab^2 h\dA$ and $\la\nab\dla$ in $\mathcal N_2$}. The second term is estimated in the same manner as the above bound for $\la\dla$ in $\mathcal N_1$. For the first term $\nab^2 h\dA$ we have
    \begin{align*}
    &\| S_k(\nab^2 h\dA)\|_{L^\infty H^{\si-1}}\\
    \lesssim &\ \| |D|^{\si_d} h\|_{L^\infty H^{s+1-\si_d}}\| \dA_k\|_{L^\infty H^{\si}}+\sum_{l\leq k} 2^{\si k+k}\| h_{k}\|_{L^\infty L^2}\|\dA_l\|_{L^\infty H^{d/2}}\\
    &+\sum_{l> k} 2^{(\si+d/2-1) (k-l)+(\si+d/2+1)l}\| h_{l}\|_{L^\infty L^2}\|\dA_l\|_{L^\infty L^2}\\
    \lesssim &\ \ep\tia_k+c_k\|\dA\|_{L^\infty H^{\si+1}}+\sum_{l> k} 2^{(\si+d/2-1) (k-l)}\| h_{l}\|_{L^\infty H^{d/2+\de}}\tia_k\\
    \lesssim &\ \ep\tia_k+c_k\|\dA\|_{L^\infty H^{\si+1}}.
    \end{align*}
    This concludes the proof of the bound \eqref{par-h_keybd}, and completes the proof of the theorem.
\end{proof}

\medskip

We continue with the bound for the $l^2Z^{\si_d,s+2}$-norm of the metric $h$.

\begin{prop}\label{cor-4.3}
Assume that $\|(h_0,A_0)\|_{\HH^s}\leq \ep$ and $\|\la\|_{l^2Z^s}\leq \ep$ for $s>d/2$ and $d\geq 2$. Then the solution $h$ also belongs to $l^2Z^{s+2}$ and satisfies the bounds
\begin{equation}   \label{h-lZ}
\| h\|_{l^2Z^{\si_d,s+2}} \lesssim \||D|^{\si_d}h_0\|_{H^{s+2-\si_d}}+ \| \la\|_{l^2 Z^s}.
\end{equation}
Assume that $c_{0k}$ and $p_k$ are admissible frequency envelopes for $|D|^{\si_d}h_0\in H^{s+2-\si_d}$, $\la\in l^2Z^s$ respectively. Then we have the frequency envelope bounds
\begin{equation}  \label{h-lZ_envelope}
\|  S_k  h\|_{l^2 Z^{\si_d,s+2}} \lesssim  c_{0k}+\ep p_k.
\end{equation}
Finally, for the linearization of the $h$-equations we have the bounds
\begin{equation}  \label{h-lZ_sigma}
\|  \dh\|_{l^2Z^{\si_d,s+2}} \lesssim  \|   |D|^{\si_d}\dh_0\|_{H^{s+2-\si_d}}+\ep\|\dla\|_{l^2Z^{s}}.
\end{equation}
\end{prop}

\smallskip

\begin{proof}[Proof of Proposition \ref{cor-4.3}]
We split the proof into two steps, where we first prove the 
appropriate bound for the linear constant coefficient heat flow and then we apply that bound to solve the nonlinear problem
perturbatively.

\smallskip

\emph{Step 1. Here we consider the linear equations 
\begin{equation}     \label{linEq-uQ}
    \d_t P_k u-\De P_k u=P_k f,
\end{equation}
with $P_k u$ localized at frequency $2^k$ for $k\in \mathbb Z$, and prove that
\begin{align}   \label{lin-uQ}
    \|P_k u\|_{l_{|k|}^2L^\infty L^2}
    \lesssim &\|P_k u(0)\|_{L^2}+2^{-2k^+} \| P_k f\|_{l_{|k|}^2L^\infty L^2}.
\end{align}
}

By Duhamel's formula, we have 
\begin{equation*}
    \|P_k u\|_{l^2_{|k|}L^\infty L^2}\lesssim \|e^{t\De} P_k u_0\|_{l^2_{|k|}L^\infty L^2}+\|\int_0^t  e^{(t-s)\De} P_k f ds\|_{l^2_{|k|}L^\infty L^2}.
\end{equation*}
Then we use \eqref{lin-heat} and \eqref{NLin-heat} to bound the above two terms respectively, then we obtain \eqref{lin-uQ}.

\bigskip

\emph{Step 2.}
	Here it suffices to write the linearized $h$ equation in the form
	\begin{equation*}
	\d_t \dh-\De\dh =\dh\nab^2 h+h\nab^2\dh+\dh\nab h\nab h+g\nab h \nab\dh+\la\dla:=\mathcal N,
	\end{equation*}
	and to prove that
	\begin{equation}    \label{h-lZ-key}
	\begin{aligned}
	\|S_k\dh\|_{l^2 Z^{\si_d,s+2}}\lesssim & \||D|^{\si_d}S_k\dh_0\|_{H^{s+2-\si_d}}+\ep (\tc_k+\tp_k) \\
	&+ (c_k+p_k)( \|\dh\|_{l^2Z^{\si_d,s+2}}+\|\dla\|_{l^2Z^{s}}),
	\end{aligned}
	\end{equation}
	where $\tc_k,\ \tp_k$ and $c_k$ are admissible frequency envelopes for $\dh \in l^2 Z^{\si_d,s+2}$, $\dla\in l^2Z^{s}$ and $h\in l^2 Z^{\si_d,s+2}$ respectively.
	
	If the bound \eqref{h-lZ-key} is true, then we choose the operator ${\bm \de}=Id$ to obtain \eqref{h-lZ}. Then by \eqref{h-lZ-key} and \eqref{Freq-envelope} we also obtain \eqref{h-lZ_envelope}. The bound \eqref{h-lZ-key} combined with \eqref{h-lZ} also implies \eqref{h-lZ_sigma}.

	We now continue with the proof of \eqref{h-lZ-key}.
	By \eqref{lin-uQ} we have
	\begin{align*}
	\| S_k \dh\|_{l^2Z^{\si_d,s+2}}
	\lesssim &\| |D|^{\si_d}\dh_{0k}\|_{H^{s+2-\si_d}}+\|S_k  \mathcal N\|_{l^2Z^{\si_d,s}}.
	\end{align*}
	For the nonlinearities, we only estimate $h\nab^2 \dh$ and $\la\dla$, the others are estimated similiarly. 
	Indeed, using a Littlewood-Paley decomposition we have 
	\begin{align*}
	\|P_k(h\nab^2\dh)\|_{l^2_{|k|}L^\infty L^2}\lesssim &\ \| h\|_{L^\infty L^\infty}2^{2k}\|P_k\dh\|_{l^2_{|k|}L^\infty L^2}
	+\| P_k h\|_{l^2_{|k|}L^\infty L^2}\| \nab^2\dh\|_{L^\infty L^\infty}\\
	&+\sum_{|k|\geq l>k}  2^{\frac{d}{2}k+2l} \| P_l h\|_{L^\infty L^2}\| P_l\dh\|_{l^2_lL^\infty L^2}\\
	&+\sum_{l>|k|}  2^{(\frac{d}{2}+2)l} \| P_l h\|_{L^\infty L^2}\| P_l\dh\|_{l^2_lL^\infty L^2}.
	\end{align*}
	By this estimate and Sobolev embedding we obtain
	\begin{align*}
	    \|S_k(h\nab^2\dh)\|_{l^2 Z^{\si_d,s}}\lesssim &\  \||D|^{\si_d}h\|_{L^\infty H^{s-\si_d}} \tc_k+c_k \||D|^{\si_d}h\|_{L^\infty H^{s+2-\si_d}}.
	\end{align*}
	For the term $\la\dla$, we also have
	\begin{align*}
	2^{s k}\|S_k(\la\dla)\|_{l^2_kL^\infty L^2}\lesssim &\ \| \la\|_{l^2Z^s}\|\dla_k\|_{L^\infty H^{s}}+
	\| \la_k\|_{l^2_kL^\infty H^s}\| \dla\|_{Z^{s}}\\
	&+\sum_{l>k} 2^{s k} 2^{\frac{d}{2}l} \| \la_l\|_{L^\infty L^2}\| \dla_l\|_{l^2_lL^\infty L^2}\\
	\lesssim &\ \ep \tp_k + p_k \|\dla\|_{l^2Z^{s}}.
	\end{align*}
	This completes the proof of Proposition \ref{cor-4.3}.
\end{proof}

\bigskip

Finally, we carry out the last step in the proof of Theorem~\ref{para-thm}, and establish bounds for the solutions $h$ in the $Y^{s+2}$ spaces: 

\begin{prop}\label{para-h_thm}
Let $d\geq 2$, $s>d/2$. Assume that $\|(h_0,A_0)\|_{\HH^s}\leq \ep$ and $\|\la\|_{l^2Z^s}\leq \ep$. Then we have the bound
\begin{equation}  \label{h-bdd_Y}
\|   h\|_{Y^{s+2}} \lesssim  \| h_0\|_{\bY_0^{s+2}}+\|\la\|_{l^2Z^s}.
\end{equation}
with Lipschitz dependence on the initial data in these topologies. Moreover, assume that $c_{0k}$ and $p_k$ are admissible frequency envelope for $h(0)\in \bY_0^{s+2}$ and $\la\in l^2 Z^{s}$, then we have the frequency envelope version 
\begin{equation}   \label{h-Y_envelope}
	  \| S_k h \|_{Y^{s+2}} \lesssim c_{0k}+ \ep p_k   .
\end{equation}
In addition, for the linearization of the  elliptic system \eqref{par-syst} we have 
the bounds
\begin{equation}         \label{h-Y_linear}
\|   \dh \|_{Y^{s+2}} \lesssim \|   \dh_0 \|_{\bY_0^{s+2}} +\ep\|\dla\|_{l^2Z^{s}}.
\end{equation}
\end{prop}

\begin{proof}
    Again it suffices to write the $h$ equation in the form:
	\begin{equation*}
	\d_t \dh-\De \dh=\dh\nab^2 h+h\nab^2\dh+\dh\nab h\nab h+g\nab h \nab\dh+\la\dla:=\mathcal N,
	\end{equation*}
	and to prove that
	\begin{equation}    \label{h-Y-key}
	\begin{aligned}
	\|\dh_k\|_{Y^{s+2}}\lesssim & \|\dh_{0k}\|_{Y_0^{s+2}}
	+\ep (\tc_k+\tp_k)+(c_k+p_k)(\|\dh\|_{\bY^{s+2}}+\|\dla\|_{l^2Z^s}) ,
	\end{aligned}
	\end{equation}
	where $\tc_k$ and $\tp_k$ are admissible frequency envelopes for $\dh \in \bY^{s+2}$ and $\dla\in l^2Z^{s}$ respectively.
	
	If \eqref{h-Y-key} is true, then the bound \eqref{h-bdd_Y} is obtained by \eqref{h-Y-key} with the operator ${\bm\de}=Id$ and the bound \eqref{h-lZ}. We also obtain \eqref{h-Y_envelope} by \eqref{h-Y-key} and \eqref{h-lZ_envelope}. The bound \eqref{h-Y-key} combined with \eqref{h-lZ_sigma} also implies \eqref{h-Y_linear}.
	
	\medskip
	
    We now return to prove the bound \eqref{h-Y-key}. By Duhamel's formula, \eqref{lin-heat} and \eqref{NLin-heat}, we have
    \begin{align*}
    \| S_k\dh\|_{Y^{s+2}}\lesssim \|e^{t\De}\dh_{0k}\|_{Y^{ s+2}}+\|\int_0^t e^{(t-s)\De} S_k\mathcal N\ ds\|_{Y^{ s+2}}.
    \end{align*}
    We estimate the first term in the right hand side. For  any decomposition $P_j \dh(0)=\sum_{l\geq |j|} \dh_{j,l}(0)$, by \eqref{lin-heat} we have
    \begin{align*}
    \| e^{t\De} P_j \dh(0)\|_{Y_j}\lesssim& \inf_{P_j \dh(0)=\sum_{l\geq |j|} h_{j,l}(0)}\sum_{l\geq |j|} 2^{l-|j|}\| e^{t\De} \dh_{j,l}(0)\|_{l^1_{|l|}L^\infty L^2}\\
    \lesssim & \inf_{P_j \dh(0)=\sum_{l\geq |j|} \dh_{j,l}(0)}\sum_{l\geq |j|} 2^{l-|j|}\|  \dh_{j,l}(0)\|_{l^1_{|l|} L^2}\\
    = & \| \dh(0)\|_{Y_{0j}}.
    \end{align*}
    This gives the bound for the first term.

    Next, for the nonlinearities, we only estimate the Duhamel contributions of $h\nab^2\dh$ and $\la\dla$ in detail. 
    In order to bound the contribution of term $h\nab^2\dh$, we use the Littlewood-Paley trichotomy to decompose it into three cases:
   
    \medskip
    
    \emph{a) Low-high interactions: $P_j(P_{<j}h\nab^2 P_{j}\dh)$}.  By \eqref{NLin-heat}, for any decomposition $P_j \dh=\sum_{l\geq |j|} \dh_{j,l}$ we have
    \begin{align*}
        \|\int_0^t e^{(t-s)\De}P_j(P_{<j}h\nab^2 P_j\dh) ds\|_{Y_j}
        =&\
        \sum_{l\geq |j|} 2^{l-|j|} 2^{-2j^+}\| P_j(P_{<j}h\nab^2 \dh_{j,l}) \|_{l^1_{l}L^\infty L^2}\\
        \lesssim & \ 2^{2j-2j^+}
        \|P_{<j}h\|_{L^\infty L^\infty}\sum_{l\geq |j|} 2^{l-|j|}\| \dh_{j,l} \|_{l^1_l L^\infty L^2}\\
        \lesssim & \ 2^{2j-2j^+}\| |D|^{\si_d}h\|_{L^\infty H^{s-\si_d}}\|P_j\dh\|_{Y_j}.
    \end{align*}
    This implies both the low-frequency part bound 
    \begin{align*}
    \|\int_0^t e^{(t-s)\De} \sum_{j\leq 0} P_j(P_{<j}h\nab^2 P_j\dh) ds\|_{Y^{ s+2}}\lesssim \ep\| S_0\dh\|_{Y^{ s+2}},
    \end{align*}
    and the high frequency part bound
    \begin{align*}
    \|\int_0^t e^{(t-s)\De}S_j(P_{<j}h\nab^2 P_j\dh) ds\|_{Y^{ s+2}}\lesssim \ep \| \dh_j\|_{Y^{ s+2}}.
    \end{align*}
    
    \medskip 
    
    \emph{b) The high-low interactions $P_j(P_j h\nab^2 P_{<j+O(1)}\dh)$} are estimated in the same manner as the above \emph{low-high} case, so we omit the computations.
    
    \medskip
    
    \emph{c) High-high interactions: $\sum_{l>j}P_j(P_lh\nab^2 P_l \dh)$}. This sum can be further decomposed as $\sum_{l>j}=\sum_{|j|>l>j}+\sum_{l\geq |j|}$. Then by \eqref{NLin-heat} we bound the contribution of the first term by
    \begin{align*}
        2^{(\frac{d}{2}-\de)j^-}\|\int_0^t e^{(t-s)\De}\sum_{|j|>l>j}P_j(P_lh & \nab^2 P_l \dh) ds\|_{Y_j}
        \lesssim  2^{(\frac{d}{2}-\de)j^-} \sum_{|j|>l>j}
        \| P_j(P_lh\nab^2 P_l \dh) \|_{l^1_{|j|}L^\infty L^2}\\
        \lesssim &\  \sum_{|j|>l>j} 2^{(d-\de)(j-l)+(d+2-\de)l} 
        \| P_lh\|_{l^2_{|j|}L^\infty L^2}\| P_l\dh\|_{l^2_{|j|}L^\infty L^2}\\
        \lesssim & \ \sum_{|j|>l>j} 2^{(d-\de)(j-l)} 
        \| P_lh\|_{l^2 Z^{\si_d,s}}\tc_0.
    \end{align*}
    Also by \eqref{NLin-heat} we bound the contribution of second term by
    \begin{equation}    \label{5.23-hh}
        \begin{aligned}
        &2^{(\frac{d}{2}-\de)j^-+( s+2)j^+}\|\int_0^t e^{(t-s)\De}\sum_{l>|j|}P_j(h_{l}\nab^2 \dh_l) ds\|_{Y_j}\\
        \lesssim &\ 2^{(\frac{d}{2}-\de)j^-+( s+2)j^+}
        \sum_{l>|j|}2^{l-|j|}\|\int_0^t  e^{(t-s)\De} P_j(h_{l}\nab^2 \dh_{l}) ds\|_{l^1_{l}L^\infty L^2}\\
        \lesssim &\  2^{(\frac{d}{2}-\de)j^-+( s+2)j^+}\sum_{l>|j|}2^{l-|j|+dj/2-2j^+ +2l}
        \|h_{l} \|_{l^2_{|j|}L^\infty L^2}\|\dh_{l} \|_{l^2_{|j|}L^\infty L^2}
    \end{aligned}
    \end{equation}
     This term is further controlled by
    \begin{align*}
        \text{LHS\eqref{5.23-hh}}\lesssim &\ 2^{(d+1-\de)j^-+( s+\frac{d}{2}-1)j^+}
        \sum_{l>|j|}2^{3l}\|h_{l}\|_{l^2_{l}L^\infty L^2}\|\dh_{l}\|_{l^2_{l}L^\infty L^2}\\
        \lesssim &\ 2^{(d+1-\de)j^-} \|h \|_{l^2 Z^{\si_d,s+1}} \tc_0+ \mathbf{1}_{>0}(j)\| h\|_{l^2Z^{\si_d,s+1}} \tc_j.
    \end{align*}
    This concludes the proof of the bound for the contribution of
    $h\nab^2\dh$. Next we consider the term $\la\dla$. We also split its  analysis into three cases:
    
    \medskip
    
    \emph{a) Low-high interactions: $P_j(P_{<j}\la P_j\dla)$ and high-low interactions: $P_j(P_j\la P_{<j}\dla)$}. These two cases are similar, we only estimate the first term. By \eqref{NLin-heat}, we have
    \begin{align*}
        \|\int_0^t e^{(t-s)\De}P_j(P_{<j}\la P_j\dla) ds\|_{Y_j}
        \lesssim &\ 2^{-2j^+}
        \| P_j(P_{<j}\la P_j\dla) \|_{l^1_{|j|}L^\infty L^2}\\
        \lesssim &\ 2^{-2j^+}\| \la\|_{l^2Z^s} \|P_j\dla\|_{l^2_{|j|}L^\infty L^2},
    \end{align*}
    which is acceptable.
    
    \medskip 
    
    \emph{b) High-high interactions: $\sum_{l>j}P_j(P_l\la\cdot P_l\dla)$}. This sum can be further decomposed as $\sum_{l>j}=\sum_{|j|>l>j}+\sum_{l\geq |j|}$. By \eqref{NLin-heat} we bound the contribution of the first sum by
    \begin{align*}
        2^{(\frac{d}{2}-\de)j^-}\|\int_0^t e^{(t-s)\De}\sum_{|j|>l>j}P_j(P_l\la\cdot P_l\dla) ds\|_{Y_j}
        \lesssim &\ 2^{(\frac{d}{2}-\de)j^-} \sum_{|j|>l>j}
        \| P_j(P_l\la\cdot P_l\dla) \|_{l^1_{|j|}L^\infty L^2}\\
        \lesssim &\ 2^{(d-\de)j^-} \sum_{|j|>l>j} 
        \| \la_{l}\|_{l^2_{|j|}L^\infty L^2}\| \dla_{l}\|_{l^2_{|j|}L^\infty L^2}\\
        \lesssim &\ 2^{(d-2\de)j^-}  \| \la\|_{l^2Z^{s}}\tp_0.
    \end{align*}
    Next we bound the contribution of the second sum by
     \begin{equation*}   
         \begin{aligned}
        &2^{(\frac{d}{2}-\de)j^-+( s+2)j^+}\|\int_0^t e^{(t-s)\De}\sum_{l>|j|}P_j(\la_{l}\dla_l) ds\|_{Y_j}\\
        \lesssim &\ 2^{(\frac{d}{2}-\de)j^-+( s+2)j^+}
        \sum_{l>|j|}2^{l-|j|}\|\int_0^t  e^{(t-s)\De} P_j(\la_{l}\dla_l) ds\|_{l^1_{l}L^\infty L^2}\\
        \lesssim &\  2^{(\frac{d}{2}-\de)j^-+( s+2)j^+}\sum_{l>|j|}2^{l-|j|+dj/2-2j^+}
        \|\la_{l} \|_{l^2_{l}L^\infty L^2}\|\dla_{l} \|_{l^2_{l}L^\infty L^2}\\
        \lesssim & \  2^{(d+1-\de)j^-+( s+\frac{d}{2}-1)j^+}
        \sum_{l>|j|}2^{l}\|\la_{l}\|_{l^2_{l}L^\infty L^2}\|\dla_{l}\|_{l^2_{l}L^\infty L^2}\\
        \lesssim & \  2^{(d+1-\de)j^-}\| \la\|_{l^2Z^s}\tp_0+\mathbf{1}_{>0}(j) \| \la\|_{l^2Z^s}\tp_j.
    \end{aligned}
     \end{equation*}
    This concludes the proof of the bound \eqref{h-Y-key}, and  completes the proof of the proposition.
\end{proof}

\bigskip

\section{Multilinear and nonlinear estimates} \label{Sec-mutilinear}

This section contains our main multilinear estimates which are needed for the analysis  of the Schr\"{o}dinger equation in \eqref{mdf-Shr-sys-2}. We begin with the following low-high bilinear estimates of $\nab h\nab \la$. 

\begin{prop}   \label{bilinear-est}
	Let $s>\frac{d}{2}$, $d\geq 2$ and $k\in \N$. Suppose that $\nab a(x)\lesssim \<x\>^{-1}$, $ h\in  Y^{s+2}$ and $\la_k\in l^2X^s$. Then for $-s\leq  \si\leq s$ we have
	\begin{align}          \label{qdt_d h<k dpsi}
	&\lV \nab h_{\leq k}\cdot\nab\la_k\rV_{l^2N^{\si}}\lesssim \min\{ \lV  h\rV_{Y^{\si+2}}\lV \la_k\rV_{l^2X^s},\lV h\rV_{Y^{s+2}}\lV\la_k\rV_{l^2X^{\si}}\},\\\label{h<0-naba}
	&\lV h_{\leq k}\nab a\nab\la_k\rV_{l^2N^{\si}}\lesssim \min\{ \lV  h\rV_{Y^{\si+2}}\lV \la_k\rV_{l^2X^s},\lV h\rV_{Y^{s+2}}\lV\la_k\rV_{l^2X^{\si}}\}. 
	\end{align}
	In addition, if $d/2-2<\si\leq s-1$ then we have 
	\begin{align}     \label{BiL-si}
	\lV \nab h_{\leq k}\cdot\nab\la_k\rV_{l^2N^{\si}}\lesssim &\min\{ \lV  h\rV_{Z^{\si_d,\si+2}}\lV \la_k\rV_{Z^s},\lV h\rV_{Z^{\si_d,s+2}}\lV\la_k\rV_{Z^{\si}}\},
	\end{align} 
	and if $d/2-2<\si\leq s-2$ then we have 
	\begin{align}     \label{qdt_h-d2psi}
	\lV h_{\leq k}\nab^2\la_k\rV_{l^2N^{\si}}\lesssim & \|h\|_{Z^{\si_d,\si+2}}\|\la_k\|_{l^2X^s}.
	\end{align} 
\end{prop}
\begin{proof}
	\emph{a) The estimate \eqref{qdt_d h<k dpsi}}. 	This is obtained by a Littlewood-Paley decomposition and the following estimate
	\begin{align*}       
	\|\nab P_j h\nab \la_k\|_{l^2_kN_k}\lesssim 2^{\frac{d}{2}j+2j^+}\|P_j h\|_{Y_j}\|\la_k\|_{X_k},\qquad j\leq k,\ j\in\Z,\ k\in \mathbb{N},
	\end{align*}
	which has been proved in \cite[Lemma 5.1]{HT21}.
	
	\medskip
	
	\emph{b) The estimate \eqref{h<0-naba}.} Compared to  \cite[(5.2)]{HT21}, the estimate \eqref{h<0-naba} is improved by decomposing physical space dyadically.  By duality, it suffices to prove that
	\begin{equation}     \label{IIj}
	II_j= \int_0^1\< P_jh\nab a\nab\la_k,z_k\>dt\lesssim 2^{dj/2}\log |j|\|P_j h\|_{Y_j}\|\la_k\|_{X_k},\ \ j\leq k,\ j\in\Z,
	\end{equation}
	for any $z_k\in l^2_k X_k$ with $\lV z_k\rV_{l^2_k X_k}\leq 1$. For any decomposition $P_jh=\sum_{l\geq |j|}h_{j,l}$, using the bound $|\nab a(x)|\lesssim \<x\>^{-1}$, we consider the two cases $|x|\geq 2^{l}$ and $|x|<2^{l}$ respectively and then obtain
	\begin{align*}
	II_j\lesssim & \ \sum_{l\geq |j|} \sup_{\lV z_k\rV_{l^2_k X_k}\leq 1}\sum_{0\leq l_1\leq l}\int_0^1\<  h_{j,l}\<x\>^{-1}\mathbf{1}_{ [2^{l_1-1},2^{l_1}]}(x)\nab\la_k,z_k\>dt\\
	&+\sum_{l\geq |j|} \sup_{\lV z_k\rV_{l^2_k X_k}\leq 1}\int_0^1\<  h_{j,l}\<x\>^{-1}\mathbf{1}_{>2^{l}}(x)\nab\la_k,z_k\>dt\\
	=& \  II_{j1}+II_{j2}.
	\end{align*}
	By Bernstein's inequality we bound the first term by
	\begin{align*}
	II_{j1}\lesssim & \  \sum_{l\geq |j|} \sup_{\lV z_k\rV_{l^2_k X_k}\leq 1}\sum_{0\leq l_1\leq l}2^{-l_1}\lV  h_{j,l}\rV_{L^{\infty}L^{\infty}} \lV \nab \la_k\rV_{l^{\infty}_{l_1}L^2L^2}\lV z_k\rV_{l^{\infty}_{l_1}L^2L^2}\\
	\lesssim & \ \sum_{l\geq |j|} \sum_{0\leq l_1\leq l}2^{\frac{dj}{2}}\lV  h_{j,l}\rV_{L^{\infty}L^{2}} \lV  \la_k\rV_{X_k}\\
	\lesssim & \ 2^{dj/2}\log |j| \sum_{l\geq |j|}  \frac{\log l}{\log |j|} \lV  h_{j,l}\rV_{l^1_l L^{\infty}L^2} \lV  \la_k\rV_{X_k}.
	\end{align*}
	The second term is bounded by
	\begin{align*}
	II_{j2}\lesssim & \ \sum_{l\geq |j|} 2^{-l}\sup_{\lV z_k\rV_{l^2_k X_k}\leq 1}\lV  h_{j,l}\rV_{l^1_l L^{\infty}L^{\infty}} \lV \nab \la_k\rV_{l^{\infty}_{l}L^2L^2}\lV z_k\rV_{l^{\infty}_{l}L^2L^2}\\
	\lesssim & \ \sum_{l\geq |j|} \lV  h_{j,l}\rV_{l^1_l L^{\infty}L^{\infty}} \lV  \la_k\rV_{X_k}\\
	\lesssim & \  2^{dj/2}\sum_{l\geq |j|} \lV  h_{j,l}\rV_{l^1_l L^{\infty}L^2} \lV  \la_k\rV_{X_k}.
	\end{align*}
	Finally we take the infimum over the decompositions of $P_j h$ to get the bound \eqref{IIj}, which in turn implies the estimate \eqref{h<0-naba}.

\medskip

\emph{c) The estimates \eqref{BiL-si} and \eqref{qdt_h-d2psi}}. By duality and Sobolev embedding, for any $j<k$ we have
\[ \| P_j h \nab^2 \la_k\|_{l^2 N^\si} \lesssim 2^{\si k} \| P_j h \nab^2 \la_k\|_{L^2L^2} \lesssim  2^{(\si+2) k} \| |D|^{\si_d }P_j h\|_{L^\infty H^{d/2-\si_d}} \| \la_k\|_{L^2L^2},\]
which gives the bound \eqref{qdt_h-d2psi}. We can also obtain the bound \eqref{BiL-si} similarly. This completes the proof of the lemma.	
\end{proof}

We  next prove the remaining  bilinear estimates and trilinear estimates.
\begin{prop}[Nonlinear estimates]      \label{Non-Est}
	Let $ s>\frac{d}{2}$ and $d\geq 2$.
	Assume that $p_k,\ \tp_k$, $s_k$ and $\ts_k$ are admissible frequency envelopes for $\la\in Z^s$, $\la\in Z^\si$, $\SS\in \EE^{s}$ and $\SS\in \EE^{\si}$ respectively. Then we have
	\begin{gather}    \label{cb-B-s}
	\lV S_k(B\la)\rV_{l^2N^{s}}\lesssim s_k\lV \la\rV_{Z^{s}}+p_k\lV B\rV_{L^2 H^{s}} ,\\\label{cb-AA-s}
	\lV S_k(A^2\la)\rV_{l^2N^{s}}\lesssim s_k\lV  A\rV_{Z^{s}}\lV\la\rV_{Z^{s}}+ p_k\lV  A\rV_{Z^{s}}^2,\\
	\label{cb_lam2-psi-s}
	\lV S_k(\lambda^3)\rV_{l^2N^{s}}\lesssim p_k \lV  \lambda\rV_{Z^{s}}^2.
	\end{gather}
	For $-s\leq \si\leq s$ we have
	\begin{gather}    		\label{qdt-gdpsi_hl&hh}
		\lV S_k\nab(h_{\geq k-4}\nab\la)\rV_{l^2N^{\si}}\lesssim  \min\{\ts_k\lV \la\rV_{Z^s},\ \tp_k\lV  h\rV_{Z^{\si_d,s+2}}\},\\ \label{qdt-Ag-dpsi}
		\lV S_k(A_{\geq k-4}\nab\la)\rV_{l^2N^{\si}}\lesssim  \min\{\ts_k\lV \la\rV_{Z^{s}},\ \tp_k\lV A\rV_{Z^{s+1}}\},
	\end{gather}
	and for $-s\leq \si\leq s-\de$ we have
	\begin{gather}     \label{cb-B}
	\lV S_k(B\la)\rV_{l^2N^{\si}}\lesssim \min\{\ts_k\lV \la\rV_{Z^{s}},\ \tp_k\lV B\rV_{L^2 H^{s}}\},\\\label{cb-AA}
	\lV S_k(A^2\la)\rV_{l^2N^{\si}}\lesssim \min\{\ts_k\lV A\rV_{Z^{s}}\lV\la\rV_{Z^{s}},\ \tp_k\lV  A\rV_{Z^{s}}^2\},\\
	\label{cb_lam2-psi}
	\lV S_k(\lambda^3)\rV_{l^2N^{\si}}\lesssim \ts_k \lV  \lambda\rV_{Z^{s}}^2.
	\end{gather}
	If $-s\leq \si\leq s-1$, then
	\begin{gather}\label{qdt-Ag-dpsi_lh}
	\lV S_k(A_{< k-4}\nab\la)\rV_{l^2N^{\si}}\lesssim p_k\lV  A\rV_{Z^{\si+1}}.
	\end{gather}
\end{prop}
\begin{proof}
	
	We first prove \eqref{qdt-gdpsi_hl&hh} and \eqref{qdt-Ag-dpsi}. These two bounds are proved by H\"older's inequality and Bernstein's inequality, here we only prove the first bound in detail. For the high-low case, by duality we have
	\begin{equation*}  
	\begin{aligned}
	\sum_{j\leq k+C}\lV S_k \nab( h_{k}\nab \la_{j})\rV_{l^2N^{\si}}\lesssim &\sum_{j\leq k+C}2^{\si k}\lV S_k \nab( h_{k}\nab \la_{j})\rV_{L^2L^2}\\
	\lesssim &\sum_{j\leq k+C} 2^{(\si+1) k}\lV  h_{k}\rV_{L^2L^2}\lV \nab\la_{j}\rV_{L^{\infty}L^{\infty}}\\
	\lesssim &\sum_{j\leq k+C} 2^{(\si+1) k+(d/2+1)j}\lV  h_k\rV_{L^2L^2}\lV \la_{j}\rV_{L^{\infty}L^2}.
	\end{aligned}
	\end{equation*}
	Then by $-s\leq \si\leq s$ and \eqref{FreEve-relation}, we can bound this by $\min\{\ts_k\lV \la\rV_{Z^s},\ \tp_k\lV  h\rV_{L^2 H^{s+2}}\}$.
	For the high-high case, when $\si+d/2+1>0$ we have
	\begin{align*}
		&\sum_{j>k}\lV S_k \nab( h_{j}\nab \la_{j})\rV_{l^2N^{\si}}\\
		\lesssim &\sum_{j_1= j_2+O(1),j_1>k}2^{(\si+1) k+dk/2}\lV S_k ( h_{j_1}\nab \la_{j_2})\rV_{L^2L^1}\\
		\lesssim & \sum_{j_1= j_2+O(1),j_1>k}2^{(\si+1+d/2+\de)(k-j_1)+(\si+2+d/2+\de)j_1}\lV  h_{j_1}\rV_{L^2L^2}\lV \la_{j_2}\rV_{L^{\infty}L^2}\\
		\lesssim & \min\{\ts_k\lV \la\rV_{Z^s},\tp_k\lV  h\rV_{Z^{1,s+2}}\},
	\end{align*}
	and when $\si+d/2+1\leq 0$ we have
	\begin{align*}
		&\sum_{j_1= j_2+O(1),j_1>k}\lV S_k \nab( h_{j_1}\nab \la_{j_2})\rV_{l^2N^{\si}}\\
		\lesssim & \sum_{j_1= j_2+O(1),j_1>k}2^{(\si+1+d/2-\de)k+(\de+1) j_1}\lV  h_{j_1}\rV_{L^2L^2}\lV \la_{j_2}\rV_{L^{\infty}L^2}\\
		\lesssim & \min\{\ts_k\lV \la\rV_{Z^s},\tp_k\lV  h\rV_{Z^{1,s+2}}\}.
	\end{align*}
	
	Next, we prove the bounds \eqref{cb-B-s}-\eqref{cb_lam2-psi-s} and \eqref{cb-B}-\eqref{cb_lam2-psi}. These are all similar, so we only prove \eqref{cb-B-s} and \eqref{cb-B} in detail. Indeed, by duality we have
	\begin{align*}
	\lV S_k(B\la)\rV_{l^2N^{\si}}\lesssim 2^{\si k}\lV S_k(B\la)\rV_{L^2L^2}.
	\end{align*}
	Then using Littlewood-Paley dichotomy to divide this into \emph{low-high}, \emph{high-low} and \emph{high-high} cases. 	
	For the low-high interactions, by Sobolev embedding we have for $-s\leq \si\leq s$
	\begin{align*}
	2^{\si k}\lV S_k(B_{<k}\la_k)\rV_{L^2L^2}\lesssim & \lV B_{<k}\rV_{L^2L^{\infty}}2^{\si k}\lV \la_k\rV_{L^{\infty}L^2}\lesssim \tp_k\lV B\rV_{L^2H^{s}}.
	\end{align*}
	If $-s\leq \si \leq s-\de$, we use $L^2 H^\si$ for $B_l=S_l B$. Then by $\|B_l\|_{L^2H^\si}\lesssim 2^{\de(k-l)}\ts_k$ we also have 
	\begin{align*}
	    2^{\si k}\lV S_k(B_{<k}\la_k)\rV_{L^2L^2}\lesssim \ts_k \lV\la\rV_{Z^{s}}.
	\end{align*}
	The high-low interactions can be estimated similarly. For the high-high interactions, by Sobolev embedding when $ -d/2-\de\leq \si\leq s$ we have
	\begin{align*}
	2^{\si k}\lV \sum_{l>k} S_k(B_l\la_l)\rV_{L^2L^2}\lesssim &\sum_{l>k} 2^{(\si+d/2+2\de)(k-l)} 2^{(\si+d/2+2\de)l} \lV B_l\rV_{L^2L^2}\lV \la_l\rV_{L^{\infty}L^2}\\
	\lesssim & \min\{\ts_k\lV \la\rV_{Z^s},\tp_k\lV B\rV_{L^2 H^s} \},
	\end{align*}
	and when $-s\leq \si<-d/2-\de$ we have
	\begin{align*}
	&2^{\si k}\lV \sum_{l>k} S_k(B_l\la_l)\rV_{L^2L^2}\lesssim \sum_{l>k} 2^{(\si+d/2)k}  \lV B_l\rV_{L^2L^2}\lV \la_l\rV_{L^{\infty}L^2}\\
	\lesssim & 2^{-\de k} \min\{\|B\|_{L^2H^\si}\|\la\|_{Z^{-\si}},\|B\|_{L^2H^{-\si}}\|\la\|_{Z^{\si}}\} \lesssim  \min\{\ts_k\lV \la\rV_{Z^s},\tp_k\lV B\rV_{L^2 H^s} \},
	\end{align*}
	These imply the bounds \eqref{cb-B-s} and \eqref{cb-B}. 
	
	Finally, we prove the bound \eqref{qdt-Ag-dpsi_lh}. If $d/2-1+\de\leq \si\leq s-1$, by duality and Sobolev embedding, we have
	\begin{align*}
	\lV A_{<k}\nab\la_k\rV_{l^2N^\si}\lesssim 2^{(\si+1) k}\|A_{<k} \|_{L^2L^\infty} \|\la_k\|_{L^\infty L^2}\lesssim p_k \lV  A\rV_{Z^{\si+1}}.
	\end{align*}
	If $\si< d/2-1+\de$, we have
	\begin{align*}
	2^{\si k}\lV A_{<k}\nab\la_k\rV_{L^2L^2}\lesssim &\sum_{0\leq l<k} 2^{(d/2-1-\si+2\delta)(l-k)}\lV  A\rV_{L^2 H^{\si+1}} 2^{(d/2+2\delta)k} \lV \la_k\rV_{L^\infty L^2}\\
	\lesssim &\ p_k\lV A\rV_{Z^{\si+1}}.
	\end{align*}
	Then the bound \eqref{qdt-Ag-dpsi_lh} follows. Hence this completes the proof of the lemma.
\end{proof}

We shall also require the following bounds for commutators.

\begin{prop}[Commutator bounds]\label{Comm-est-Lemma} Let $d\geq 2$ and $s>\frac{d}{2}$. Let $m(D)$
be a multiplier with symbol $m\in S^0$. Assume $ h\in Y^{s+2}$, $ \d_x A\in L^2 H^s$ and $\la_k \in l^2X^s$, frequency localized at frequency $2^k$. If $-s\leq \si \leq s$ then we have
	\begin{gather}          \label{Comm-bd}
	\lV \nab[S_{<k-4}h,m(D)]\nab\la_k\rV_{l^2N^{\si}}\lesssim \min\{ \lV  h\rV_{Y^{\si+2}}\lV \la_k\rV_{l^2X^s},\lV h\rV_{Y^{s+2}}\lV\la_k\rV_{l^2X^{\si}}\},\\
	\label{Com_Ag-dpsi}
	\lV [S_k,A_{<k-4}]\nab\la_k\rV_{l^2N^{\si}}\lesssim \min\{\lV \d_x A\rV_{L^2H^{s}}\lV \la_k\rV_{L^\infty H^\si},\lV \d_x A\rV_{L^2 H^{\si}}\lV \la_k\rV_{L^\infty H^s}\}.
	\end{gather}
\end{prop}
\begin{proof} This is similar to Proposition 5.3 in \cite{HT21}.
	First we estimate \eqref{Comm-bd}. In \cite[Proposition 3.2]{MMT3}, it was shown that 
	\begin{equation*}
	\nab[S_{<k-4}g,m(D)]\nab S_k \la=L(\nab S_{<k-4}g,\nab S_k \la),
	\end{equation*}
	where $L$ is a translation invariant operator satisfying
	\begin{equation*}
	L(f,g)(x)=\int f(x+y)g(x+z)\tilde{m}(y+z)\ dydz,\ \ \ \tilde{m}\in L^1.
	\end{equation*}
	Given this representation, as we are working in translation-invariant spaces, by \eqref{qdt_d h<k dpsi} the bound \eqref{Comm-bd} follows.

	Next, for the bound \eqref{Com_Ag-dpsi}. Since 
	\begin{equation*}
	[S_k,A_{<k}]\nab\la=\int_0^1 \int 2^{kd}\check{\varphi}(2^k y)2^k y \nab A_{<k}(x-sy) 2^{-k}\nab\la_{[k-3,k+3]}(x-y)\ dyds,
	\end{equation*}
	By translation-invariance and the similar argument to \eqref{cb-B}, the bound \eqref{Com_Ag-dpsi} follows. This completes the proof of the lemma.
\end{proof}

\bigskip

\section{Local energy decay and the linearized problem}\label{Sec-LED}

In this section, we consider a linear  Schr\"odinger equation 
\begin{equation}\label{Lin-eq0}
\left\{
\begin{aligned}      
&i\d_t \la+\d_{\al}(g^{\al\be}\d_{\be}\la)+2i A^{\al}\d_{\al}\la
 =F,\\
&\la(0)=\la_0,
\end{aligned}\right.
\end{equation}
and, under suitable assumptions on the coefficients,
we prove that the solution satisfies suitable energy and local energy bounds.

\subsection{The linear paradifferential Schr\"odinger flow}
As an intermediate step, here we  prove energy and local energy bounds for a frequency localized linear paradifferential Schr\"odinger equation
\begin{equation}        \label{LinSch}
	i\d_t\la_k+\d_{\al}(g^{\al\be}_{<k-4}\d_{\be}\la_k)+2iA^{\al}_{<k-4}\d_{\al}\la_k=f_k.
\end{equation}
We begin with the energy estimates, which are fairly standard:

\begin{lemma}[Energy-type estimate]
	Let $d\geq 2$. Assume that $\la_k$ solves the equation \eqref{LinSch} with initial data $\la_k(0)$ in the time interval $[0,1]$. For a fixed $s>\frac{d}{2}$, assume that $\d_x A\in L^2 H^{s}$, $\la_k\in l^2_kX_k$, and $f_k=f_{1k}+f_{2k}$
with $f_{1k}\in N$ and $f_{2k}\in L^1L^2$. Then we have
	\begin{equation}       \label{Eng-Estm}
		\begin{aligned}
			\lV \la_k\rV_{L^{\infty}_tL^2_x}^2\lesssim & \  \lV\la_k(0)\rV_{L^2}^2+\lV  \d_x A\rV_{L^2 H^{s}}\lV \la_k\rV_{X_k}^2	+\lV\la_k\rV_{X_k}\lV f_{1k}\rV_{N_k}
			+\lV \la_k\rV_{L^{\infty}L^2}\lV f_{2k}\rV_{L^1L^2}.
		\end{aligned}
	\end{equation}
\end{lemma}
\begin{proof}
     For the proof, we refer the readers to Lemma 6.1 in \cite{HT21}. Here we just replace the condition $A\in Z^{1,s+1}$ in \cite{HT21} by the assumption $\d_x A\in L^2H^s$. 
\end{proof}

Next, we prove the  main result of this section, namely the local energy estimates
for solutions to \eqref{LinSch}:

\begin{prop} [Local energy decay]     \label{Local-Energy-Decay}
	Let $d\geq 2$. Assume that the coefficients $h=g-I_d$ and $A$ in \eqref{LinSch} satisfy
	\begin{equation}\label{small-ass}
	\lV h\rV_{\bY^{s+2}}+ \lV  A\rV_{Z^{s+1}}\lesssim \ep
	\end{equation}
	for some $s>\frac{d}{2}$ and $\ep>0$ small. Let $\la_k$ be a solution to \eqref{LinSch} which is localized at frequency $2^k$. Then the following estimate holds:
	\begin{equation}      \label{energy-decay}
		\lV\la_k\rV_{l^2_k X_k}\lesssim \lV\la_{0k}\rV_{L^2}+\lV f_k\rV_{l^2_kN_k}
	\end{equation}
\end{prop}

\begin{proof}
	The proof is closely related to that given in \cite{MMT3,MMT4,HT21}. However, here the metric $g=I_d+h$ and magnetic potential $A$ will satisfy some parabolic equations, so 
	we need to modify the assumptions both on $h$ and $A$ to match our main results.
	
	As an intermediate step in the proof, we will establish a local energy decay bound 
in a cube  $Q\in \QQ_l$ with $0\leq l\leq k$:	
\begin{equation}   \label{LED-pre0}
	\begin{aligned}      
	2^{k-l}\lV \la_k\rV_{L^2L^2([0,1]\times Q)}^2\lesssim &\ \lV \la_k\rV_{L^{\infty}L^2}^2+
	\|f_k\|_{N_k} \|\la_k \|_{X_k}  \\ 
	&+(2^{-k}+\lV  A\rV_{Z^{1-\de,s+1}}+\lV  h\rV_{\bY^{s+2}})\lV \la_k\rV_{l^2_kX_k}^2.
	\end{aligned}
	\end{equation}
	
The proof of this bound is based on a positive commutator argument using a well chosen 
multiplier $\MM$. This will be first-order differential operator with smooth coefficients which are localized at frequency $\lesssim 1$. Precisely, we will use a multiplier $\MM$ which is 
a  self-adjoint differential operator having the form 
\begin{equation}           \label{M-def}
		i2^k\MM=a^{\al}(x)\d_{\al}+\d_{\al}a^{\al}(x)
\end{equation}
with uniform bounds on $a$ and its derivatives.

	Before proving \eqref{energy-decay}, we need the following lemma which is used to dismiss the $(g-I)$ contribution to the commutator $[\d_{\al}g^{\al\be}\d_{\be},\MM]$.

	\begin{lemma}\label{Comm_g-I}
		Let $d\geq 2$ and $s>\frac{d}{2}$. Assume that $ h\in \bY^{s+2}$, $A\in Z^{1-\de,s+1}$ and $\la\in l^2_kX_k$, and let $\MM$ be as \eqref{M-def}. Then we have
		\begin{gather}           \label{Comm-M}
		\int_0^1\< [\d_{\al}h^{\al\be}_{\leq k}\d_{\be},\MM]\la_k,\la_k\>dt\lesssim \lV  h\rV_{\bY^{s+2}}\lV \la_k\rV_{l^2_kX_k}^2,\\\label{non-(Ag)dpsi,Mpsi}
		\int_0^1\Re\<A^{\al}_{<k-4}\d_{\al}\la_k,\MM\la_k \>dt\lesssim \lV  A\rV_{Z^{1-\de,s+1}}\lV\la_k\rV_{X_k}^2.
		\end{gather}
	\end{lemma}
	\begin{proof}[Proof of Lemma \ref{Comm_g-I}]
		By \eqref{M-def} and direct computations, we get
		\begin{align*}
		[\d_{\al}h^{\al\be}\d_{\be},\MM]\approx 2^{-k}[\nab(h\nab a+a\nab h)\nab+\nab h\nab^2 a+h\nab^3 a].
		\end{align*} 
		Then it suffices to estimate
		\begin{align*}
		2^{-k} \int_0^1\<(h_{\leq k}\nab a+a\nab h_{\leq k})\nab\la_k,\nab\la_k\>dt+2^{-k}\int_0^1\< (\nab h_{\leq k}\nab^2 a+h_{\leq k}\nab^3 a)\la_k,\la_k\>dt.
		\end{align*}
		The first integral is estimated by \eqref{qdt_d h<k dpsi} and \eqref{h<0-naba}, while the second integral is bounded by Sobolev embeddings. Hence, the bound \eqref{Comm-M} follows.
		
		For the second bound \eqref{non-(Ag)dpsi,Mpsi}, by \eqref{M-def} and integration by parts we rewrite the left-hand side of \eqref{non-(Ag)dpsi,Mpsi} and bound it by
		\begin{align*}
		\Re\int_0^1\<A^{\al}_{<k-4}\d_{\al}\la_k,\MM\la_k \>dt
		\lesssim &2^{-k} \int_0^1\int_{\R^d} |\<\nab\> A_{<k} \la_k \nab \la_k |dx dt\\
		\lesssim &\lV |\nab|^{1-\de} A\rV_{L^2H^{s+\de}}\lV\la_k\rV_{L^\infty L^2}^2.
		\end{align*}
		This implies the bound \eqref{non-(Ag)dpsi,Mpsi}, and hence completes the proof of the lemma.
	\end{proof}
	
	\medskip
	
	Returning to the proof of \eqref{LED-pre0}, for the  self-adjoint multiplier $\MM$ we compute
	\begin{align*}
		\frac{d}{dt}\<\la_k,\MM\la_k\>=&2\Re\<\d_t\la_k,\MM\la_k\>\\
		=&2\Re\< i\d_{\al}(g^{\al\be}_{<k-4}\d_{\be}\la_k)-2A^{\al}_{<k-4}\d_{\al}\la_k-if_k,\MM\la_k \>\\
		=&i\< [-\d_{\al}g^{\al\be}_{<k-4}\d_{\be},\MM]\la_k,\la_k\>
		+2\Re\<-2A^{\al}_{<k-4}\d_{\al}\la_k-if_k,\MM\la_k \>
	\end{align*}
	We then use the multiplier $\MM$ as in \cite{MMT3,MMT4} so that the following three properties hold:
	\begin{enumerate}
	\item[(1)] Boundedness on frequency $2^k$ localized functions,
	\[
	\lV\MM u\rV_{L^2_x}\lesssim \lV u\rV_{L^2_x}.
	\]
	
	\item[(2)] Boundedness in $X$, 
	\[
	\lV\MM u\rV_{X}\lesssim \lV u\rV_{X}.
	\]
	
	\item[(3)] Positive commutator,
	\[
	i\< [-\d_{\al}g^{\al\be}_{<k-4}\d_{\be},\MM]u,u\>\gtrsim 2^{k-l}\lV u\rV^2_{L^2_{t,x}([0,1]\times Q)}-O(2^{-k}+\lV h\rV_{\bY^{s+2}})\lV u\rV_{l^2_kX_k}^2.
	\]
\end{enumerate}
	
	\noindent If these three properties hold for $u=\la_k$, then by \eqref{non-(Ag)dpsi,Mpsi} and \eqref{small-ass} the bound \eqref{LED-pre0} follows.

	We first do this when the Fourier transform of the solution $\la_k$ is restricted to a small angle
	\begin{equation}       \label{xi1-near}
		\text{supp}\ \widehat{\la}_k\subset \{|\xi|\lesssim \xi_1\}.
	\end{equation}
	Without loss of generality due to translation invariance, $Q=\{|x_j|\leq 2^l:j=1,\ldots,d\}$, and we set $m$ to be a smooth, bounded, increasing function such that $m'(s)=\varphi^2(s)$ where $\varphi$ is a Schwartz function localized at frequencies $\lesssim 1$, and $\varphi\approx 1$ for $|s|\leq 1$. We rescale $m$ and set $m_l(s)=m(2^{-l}s)$. Then, we fix
	\[
	\MM=\frac{1}{i2^k} (m_l(x_1)\d_1+\d_1 m_l(x_1)).
	\]
	
	The properties $(1)$ and $(2)$ are immediate due to the frequency localization of $u=\la_k$ and $m_l$ as well as the boundedness of $m_l$. By \eqref{Comm-M} it suffices to consider the property $(3)$ for the operator 
	\[
	-\De=-\d_{\al}g^{\al\be}_{<k-4}\d_{\be}+\d_{\al}h^{\al\be}_{<k-4}\d_{\be}.
	\]
	This yields
	\begin{align*}
		i2^k[-\Delta,\MM]=-2^{-l+2}\d_1 \varphi^2(2^{-l}x_1)\d_1+O(1),
	\end{align*}
	and hence
	\begin{equation*}
		i2^k\<[-\Delta,\MM]\la_k,\la_k\>=2^{-l+2}\lV \varphi(2^{-l}x_1)\d_1\la_k\rV_{L^2L^2}^2+O(\lV\la_k\rV_{L^2L^2}^2)
	\end{equation*}
	Utilizing our assumption \eqref{xi1-near}, it follows that
	\begin{equation*}
		2^{k-l}\lV \varphi(2^{-l}x_1)\la_k\rV_{L^2L^2}^2\lesssim i\< [-\Delta,\MM]\la_k,\la_k\>+2^{-k}O(\lV\la_k\rV_{L^2L^2}^2)
	\end{equation*}
	which yields $(3)$ when combined with \eqref{Comm-M}.
	
	We proceed to reduce the problem to the case when \eqref{xi1-near} holds. We let $\{ \theta_j (\om) \}_{j=1}^d$ be a partition of unity,
	\begin{equation*}
		\sum_{j}\theta_j(\om)=1,\ \ \ \ \om\in\S^{d-1},
	\end{equation*}
	where $\theta_j(\om)$ is supported in a small angle about the $j$-th coordinate axis. Then, we can set $\la_{k,j}=\Theta_{k,j}\la_k$ where
	\begin{equation*}
		\FF\Theta_{k,j}\la=\theta_j(\frac{\xi}{|\xi|})\sum_{k-1\leq l\leq k+1}\varphi_l(\xi)\widehat{\la}(t,\xi).
	\end{equation*}
	We see that 
	\begin{align*}
		&(i\d_t+\d_{\al}g^{\al\be}_{<k-4}\d_{\be})\la_{k,j}+2iA^{\al}_{<k-4}\d_{\al}\la_{k,j}\\
		=&\Theta_{k,j}f_k-\d_{\al}[\Theta_{k,j},g^{\al\be}_{<k-4}]\d_{\be}\la_k-2i[\Theta_{k,j},A^{\al}_{\leq k-4}]\d_{\al}\la_k.
	\end{align*}
	
	By applying $\MM$, suitably adapted to the correct coordinate axis, to $\la_{k,j}$ and summing over $j$, we obtain
	\begin{align*}
		&2^{k-l}\lV \la_k\rV_{L^2L^2([0,1]\times Q)}^2\\
		\lesssim &\ \lV\la_k\rV_{L^{\infty}L^2}^2+\sum_{j=1}^d\int_0^1\<-\Theta_{k,j}f_k,\MM\la_{k,j}\>ds\\
		&\ +\sum_{j=1}^d\int\< [\Theta_{k,j},\d_{\al}g^{\al\be}_{<k-4}\d_{\be}]\la_k+[\Theta_{k,j},2iA^{\al}_{<k-4}]\d_{\al}\la_k,\MM\la_{k,j}\>ds\\
		&\ +(2^{-k}+\lV |D|^{1-\de} A\rV_{L^2 H^{s+\de}}+\lV  h\rV_{\bY^{s+2}})\lV \la_k\rV_{l^2_kX_k}^2\\
		\lesssim &\  \lV\la_k\rV_{L^{\infty}L^2}^2+\lV f_k\rV_{N_k}\lV \la_k\rV_{X_k} 
		+(2^{-k}+\lV |D|^{1-\de} A\rV_{L^2 H^{s+\de}}+\lV  h\rV_{\bY^{s+2}})\lV \la_k\rV_{l^2_kX_k}^2.
	\end{align*}
	The commutator is done via \eqref{Comm-bd} and \eqref{Com_Ag-dpsi}. Then \eqref{LED-pre0} follows.

	Next we use the bound \eqref{LED-pre0} to complete the proof of Proposition \ref{Local-Energy-Decay}. 	Taking the supremum in \eqref{LED-pre0} 
	over $Q\in\QQ_l$ and over $l$, we obtain
	\begin{equation*}
	\begin{aligned}
	2^k\lV \la_k\rV_{X}^2\lesssim &\ \lV \la_k\rV_{L^{\infty}L^2}^2+\lV f_{1k}\rV_{N_k}\lV \la_k\rV_{X_k}+\lV f_{2k}\rV_{L^1L^2}\lV \la_k\rV_{L^{\infty}L^2}\\
	&\ +(2^{-k}+\lV |D|^{1-\de} A\rV_{L^2 H^{s+\de}}+\lV  h\rV_{\bY^{s+2}})\lV \la_k\rV_{l^2_kX_k}^2\\
	\lesssim &\ \lV \la_k\rV_{L^{\infty}L^2}^2+\lV f_{1k}\rV_{N_k}\lV \la_k\rV_{X_k}+\lV f_{2k}\rV_{L^1L^2}^2\\
	&\ +(2^{-k}+\lV |D|^{1-\de} A\rV_{L^2 H^{s+\de}}+\lV  h\rV_{\bY^{s+2}})\lV \la_k\rV_{l^2_kX_k}^2.
	\end{aligned}
	\end{equation*}
	Combined with \eqref{Eng-Estm}, we get
	\begin{equation}          \label{psiXk}
	\begin{aligned}
	\lV \la_k\rV_{X_k}^2\lesssim &\  \lV\la_k(0)\rV_{L^2}^2
	+\lV f_{1k}\rV_{N_k}^2+\lV f_{2k}\rV_{L^1L^2}^2\\
	&\ +(2^{-k}+\lV |D|^{1-\de} A\rV_{L^2 H^{s+\de}}+\lV  h\rV_{\bY^{s+2}})\lV \la_k\rV_{l^2_kX_k}^2.
	\end{aligned}
	\end{equation}

We now finish the proof by incorporating the summation over cubes. We let $\{\chi_Q\}$ denote a partition via functions which are localized to frequencies $\lesssim 1$ which are associated to cubes $Q$ of scale $M2^k$. We also assume that $|\nab^l\chi_Q|\lesssim (2^k M)^{-l}$, $l=1,2$. Thus,
	\begin{align*}
	&(i\d_t+\d_{\al}g^{\al\be}_{<k-4}\d_{\be})\chi_Q \la_k+2iA^{\al}_{<k-4}\d_{\al}\chi_Q \la_k\\
	=&\ \chi_Q f_k+[\d_{\al}g^{\al\be}_{<k-4}\d_{\be},\chi_Q]\la_k+2iA^{\al}_{<k-4}\d_{\al}\chi_Q\cdot \la_k
	\end{align*}
	Applying \eqref{Eng-Estm} to $\chi_Q\la_k$, we obtain
	\begin{align*}
	&\sum_Q \lV \chi_Q\la_k\rV_{L^{\infty}L^2}^2\\
	\lesssim & \sum_Q \lV \chi_Q\la_k(0)\rV_{L^2}^2
	+\lV \d_x A\rV_{L^2 H^{s}}\sum_Q\lV \chi_Q\la_k\rV_{X_k}^2\\
	&+(\sum_Q\lV \chi_Qf_k\rV_{N_k}^2)^{1/2}(\sum_Q\lV \chi_Q\la_k\rV_{X_k}^2)^{1/2}\\
	&+\sum_Q\lV[\d_{\al}g^{\al\be}_{<k-4}\d_{\be},\chi_Q]\la_k+2iA^{\al}_{<k-4}\d_{\al}\chi_Q\cdot \la_k\rV_{L^1L^2}^2.
	\end{align*}
	But by \eqref{small-ass} we have
	\begin{equation}\label{chiQ-comm-1}
	\begin{aligned}
	\sum_Q\lV[\nab g\nab,\chi_Q]\la_k \rV_{L^1L^2}^2 \lesssim &\sum_Q\lV \nab g\cdot\nab \chi_Q\cdot\la_k+g\nab(\nab\chi_Q\cdot\la_k)\rV_{L^1L^2}^2\\
	\lesssim &\ (1+\lV |D|^{\si_d} h\rV_{L^\infty H^{s+1-\si_d}}) M^{-2}\sum_Q \lV\chi_Q\la_k\rV_{L^{\infty}L^2}^2,
	\end{aligned}
	\end{equation}
	and also
	\begin{align}\label{chiQ-comm-2}
	\sum_Q\lV2iA^{\al}_{<k-4}\d_{\al}\chi_Q\cdot \la_k\rV_{L^1L^2}^2\lesssim (1+\lV |D|^{1-\de} A\rV_{L^2 H^{s+\de}}) M^{-2}\sum_Q \lV\chi_Q\la_k\rV_{L^{\infty}L^2}^2.
	\end{align}
	For $M$ sufficiently large, we can bootstrap the commutator terms, and, after a straightforward transition to cubes of scale $2^k$ rather than $M2^k$, we observe that
	\begin{equation}      \label{energy-l^2L^infL^2}
	\begin{aligned}
	\lV \la_k\rV_{l^2_kL^{\infty}L^2}^2
	\lesssim   \lV \la_k(0)\rV_{L^2}^2
	+\lV |D|^{1-\de} A\rV_{L^2 H^{s+\de}}\lV \la_k\rV_{l_k^2X_k}^2
	+\lV f_k\rV_{l^2_kN_k}\lV \la_k\rV_{l^2_kX_k}.
	\end{aligned}  
	\end{equation}
	
	We now apply \eqref{psiXk} to $\chi_Q\la_k$, and then by \eqref{chiQ-comm-1} and \eqref{chiQ-comm-2} we see that
	\begin{align*}
	\sum_Q \lV\chi_Q \la_k\rV_{X_k}^2\lesssim & \lV\la_k(0)\rV_{L^2}^2+\sum_Q\lV\chi_Q f_k\rV_{N_k}^2+M^{-2}\sum_Q\lV \chi_Q\la_k\rV_{X_k}^2\\
	&+(2^{-k}+\lV h\rV_{\bY^{s+2}}+\lV |D|^{1-\de} A\rV_{L^2 H^{s+\de}})\sum_Q\lV \chi_Q\la_k\rV_{l^2_kX_k}^2.
	\end{align*}
	For $M\gg 1$, we have
	\begin{align*}
	M^{-d}\lV \la_k\rV_{l^2_kX_k}^2\lesssim  \lV\la_k(0)\rV_{L^2}^2+\lV f_k\rV_{l^2_kN_k}^2
	+(2^{-k}+\lV h\rV_{\bY^{s+2}}+\lV |D|^{1-\de} A\rV_{L^2 H^{s+\de}})\lV \la_k\rV_{l^2_kX_k}^2.
	\end{align*}
	By \eqref{small-ass}, for $k$ sufficiently large (depending on $M$), we may absorb the the last term in the right-hand side into the left, i.e
    \begin{equation*}
    \lV \la_k\rV_{l^2_kX_k}^2\lesssim \lV\la_k(0)\rV_{L^2}^2+\lV f_k\rV_{l^2_kN_k}^2.
    \end{equation*} 
	On the other hand, for the remaining bounded range of $k$, we have
	\begin{equation*}
	\lV \la\rV_{X_k}\lesssim \lV \la\rV_{L^{\infty}L^2},
	\end{equation*}
	and then \eqref{energy-l^2L^infL^2} and \eqref{small-ass} give 
	\begin{align*}
	\lV\la_k\rV_{l^2_kX_k}^2\lesssim &\ \lV \la_k(0)\rV_{L^2}^2
	+\lV |D|^{1-\de}A\rV_{L^2H^{s+\de}}\lV \la_k\rV_{l_k^2X_k}^2
	+\lV f_k\rV_{l^2_kN_k}\lV \la_k\rV_{l^2_kX_k}\\
	\lesssim &\ \lV \la_k(0)\rV_{L^2}^2
	+\lV f_k\rV_{l^2_kN_k}^2,
	\end{align*}
	which finishes the proof of \eqref{energy-decay}. 
\end{proof}

\subsection{ The full linear problem}
Here we use the bounds for the paradifferential equation in the previous subsection 
in order to prove similar bounds for the full equation \eqref{Lin-eq0}:

\begin{prop}[Well-posedness]   \label{p:well-posed}
	Let $s>\frac{d}{2}$, $d\geq 2$ and $h=g-I_d$. Assume that the metric $g$ and the magnetic potential $A$  satisfy
	\begin{equation*}
	\lV h\rV_{\bY^{s+2}},\ \lV |D|^{1-\de} A\rV_{L^2 H^{s+\de}}\ll 1.
	\end{equation*}
Then the equation \eqref{Lin-eq0} is well-posed for initial data $\la_0\in H^{\si}$ with $-s\leq \si\leq s$, and we have the estimate 
	\begin{equation}\label{well-posed-s0}
	\lV \la\rV_{l^2X^{\si}}\lesssim \lV \la_0\rV_{H^{\si}}+\lV F\rV_{l^2N^{\si}}.
	\end{equation}
\end{prop}
\begin{proof}
    The well-posedness follows in a standard fashion from a similar energy estimate for the adjoint equation. Since the adjoint equation has a similar form, with similar bounds on the coefficients, such an estimate follows directly from \eqref{well-posed-s0}. Thus, we now focus on the proof of the bound \eqref{well-posed-s0}. For $\la$ solving \eqref{Lin-eq0}, we see that $\la_k$ solves
	\begin{equation*}
	\left\{
	\begin{aligned}      
	&i\d_t \la_k+\d_{\al}(g^{\al\be}_{<k-4}\d_{\be}\la_k)+2iA^{\al}_{<k-4}\d_{\al}\la_k=F_k+H_k,\\
	&\la_k(0)=\la_{0k},
	\end{aligned}\right.
	\end{equation*}
	where $F_k:=S_k F$ and
	\begin{align*}
	H_k:=&-S_k\d_{\al}(g^{\al\be}_{\geq k-4}\d_{\be}\la)-\d_{\al}[S_k,g^{\al\be}_{< k-4}]\d_{\be}\la-2i[S_k,A^{\al}_{<k-4}]\d_{\al}\la\\
	&-2iS_k( A^{\al}_{\geq k-4}\d_{\al}\la).
	\end{align*}
	If we apply Proposition \ref{Local-Energy-Decay} to each of these equations, we see that
	\begin{equation*}
	\lV\la_k\rV_{l^2X^{\si}}^2\lesssim \lV\la_{0k}\rV_{H^{\si}}^2+\lV F_k\rV_{l^2N^{\si}}^2+\lV H_k\rV_{l^2N^{\si}}^2.
	\end{equation*}
	
	We claim that 
	\begin{gather}  \label{Hk}
	\sum_{k}\lV H_k\rV_{l^2N^{\si}}^2\lesssim (\lV  h\rV_{\bY^{s+2}}+\lV  \d_x A\rV_{L^2 H^{s}})^2\lV \la\rV_{l^2X^{\si}}^2,\ \text{for }-s\leq \si\leq s.
	\end{gather}
	Indeed, the bound for the terms in $H_k$ follows from \eqref{qdt-gdpsi_hl&hh}, \eqref{Comm-bd}, \eqref{Com_Ag-dpsi} and \eqref{qdt-Ag-dpsi}, respectively. Then by the above two bounds and the smallness of $h$ and $A$, we obtain the estimate \eqref{well-posed-s0}.
\end{proof}

\subsection{The linearized problem.} Here we consider the linearized equation:
\begin{equation}\label{Lin-eq}
\left\{
\begin{aligned}      
&i\d_t \La+\d_{\al}(g^{\al\be}\d_{\be}\La)+2iA^{\al}\d_{\al}\La=F+G,\\
&\La(0)=\La_0,
\end{aligned}\right.
\end{equation}
where
\begin{equation*}
G=-\nab(\mathcal{G}\nab\la)-2i\mathcal{A}^{\al}\d_{\al} \la,
\end{equation*}
and we prove the following.
\begin{prop}    \label{well-posedness-lemma}
	Let $s>\frac{d}{2}$, $\frac{d}{2}-2<\si\leq s-2$, $d\geq 2$ and $h=g-I_d\in \bY^{s+2}$, assume that $\La$ is a solution of \eqref{Lin-eq}, the metric $g$ and $A$ satisfy
	\begin{equation*}
	\lV h\rV_{\bY^{s+2}},\ \lV  |D|^{1-\de}A\rV_{L^2 H^{s+\de}}\ll 1.
	\end{equation*}
	Then we have the estimate
	\begin{equation}         \label{well-posed-s-1}
	\lV \La\rV_{l^2 X^{\si}}\lesssim \lV \La_0\rV_{H^{\si}}+\lV F\rV_{l^2N^{\si}}+(\lV\mathcal{G}\rV_{Z^{\si_d,\si+2}}+\lV\mathcal{A}\rV_{Z^{\de_d,\si+1}})\lV\la\rV_{l^2X^{s}}.
	\end{equation}
\end{prop}

\begin{proof}
	For $\La$ solving \eqref{Lin-eq}, we see that $\La_k$ solves
	\begin{equation*}
	\left\{
	\begin{aligned}      
	&i\d_t \La_k+\d_{\al}(g^{\al\be}_{<k-4}\d_{\be}\La_k)+2iA^{\al}_{<k-4}\d_{\al}\La_k=F_k+G_k+H_k,\\
	&\La_k(0)=\La_{0k},
	\end{aligned}\right.
	\end{equation*}
	where
	\begin{gather*}
	G_k=-S_k(\nab(\mathcal{G}\nab\la)-2i\mathcal{A}^{\al}\d_{\al} \la),
	\end{gather*}
	\begin{align*}
	H_k=&-S_k\d_{\al}(g^{\al\be}_{\geq k-4}\d_{\be}\La)-\d_{\al}[S_k,g^{\al\be}_{< k-4}]\d_{\be}\La-2i[S_k,A^{\al}_{<k-4}]\d_{\al}\La\\
	&-2iS_k(A^{\al}_{\geq k-4}\d_{\al}\La).
	\end{align*}
	The proof of \eqref{well-posed-s-1} is similar to that of \eqref{Hk}. Here it suffices to prove
	\begin{gather*}
	\sum_{k}\lV G_k\rV_{l^2N^{\si}}^2\lesssim \lV  \mathcal{G}\rV_{Z^{\si_d,\si+2}}^2\lV \la\rV_{l^2X^{s}}^2+\lV  \mathcal{A}\rV_{ Z^{\de_d,\si+1}}^2\lV \la\rV_{l^2X^{s}}^2.
	\end{gather*}
	Indeed, the bound for the terms in $G_k$ follows from \eqref{qdt-gdpsi_hl&hh}, \eqref{qdt_h-d2psi}, \eqref{qdt-Ag-dpsi} and \eqref{qdt-Ag-dpsi_lh}. 
	This completes the proof of the Lemma.
\end{proof}

\bigskip

\section{Well-posedness in the good gauge}
\label{Sec-LWP}

In this section we use the parabolic results in Section~\ref{Sec-Para}, the multilinear estimates in Section~\ref{Sec-mutilinear} and the linear local energy decay bounds in Section~\ref{Sec-LED}
in order to prove the good gauge formulation of our main result,  namely Theorem \ref{LWP-MSS-thm}.

\subsection{The iteration scheme: uniform bounds}
Here we seek to construct solutions to \eqref{mdf-Shr-sys-2} iteratively, based on the scheme
\begin{equation}         \label{system-iteration}
	\left\{\begin{aligned}	
		&i\d_t\la^{(n+1)}+\d_{\al}(g^{(n)\al\be}\d_{\be}\la^{(n+1)})+2iA^{(n)\al}\d_{\al}\la^{(n+1)}=F^{(n)},\\
		&\la^{(n+1)}(0)=\la_0,
	\end{aligned}\right.	
\end{equation}
with the trivial initialization
\begin{equation*}
	\la^{(0)}=0,
\end{equation*}
where the nonlinearities $F^{(n)}$ are the following $F$ with $(\la,h,A)=(\la^{(n)},h^{(n)},A^{(n)})$
\begin{equation}\label{Non-itera}
\begin{aligned}
F
=&\ \d_\mu (g^{\mu\nu}\d_\nu \la_{\al\be})-\nab^\si \nab_\si \la_{\al\be}+iV^\si\nab_\si \la_{\al\be}-i\nab_\si A^\si \la_{\al\be}+i\la^{\ga}_{\al}\nab_{\be} V_{\ga}
    +i\la_\be^\ga\nab_\al V_\ga\\
    &\ +(B+A_\si A^\si-V_\si A^\si)\la_{\al\be}
    +\psi\Re(\la_{\al\de}\bar{\la}^\de_\be)
    -R_{\al\si\be\de}\la^{\si\de}
    -\la_{\al\mu}\bar{\la}^\mu_\si\la^\si_\be,
\end{aligned}
\end{equation}
and $\SS^{(n)}=(h^{(n)},A^{(n)})$ are the solutions of parabolic system \eqref{par-syst} with $\la=\la^{(n)}$ and initial data 
\begin{equation}   \label{initial-nstep}
    h^{(n)}(0,x)=h_0(x), \quad A^{(n)}(0,x)=A_0(x).
\end{equation}

We assume that $(\la_0,h_0)$ is small in $H^s\times \bY^{s+2}$. Due to the above trivial initialization for $\la^{(0)}$, we also 
inductively assume that
\begin{equation}      \label{Ass-itera}  
\lV \la^{(n)}\rV_{l^2 X^s}\leq C\lV\la_0\rV_{H^s},
\end{equation}
where $C$ is a large constant.

Applying the parabolic estimates \eqref{para-bd0} to \eqref{par-syst} with $\la=\la^{(n)}$ and initial data \eqref{initial-nstep} at each step, we obtain
\begin{align}       \label{BD-SS^n}
\lV \SS^{(n)}\rV_{\bEE^{s}}\lesssim \|(h_0,A_0)\|_{\bEE_0^s} +  \lV \la^{(n)}\rV_{l^2Z^s}\lesssim \|(h_0,A_0)\|_{\bEE_0^s}+\|\la_0\|_{H^s}\lesssim \ep_0.
\end{align}

In order to estimate $\la^{(n+1)}$, we bound the nonlinear terms in $F^{(n)}$ first. In the computations we would omit the superscript $(n)$. Precisely, for the first three terms in \eqref{Non-itera}, by covariant derivatives \eqref{co_d} and $V^\ga =g^{\al\be}\Ga^\ga_{\al\be}$ we have the form
\begin{align*}
\d_\mu (g^{\mu\nu}\d_\nu \la_{\al\be})-\nab^\si \nab_\si \la_{\al\be}+iV^\si \nab_\si \la_{\al\be}\approx \nab h \nab\la+\nab h\nab h \la.
\end{align*}
Then the first term $\nab h\nab \la$ is estimated using \eqref{qdt_d h<k dpsi} and \eqref{qdt-gdpsi_hl&hh}, the second term $\nab h\nab h\la$ is estimated using \eqref{cb-AA-s} with its $A=\nab h$. We obtain
\begin{align*}
\|\nab h \nab\la+\nab h\nab h \la\|_{l^2 N^s}\lesssim \| h\|_{\bY^{s+2}}\|\la\|_{l^2 X^s}+\| h\|_{Z^{s+1}}^2\|\la\|_{Z^s}.
\end{align*}
For the fourth to seventh terms in \eqref{Non-itera}, we have the expression
\begin{align*}
&-i\nab_\si A^\si \la_{\al\be}+i\la^{\ga}_{\al}\nab_{\be} V_{\ga}
+i\la_\be^\ga\nab_\al V_\ga
+(B+A_\si A^\si-V_\si A^\si)\la_{\al\be}\\
\approx &\ (\nab^2 h+\nab A)\la+(\nab h+A)^2 \la.
\end{align*}
Then these two terms are estimated using \eqref{cb-B-s} and \eqref{cb-AA-s} respectively. We obtain 
\begin{align*}
\|(\nab^2 h+\nab A)\la+(\nab h+A)^2 \la\|_{l^2N^s}\lesssim (1+\| \SS\|_{\EE^s})\| \SS\|_{\EE^s}\|\la\|_{Z^s}.
\end{align*}
For the last three terms in \eqref{Non-itera}, by \eqref{R-la} we have
\begin{align*}
\psi\Re(\la_{\al\de}\bar{\la}^\de_\be)
-R_{\al\si\be\de}\la^{\si\de}
-\la_{\al\mu}\bar{\la}^\mu_\si\la^\si_\be\approx \la^3.
\end{align*}
Using \eqref{cb_lam2-psi} we obtain 
\begin{align*}
\|\la^3\|_{l^2N^s}\lesssim \|\la\|_{Z^s}^3.
\end{align*}
Hence, by the above estimates, \eqref{BD-SS^n} and \eqref{Ass-itera} we bound the $F^{(n)}$ by 
\begin{align*}
\|F^{(n)}\|_{l^2N^s}\lesssim \ (1+\| \SS^{(n)}\|_{\bEE^s})\| \SS^{(n)}\|_{\bEE^s}\|\la^{(n)}\|_{l^2X^s}
\lesssim  \ \ep_0 \|\la_0\|_{H^s}.
\end{align*}
Now applying at each step the local energy bound \eqref{well-posed-s0} with $\si=s$ we obtain the estimate
\begin{align}      \label{Unif-bound}
\lV\la^{(n+1)}\rV_{l^2 X^s}\lesssim  \lV\la_0\rV_{H^s}+\lV F^{(n)}\rV_{l^2N^s}
\lesssim \lV\la_0\rV_{H^s}+C\ep_0 \|\la_0\|_{H^s}
\leq C\lV\la_0\rV_{H^s},
\end{align}
which closes our induction.

\subsection{The iteration scheme: weak convergence.}
Here we prove that our iteration scheme converges in the weaker $H^{s-2}$ topology. We denote the differences by
\begin{gather*}
\La^{(n+1)}=\la^{(n+1)}-\la^{(n)},\\ \dSS^{(n+1)}=(\GG^{(n+1)},\AA^{(n+1)},\BB^{(n+1)})=\SS^{(n+1)}-\SS^{(n)}
\end{gather*}
Then from \eqref{system-iteration} we obtain the system
\begin{equation*}              
\left\{\begin{aligned}
&i\d_t \La^{(n+1)}+\d_\al( g^{(n)\al\be}\d_{\be}\La^{(n+1)})+2iA^{(n)\al}\d_{\al}\La^{(n+1)}=F^{(n)}-F^{(n-1)}+G^{(n)},\\
&\La^{(n+1)}(0,x)=0,
\end{aligned}\right.
\end{equation*}
where the nonlinearities $G^{(n)}$ have the form
\begin{align*}
G^{(n)}=&-\d_\al(\GG^{(n)}\d_{\be}\la^{(n)})-2i\AA^{(n)\al}\d_{\al}\la^{(n)},
\end{align*}

By \eqref{para-delta0} we obtain 
\begin{equation}      \label{dSSn}
    \lV \dSS^{(n)}\rV_{\EE^{s-2}}\lesssim \lV \La^{(n)}\rV_{l^2 X^{s-2}}.
\end{equation}
Applying \eqref{well-posed-s-1} with $\si=s-2$ for the $\La^{(n+1)}$ equation we have
\begin{align*}
\lV \La^{(n+1)}\rV_{l^2X^{s-2}}\lesssim  \lV F^{(n)}-F^{(n-1)}\rV_{l^2N^{s-2}}+(\lV\GG^{(n)}\rV_{Z^{\si_d,s}}+\lV\AA^{(n)}\rV_{ Z^{\de_d,s-1}})\lV\la^{(n)}\rV_{l^2X^{s}}.
\end{align*}
For the nonlinear terms $F^{(n)}-F^{(n-1)}$, using \eqref{BiL-si}, \eqref{qdt-gdpsi_hl&hh}, \eqref{cb-AA}, \eqref{cb-B} and \eqref{cb_lam2-psi}  we have
\begin{align*}
\lV F^{(n)}-F^{(n-1)}\rV_{l^2N^{s-2}}
\lesssim & \ (1+\|(\SS^{(n)},\SS^{(n-1)})\|_{\EE^s})^N\big(\|\dSS^{(n)}\|_{\EE^{s-2}}\|(\la^{(n)},\la^{(n-1)})\|_{l^2X^s}\\
&\ +\|(\SS^{(n)},\SS^{(n-1)})\|_{\EE^s}\|\La^{(n)}\|_{l^2X^{s-2}}\big).
\end{align*}
Then by \eqref{dSSn} and the uniform bounds \eqref{BD-SS^n}, \eqref{Unif-bound} we bound the right hand side above by
\begin{align*}
\lV \La^{(n+1)}\rV_{l^2X^{s-2}}
\lesssim &\ (1+\| \SS_0\|_{\bEE_0^s}+\|\la_0\|_{H^s})^N \\
&\quad\cdot\big[ \|\La^{(n)}\|_{l^2X^{s-2}} \|\la_0\|_{H^s}+(\| \SS_0\|_{\bEE_0^s}+\|\la_0\|_{H^s})\|\La^{(n)}\|_{l^2 X^{s-2}}  \big] \\
\ll &\ \lV \La^{(n)}\rV_{l^2X^{s-2}}.
\end{align*}
This implies that our iterations $\la^{(n)}$ converge in $l^2X^{s-2}$ to some function $\la$.
Furthermore, by the uniform bound \eqref{Unif-bound} it follows that
\begin{equation}\label{Energy-bound}
\lV\la\rV_{l^2X^s}\lesssim \lV\la_0\rV_{H^s}.
\end{equation}
Interpolating, it follows that $\la^{(n)}$ converges to $\la$ in $l^2X^{s-\epsilon}$ 
for all $\epsilon > 0$. This allows us to conclude that the auxiliary functions 
$\SS^{(n)}$ associated to $\la^{(n)}$ converge to the functions $\SS$ associated to $\la$,
and also to pass to the limit and conclude that $\la$ solves the (SMCF) equation \eqref{mdf-Shr-sys-2}. Moreover, we have the bound for $\SS$
\begin{align}   \label{Ene-SS}
\lV \SS\rV_{\bEE^{s}}\lesssim \| \SS_0\|_{\bY_0^{s+2}}+\|\la_0\|_{H^s}.
\end{align}
Thus we have established the existence part of our main theorem.

\subsection{Uniqueness via weak Lipschitz dependence.}
Consider the difference of two solutions 
\[
(\La,\dSS)=(\la^{(1)}-\la^{(2)},\SS^{(1)}-\SS^{(2)}).
\]
The $\La$ solves an equation of this form
\begin{equation*}              
\left\{\begin{aligned}
&i\d_t \La+\d_{\al}(g^{(1)\al\be}\d_{\be}\La)+2iA^{(1)\al}\d_{\al}\La=F^{(1)}-F^{(2)}+G,\\
&\La(0,x)=\la^{(1)}_0(x)-\la^{(2)}_0(x),
\end{aligned}\right.
\end{equation*}
where the nonlinearity $G$ is 
\begin{align*}
G=&-\d_{\al}(\GG\d_{\be}\la^{(2)})-2i\AA^{\al}\d_{\al}\la^{(2)}.
\end{align*}

By \eqref{para-delta0} we have
\begin{equation*}
    \lV \dSS\rV_{\EE^{s-2}}\lesssim \|\dSS_0\|_{\HH^{s-2}}+ \lV \La\rV_{l^2 X^{s-2}}.
\end{equation*}
Applying \eqref{well-posed-s-1} with $\si=s-2$ to the $\La$ equation, we obtain the estimate
\begin{align*}
\lV \La\rV_{l^2 X^{s-2}}\lesssim &\ \lV\La_0\rV_{H^{s-2}}+\lV F^{(1)}-F^{(2)}\rV_{l^2N^{s-2}}+(\lV\GG\rV_{Z^{\si_d,s}}+\lV\AA\rV_{Z^{\de_d,s-1}})\lV\la^{(2)}\rV_{l^2X^{s}}\\
\lesssim &\  \lV\La_0\rV_{H^{s-2}}+C\lV(\la^{(1)}_0,\la^{(2)}_0)\rV_{H^{s}}\lV (\La,\dSS)\rV_{l^2 X^{s-2}\times\EE^{s-2}}.
\end{align*}
Then, by the above bound for $\dSS$, we further have
\begin{align*}
\lV \La\rV_{l^2X^{s-2}}\lesssim \lV\La_0\rV_{H^{s-2}}+C\lV(\la^{(1)}_0,\la^{(2)}_0)\rV_{H^{s}}(\|\dSS_0\|_{\HH^{s-2}}+ \lV \La\rV_{l^2 X^{s-2}})
\end{align*}
Since the initial data $\la^{(1)}_0$ and $\la^{(2)}_0$ are sufficiently small, we obtain 
\begin{equation}      \label{Uniq-bound}
\lV \La\rV_{l^2X^{s-2}}\lesssim \lV\La_0\rV_{H^{s-2}}+\|\dSS_0\|_{\HH^{s-2}}.
\end{equation}
This gives the weak Lipschitz dependence, as well as the uniqueness of solutions for \eqref{mdf-Shr-sys-2}.

\subsection{Frequency envelope bounds}\label{Freq-envelope-section}
Here we prove a stronger frequency envelope version of estimate \eqref{Energy-bound}.
\begin{prop}\label{Freq-envelope-bounds}
	Let $\la\in l^2X^s$, $\SS\in \bEE^{s}$ be small data solution to \eqref{mdf-Shr-sys-2}-\eqref{par-syst}, which satisfies \eqref{Energy-bound} and \eqref{Ene-SS}. Let $\{p_{0k}\}$, $\{s_{0k}\}$ be admissible frequency envelopes for the initial data $\la_0\in H^s$ and $\SS_0\in \bEE_0^{s}$. Then $\{p_{0k}+s_{0k}\}$ is also frequency envelope for $(\la,\SS)$ in $l^2X^s\times \bEE^{s}$.
\end{prop}
\begin{proof}
	Let $p_k$ and $s_k$ be the admissible frequency envelopes for solution $(\la,\SS)\in l^2X^s\times \bEE^s$. Applying $S_k$ to the modified Schr\"{o}dinger equation in \eqref{mdf-Shr-sys-2}, we obtain the paradifferential equation
	\begin{equation*}         
	\left\{\begin{aligned}
	&i\d_t\la_k+\d_{\al}(g_{<k-4}^{\al\be}\d_{\be}\la_k)+2iA_{<k-4}^{\al}\d_{\al}\la_k=F_k+J_k,\\
	&\la(0,x)=\la_0(x),
	\end{aligned}\right.
	\end{equation*}
	where 
	\begin{align*}
	J_k=&-S_k\d_{\al}(g^{\al\be}_{\geq k-4}\d_{\be}\la)-[S_k,\d_{\al}g^{\al\be}_{< k-4}\d_{\be}]\la\\
	&-2i[S_k,A^{\al}_{<k-4}]\d_{\al}\la
	-2iS_k[ A^{\al}_{\geq k-4}\d_{\al}\la_k],
	\end{align*}
	and $\SS=(h,A)$ is the solution to the parabolic system \eqref{par-syst}.
	We estimate $\la_k=S_k\la$ using Proposition \ref{p:well-posed},
	\begin{align*}
		\lV \la_k\rV_{l^2X^s}\lesssim& \ p_{0k}+\| F_k  \|_{l^2N^s}+\| J_k  \|_{l^2N^s}.
	\end{align*}
	By Proposition \ref{Non-Est}, Lemma \ref{bilinear-est} and Lemma \ref{Comm-est-Lemma} we bound the nonlinear terms by
	\begin{equation*}    \label{freq-envelope-psi}
	\begin{aligned}
	\| F_k  \|_{l^2N^s}+\|J_k\|_{l^2N^s}\lesssim& \ (1+\|\SS\|_{\bEE^s}+\|\la\|_{l^2X^s})^N(\|\SS\|_{\bEE^s}p_k+s_k\|\la\|_{l^2X^s}).
	\end{aligned}
	\end{equation*}
	Then by \eqref{Ene-SS}, \eqref{Energy-bound}, \eqref{para-envelope} and the smallness of initial data we obtain
	\begin{align*}
	\|\la_k\|_{l^2X^s}\lesssim p_{0k}+\ep p_k+\ep(s_{0k}+p_k)\lesssim p_{0k}+s_{0k}+\ep p_k.
	\end{align*}
	For metric $g=I_d+h$, by \eqref{para-env0} we also have
	\begin{align*}
	\| \SS_k\|_{\bEE^{s}}\lesssim s_{0k}+\ep p_k.
	\end{align*}
	From the definition of frequency envelope \eqref{Freq-envelope}, these two bounds imply
	\begin{align*}
	p_k+s_k\lesssim p_{0k}+s_{0k}.
	\end{align*}
	and conclude the proof.
\end{proof}

\subsection{Continuous dependence on the initial data}
Here we show that the map $(\la_0,\SS_0)\rightarrow(\la,\SS)$ is continuous from $H^s\times \bEE_0^{s}$ into $l^2X^s\times\bEE^s$. By \eqref{para-s}, it suffices to prove $(\la_0,\SS_0)\rightarrow \la$ is continuous from $H^s\times \bEE_0^{s}$ to $l^2 X^s$.

Suppose that $(\la_0^{(n)},\SS_0^{(n)})\rightarrow (\la_0,\SS_0)$ in $H^s\times \bEE_0^{s}$. Denote by $(p_{0k}^{(n)},s_{0k}^{(n)})$, respectively $(p_{0k},s_{0k})$ the frequency envelopes associated to $(\la_0^{(n)},\SS_0^{(n)})$, respectively $(\la_0,\SS_0)$, given by \eqref{Freq-envelope}. If $(\la_0^{(n)},\SS_0^{(n)})\rightarrow (\la_0,\SS_0)$ in $H^s\times \bEE_0^{s}$ then $(p_{0k}^{(n)},s_{0k}^{(n)})\rightarrow (p_{0k},s_{0k})$ in $l^2$. Then for each $\ep>0$ we can find some $N_{\ep}$ so that
\[
\lV p_{0,>N_{\ep}}^{(n)}\rV_{l^2}+\lV s_{0,>N_{\ep}}^{(n)}\rV_{l^2}\leq \ep,\ \text{for all }n.
\]
By Proposition \ref{Freq-envelope-bounds} we obtain that
\begin{equation}  \label{high-freq-small}
\lV \la_{>N_{\ep}}^{(n)}\rV_{l^2X^s}\leq \ep,\ \text{for all }n.
\end{equation}
To compare $\la^{(n)}$ with $\la$ we use \eqref{Uniq-bound} for low frequencies and \eqref{high-freq-small} for the high frequencies,
\begin{align*}
\lV \la^{(n)}-\la\rV_{l^2X^s}\lesssim&\ \lV S_{<N_{\ep}}(\la^{(n)}-\la)\rV_{l^2X^s}+\lV S_{>N_{\ep}}\la^{(n)}\rV_{l^2X^s}+\lV S_{>N_{\ep}}\la\rV_{l^2X^s}\\
\lesssim &\  2^{2N_{\ep}}\lV S_{<N_{\ep}}(\la^{(n)}-\la)\rV_{l^2X^{s-2}}+2\ep\\
\lesssim &\ 2^{2N_{\ep}}(\lV S_{<N_{\ep}}(\la^{(n)}_0-\la_0)\rV_{H^{s-2}}+\|S_{<N_\ep}(\SS^{(n)}_0-\SS_0)\|_{\HH^{s-2}})+2\ep.
\end{align*}	
Letting $n\rightarrow\infty$ we obtain
\[
\limsup_{n\rightarrow\infty}\lV \la^{(n)}-\la\rV_{l^2X^s}\lesssim \ep.
\]	
Letting $\ep\rightarrow 0$ we obtain
\begin{equation*}   
\lim_{n\rightarrow 0}\lV \la^{(n)}-\la\rV_{l^2X^s}=0,
\end{equation*}
which completes the desired result.

\subsection{Higher regularity}
Here we prove that the solution $(\la,\SS)$ satisfies the bound
\begin{equation}   \label{HiReg}
\lV (\la,\SS)\rV_{l^2X^{\si}\times \bEE^{\si}}\lesssim \lV\la_0\rV_{H^{\si}}+\lV \SS_0\rV_{\bEE_0^{\si}},\quad\si\geq s,
\end{equation}
whenever the right hand side is finite. 

The proof of \eqref{HiReg} is similar to that in \cite[Section 7.6]{HT21}. Here we simply repeat this process. Differentiating the original Schr\"odinger equation \eqref{mdf-Shr-sys-2}, and then using Proposition \ref{p:well-posed}, Lemma \ref{bilinear-est} and Proposition \ref{Non-Est} we easily obtain
\begin{align*}
\lV \nab\la\rV_{l^2 X^s}
\lesssim \lV \nab\la_0\rV_{H^s}+\lV (\nab\la,\nab\SS)\rV_{l^2X^s\times\bEE^s}\lV (\la,\SS)\rV_{l^2X^s\times\bEE^s}(1+\lV (\la,\SS)\rV_{l^2X^s\times\bEE^s})^N.
\end{align*}
For the parabolic equations, by \eqref{para-s} we obtain
\begin{align*}
\lV \nab \SS\rV_{\bEE^s}\lesssim \|\nab \SS_0\|_{\bEE_0^{s}} + \|\la\|_{l^2X^s} \lV \nab\la\rV_{l^2X^s}.
\end{align*}
Hence, by \eqref{Energy-bound} and \eqref{Ene-SS}, these imply \eqref{HiReg} with $\si=s+1$. Inductively, we can further obtain \eqref{HiReg} for any $\si\geq s$.

\subsection{The compatibilities conditions}  
As part of our derivation of the (SMCF) equations \eqref{mdf-Shr-sys-2} for the second fundamental form $\la$ in  the good gauge, coupled with the parabolic system \eqref{par-syst},
we have seen that the compatibility conditions are described by the equations 
\eqref{Ric}, \eqref{R-la}, \eqref{la-commu}, \eqref{cpt-AiAj-2}, \eqref{heat-gauge} and \eqref{Cpt-A&B}.  However, our proof of the well-posedness
result for the Schr\"odinger evolution \eqref{mdf-Shr-sys-2} does not apriori guarantee
that these constraints hold. Here we rectify this 
omission:
\begin{lemma}[Constraint conditions] 
	Assume that $\la \in C[0,T;H^s]$ solves the SMCF equation \eqref{mdf-Shr-sys-2}
	coupled with the parabolic system \eqref{par-syst}. Then the relations \eqref{Ric}, \eqref{R-la}, \eqref{la-commu}, \eqref{cpt-AiAj-2}, \eqref{heat-gauge} and \eqref{Cpt-A&B} hold.
\end{lemma}
\begin{proof}
To shorten the notations, we define
\begin{align*}
&T^1_{\al\be}=\Ric_{\al\be}-\tRic_{\al\be},\ \quad \qquad   \tRic_{\al\be}:=\Re(\lambda_{\al\be}\bar{\psi}-\la_{\al\si}\bar{\la}^\si_{\ \be}),\\
&T^2_{\si\ga\al\be}=R_{\si\ga\al\be}-\tR_{\si\ga\al\be},\qquad \tR_{\si\ga\al\be}:= \Re(\lambda_{\ga\be}\bar{\la}_{\si\al}-\la_{\ga\al}\bar{\la}_{\si\be}), \\
&T^3_{\al\be,\ga}=\nab^A_\al \la_{\be\ga}-\nab^A_\be \la_{\al\ga},\\
&T^4_{\al\be}=\bmF_{\al\be}-\tbmF_{\al\be},\qquad  \bmF_{\al\be}:=\nab_\al A_\be-\nab_\be A_\al,\ \tbmF_{\al\be}:=\Im(\la_\al^\ga \bar{\la}_{\ga\be}),\\
&T^5_{\al}= \bmF_{0\al} - \tbmF_{0\al},\ \qquad \bmF_{0\al}:=\d_t A_\al-\nab_\al B,\ \tbmF_{0\al}:= \Re(\la_{\al}^{\ga}\bar{\d}^A_{\ga}\bar{\psi})+\Im (\la^\ga_\al \bar{\la}_{\ga\si})V^\si.
\end{align*}
Here $T^3$ and $T^4$ are antisymmetric, $T^1$ is symmetric and $T^2$ inherits all the linear symmetries
of the curvature tensor.

Our goal is to show that all these functions vanish,
knowing that they vanish at the initial time. 
We will prove this by showing that
they solve a coupled linear homogeneous evolution system of the form
\begin{equation*}
\left\{\begin{aligned}
&(\d_t-\De_g)T^{1,\be}_{\al}=\la^2 T^4+T^1\nab V+V\nab T^1+T^3\nab\la +\la\nab T^3,\\
&\nab_\delta T^2_{\si\ga\al\be} + \nab_{\si} T^2_{\ga\de\al\be} 
+ \nab_{\ga} T^2_{\de\si\al\be} = T^1\la,\\
&\nab^{\si}T^2_{\si\ga\al\be} = \nabla_{\al} T^1_{\ga\be} - \nabla_{\be} T^1_{\ga\al}
+  T^1\la,\\
&\begin{aligned}
(i\d^B_t -\De^A_g)T^3_{\al\be,\ga}=&\ \la T^5+T^3(\nab V+\la^2+R)+(\nab^A\la+\la V)(T^1+T^2+T^4)\\
&\ +\la\nab (T^2+T^4)+V\nab T^3
\end{aligned}\\
&(\d_t-\De_g)T^4_{\al\be}=\Ric T^4+\nab^A\la T^3+V\la T^3,\\
&T^5_\al=\nab^\si T^4_{\si\al}+T^1_{\al\de}A^\de.
\end{aligned}\right.
\end{equation*}
Then standard energy estimates show that zero is the only solution for this 
system.

The formulas for $T^5$ are obtained directly by the equations for $A$ \eqref{Heat-A-pre} and heat gauge $B=\nab^\al A_\al$. It remains to derive the system for $(T^1,\cdots, T^4)$.
\begin{proof}[The equation for $T^1$]
	This has the form
	\begin{equation*}
	(\d_t-\De_g)T^{1,\be}_{\al}=\la^2 T^4+T^1\nab V+V\nab T^1+T^3\nab\la +\la\nab T^3.
	\end{equation*}
	Using the parabolic equations for $h$ we recover the representation of $\d_t g$ as
	\begin{equation}  \label{dtg-formula}
	\d_t g_{\mu\nu}=2G_{\mu\nu}-2T^1_{\mu\nu},\quad G_{\mu\nu}:=\Im(\psi\bar{\la}_{\mu\nu})+\frac{1}{2}\nab_\mu V_\nu+\frac{1}{2}\nab_\nu V_\mu,
	\end{equation}
	and obtain 
	\begin{equation}   \label{dtGa}
	\d_t \Ga_{\al\be}^\ga=\nab_\al G_{\be}^\ga+\nab_\be G_{\al}^\ga-\nab^\ga G_{\al\be}-(\nab_\al T_{\be}^{1,\ga}+\nab_\be T_{\al}^{1,\ga}-\nab^\ga T^1_{\al\be}).
	\end{equation}
	We then use the two formulas to write 
	\begin{align*}
	\d_t {\Ric_\al}^\be=&\ \d_t g^{\mu\nu} {R_{\mu\al\nu}}^\be- g^{\mu\nu}\d_t {R^\be}_{\mu\nu\al}\\
	=&\ (-2G^{\mu\nu}+2T^{1,\mu\nu}){R_{\mu\al\nu}}^\be+g^{\mu\nu}(\nab_\al \d_t \Ga^\be_{\mu\nu}-\nab_\nu \d_t\Ga^\be_{\mu\al})\\
	=&\ 2T^{1,\mu\nu}{R_{\mu\al\nu}}^\be-2G^{\mu\nu}{R_{\mu\al\nu}}^\be\\
	&\ +\nab_\al[2\nab^\mu (G_{\mu}^\be-T_{\mu}^{1,\be})-\nab^\be (G_{\mu}^\mu-T^{1,\mu}_{\mu})]\\
	&\ -\nab^\mu[\nab_\mu (G_{\al}^\be-T_{\al}^{1,\be})+\nab_\al (G_{\mu}^\be-T_{\mu}^{1,\be})-\nab^\be (G_{\mu\al}-T^1_{\mu\al})]\\
	=&\ \nab^\mu\nab_\mu T_{\al}^{1,\be}+ 2T^{1,\mu\nu}{R_{\mu\al\nu}}^\be+\nab_\al(-2\nab^\mu T_{\mu}^{1,\be}+\nab^\be T^{1,\mu}_{\mu})+\nab^\mu(\nab_\al T_{\mu}^{1,\be}-\nab^\be T^1_{\mu\al})\\
	&\ -2G^{\mu\nu}{R_{\mu\al\nu}}^\be+2[\nab_\al,\nab^\mu] G_{\mu}^\be-\nab_\al\nab^\be G_{\mu}^\mu
	-\nab^\mu(\nab_\mu G_{\al}^\be-\nab_\al G_{\mu}^\be-\nab^\be G_{\mu\al}).
	\end{align*}
	By the relation $\nab^\mu T^1_{\mu\nu}=\frac{1}{2}\nab_\nu T^{1,\mu}_\mu$, the third term in the right hand side vanishes. We can also rewrite the fourth term as 
	\begin{align*}
	\nab^\mu(\nab_\al T_{\mu}^{1,\be}-\nab^\be T^1_{\mu\al})=&\ [\nab^\mu,\nab_\al] T_{\mu}^{1,\be}-[\nab^\mu,\nab^\be] T^1_{\mu\al}+\nab_\al\nab^\mu T_{\mu}^{1,\be}-\nab^\be\nab^\mu T^1_{\mu\al}\\
	=&\ {R^\mu}_{\al\mu\de}T^{1,\de\be}+{R^\mu}_{\al\be\de}T^{1,\de}_\mu-{R^{\mu\be}}_{\mu\de}T^{1,\de}_\al-{R^{\mu\be}}_{\al\de}T^{1,\de}_\mu+[\nab_\al,\nab^\be]T^{1,\mu}_\mu,
	\end{align*}
	where the last term vanishes. Commuting
	we compute the fifth and sixth terms as
	\begin{align*}
	-2G^{\mu\nu}{R_{\mu\al\nu}}^\be+2[\nab_\al,\nab^\mu] G^{\be}_\mu
	=-2G^{\mu\nu}{R_{\mu\al\nu}}^\be+2 R_{\al\mu\mu\nu}G^{\be\nu}+2R_{\al\mu\be\nu}G^{\mu\nu}=-2\Ric_{\al\nu} G^{\be\nu}.
	\end{align*}
	Hence from the above three formulas and the representation of $G_{\mu\nu}$ \eqref{dtg-formula} we rearrange $\d_t {\Ric_\al}^\be$ as
	\begin{align*}
	&\ \d_t {\Ric_\al}^\be-\De_g T^{1,\be}_\al\\
	=&\ RT^1-2\Ric_{\al\nu}G^{\be\nu}-\nab_\al\nab^\be G_{\mu}^\mu
	+\nab^\mu(-\nab_\mu G_{\al}^\be+\nab_\al G_{\mu}^\be+\nab^\be G_{\mu\al})\\
	=&\ RT^1 -(\nab^\mu V^\nu+\nab^\nu V^\mu){R_{\mu\al\nu}}^\be\\  \tag{$I_1$}
	&\ -2\Ric_{\al\nu}\Im(\psi\bar{\la}^{\be\nu})
	+\nab^\mu(-\nab_\mu \Im(\psi\bar{\la}_\al^{\be})+\nab_\al \Im(\psi\bar{\la}_\mu^{\be})+\nab^\be \Im(\psi\bar{\la}_{\mu\al}))\\ \tag{$I_2$}
	&\ +\big[\nab_\al\nab^\mu(\nab_\mu V^\be+\nab^\be V_\mu) -\nab_\al\nab^\be \nab^\mu V_\mu\\
	&\ \quad +\frac{1}{2}\nab^\mu[-\nab_\mu (\nab_\al V^\be+\nab^\be V_\al)-\nab_\al(\nab_\mu V^\be+\nab^\be V_\mu)+\nab^\be (\nab_\mu V_\al+\nab_\al V_\mu)]\big].
	\end{align*}
	We write $I_1$ as
	\begin{align*}
	I_1=&\ -2\Ric_{\al\nu}\Im(\psi\bar{\la}^{\be\nu})\\
	&\ +\Im(-\nab^{A,\mu}\nab^A_{\mu}\psi\bar{\la}^\be_\al-2\nab^{A,\mu}\psi\overline{\nab^A_\mu \la^\be_\al}-\psi\overline{\nab^{A,\mu}\nab^A_\mu \la^\be_\al}\\
	&\ +\nab^{A,\mu}\nab^A_\al\psi \bar{\la}^\be_\mu+\nab^{A,\mu}\psi \overline{\nab^A_\al \la^\be_\mu}+\nab^{A}_\al\psi \overline{\nab^{A,\mu} \la^\be_\mu}+\psi \overline{\nab^{A,\mu}\nab^A_\al\la^\be_\mu}\\
	&\ +\nab^{A,\mu}\nab^{A,\be}\psi \bar{\la}_{\mu\al}+\nab^{A,\mu}\psi \overline{\nab^{A,\be} \la_{\mu\al}}+\nab^{A,\be}\psi \overline{\nab^{A,\mu} \la_{\mu\al}}+\psi \overline{\nab^{A,\mu}\nab^{A,\be}\la^{\mu\al}})\\
	=&\ \nab\psi T^3+\psi\nab T^3-2\Ric_{\al\nu}\Im(\psi\bar{\la}^{\be\nu})\\
	&\ +\Im(-\nab^{A,\mu}\nab^A_{\mu}\psi\bar{\la}^\be_\al
	+\nab^{A,\mu}\nab^A_\al\psi \bar{\la}^\be_\mu
	+\nab^{A,\mu}\nab^{A,\be}\psi \bar{\la}_{\mu\al}+\psi \overline{\nab^{A,\mu}\nab^{A,\be}\la^{\mu\al}})
	\end{align*}
	Here the $I_1$ term will be cancelled by $J_1,J_2$ later modulo $\{ \psi \nab T^3,\la\la T^1,\la\la T^4\}$.
	Using commutators we rearrange $I_2$ as 
	\begin{align*}
	I_2
	=&\ [\nab_\al,\nab^\mu]\nab_\mu V^\be+\frac{1}{2}\nab^\mu[\nab_\al,\nab_\mu]V^\be+\frac{1}{2}\nab^\mu[\nab^\be,\nab_\mu]V_\al\\
	&\ +\nab_\al[\nab^\mu,\nab^\be]V_\mu+\frac{1}{2}\nab^\mu[\nab^\be,\nab_\al]V_\mu.
	\end{align*}
	Then by Riemannian curvature and Bianchi identities we have
	\begin{align*}
	I_2=&\  R_{\al\mu\mu\de}\nab^\de V^\be+R_{\al\mu\be\de}\nab^\mu V^\de+\frac{1}{2}\nab^\mu (R_{\al\mu\be\de}V^\de+R_{\be\mu\al\de}V^\de+R_{\be\al\mu\de}V^\de)+\nab_\al(R_{\mu\be\mu\de}V^\de)\\
	=&\ -\Ric_{\al\de}\nab^\de V^\be+R_{\al\mu\be\de}\nab^\mu V^\de+\nab^\mu(R_{\be\mu\al\de}V^\de)+\nab_\al(\Ric_{\be\de}V^\de)\\
	=&\ -\Ric_{\al\de}\nab^\de V^\be+(R_{\al\mu\be\de}+R_{\be\mu\al\de})\nab^\mu V^\de-\nab_\al R_{\de\mu\be\mu}V^\de-\nab_\de R_{\mu\al\be\mu}V^\de\\
	&\ +\nab_\al \Ric_{\be\de}V^\de+\Ric_{\be\de}\nab_\al V^\de\\
	=&\ -\Ric_{\al\de}\nab^\de V^\be+R_{\al\mu\be\de}(\nab^\mu V^\de+\nab^\de V_\mu)+\nab_\de \Ric_{\al\be}V^\de
	+\Ric_{\be\de}\nab_\al V^\de,
	\end{align*}
	which gives
	\begin{align*}
	I_2-(\nab^\mu V^\nu+\nab^\nu V^\mu){R_{\mu\al\nu}}^\be=-\Ric_{\al\de}\nab^\de V^\be+V^\de\nab_\de \Ric_{\al\be}
	+\Ric_{\be\de}\nab_\al V^\de.
	\end{align*}
	This term will be cancelled by $J_3$ modulo $\{T^1\nab V,V\nab T^1\}$.
	
	Next, we compute the expression for $-\d_t {\tRic_\al}^\be$. From the $\la$-equations \eqref{mdf-Shr-sys-2} and the formula \eqref{dtg-formula} we have the evolution equation for $\la^\si_\al$ 
	\begin{equation}    \label{dtla-rep}
	\begin{aligned}
	i\d_t^B \la_\al^\si+\frac{1}{2}(\nab^A_\al \nab^{A,\si}+\nab^{A,\si}\nab^A_\al)\psi+\la(T^1+T^2+T^4)
	+i\la^\ga_\al (\frac{3}{2}\Im(\psi\bar{\la}^\si_\ga)+\nab_\ga V^\si)&\\
	-i\la^{\ga\si}(\frac{1}{2}\Im(\psi\bar{\la}_{\ga\al})+\nab_\al V_\ga)-iV^\ga \nab^A_\ga \la^\si_\al=&0,
	\end{aligned}
	\end{equation}
	and the evolution equation for the mean curvature $\psi$
	\begin{align*}
	i\d_t^B \psi+\nab^A_\si \nab^{A,\si}\psi+\la(T^1+T^2+T^4)
	+i\la^\ga_\si \Im(\psi\bar{\la}^\si_\ga)-iV^\ga \nab^A_\ga \psi=0.
	\end{align*}
	Then for ${\tRic_\al}^\be=\Re(\la_\al^\be \bar{\psi}-\la_\al^\si \bar{\la}_\si^\be)$, by the above two formulas we have
	\begin{align*}
	-\d_t {\tRic_\al}^\be=&\ -\Re(\d^B_t \la^\be_\al \bar{\psi}+\la^\be_\al \overline{\d^B_t \psi}-\d^B_t\la_{\al}^\mu\bar{\la}^{\be}_\mu-\la_{\al}^\mu\overline{\d^B_t\la^\be_\mu})\\
	=&\ \Im(-i\d^B_t \la^\be_\al \bar{\psi}-\bar{\la}^\be_\al i\d^B_t \psi+i\d^B_t\la_{\al}^\mu\bar{\la}^{\be}_\mu+\bar{\la}_{\al}^\mu i\d^B_t\la^\be_\mu)\\
	=&\ \la^2(T^1+T^2+T^4)+K_1+K_2+K_3+K_4,
	\end{align*}
	where
	\begin{align*}
	K_1=&\ \Im \big[ \big( \frac{1}{2}(\nab^A_\al \nab^{A,\be}+\nab^{A,\be}\nab^A_\al)\psi
	+i\la^\ga_\al (\frac{3}{2}\Im(\psi\bar{\la}^\be_\ga)+\nab_\ga V^\be)\\
	&\quad\quad-i\la^{\ga\be}(\frac{1}{2}\Im(\psi\bar{\la}_{\ga\al})+\nab_\al V_\ga)-iV^\ga \nab^A_\ga \la^\be_\al  \big)\bar{\psi}\big],\\
	K_2=&\ \Im\big[\bar{\la}^\be_\al \big( \nab^A_\si \nab^{A,\si}\psi
	+i\la^\ga_\si\Im(\psi\bar{\la}^\si_\ga)-iV^\ga \nab^A_\ga \psi  \big) \big],\\
	K_3=&\ -\Im \big[ \big( \frac{1}{2}(\nab^A_\al \nab^{A,\si}+\nab^{A,\si}\nab^A_\al)\psi
	+i\la^\ga_\al (\frac{3}{2}\Im(\psi\bar{\la}^\si_\ga)+\nab_\ga V^\si)\\
	&\quad\quad-i\la^{\ga\si}(\frac{1}{2}\Im(\psi\bar{\la}_{\ga\al})+\nab_\al V_\ga)-iV^\ga \nab^A_\ga \la^\si_\al  \big)\bar{\la}^\be_\si\big],\\
	K_4=&\ -\Im \big[ \bar{\la}^\si_\al\big( \frac{1}{2}(\nab^A_\si \nab^{A,\be}+\nab^{A,\be}\nab^A_\si)\psi
	+i\la^\ga_\si (\frac{3}{2}\Im(\psi\bar{\la}^\be_\ga)+\nab_\ga V^\be)\\
	&\quad\quad-i\la^{\ga\be}(\frac{1}{2}\Im(\psi\bar{\la}_{\ga\si})+\nab_\si V_\ga)-iV^\ga \nab^A_\ga \la^\be_\si  \big)\big].
	\end{align*}
	This can be further rearranged as
	\begin{align*}
	-\d_t {\tRic_\al}^\be=\la^2(T^1+T^2+T^4)+J_1+J_2+J_3,
	\end{align*}
	where $J_1$, $J_2$ and $J_3$ are
	\begin{align*}
	J_1=&\ \Im \big[  \frac{1}{2}(\nab^A_\al \nab^{A,\be}\psi+\nab^{A,\be}\nab^A_\al\psi)
	\bar{\psi}+\bar{\la}^\be_\al \nab^A_\si \nab^{A,\si}\psi\\
	&\ \quad -\frac{1}{2}(\nab^A_\al \nab^{A,\si}\psi+\nab^{A,\si}\nab^A_\al\psi)
	\bar{\la}^\be_\si
	-\frac{1}{2}\bar{\la}^\si_\al (\nab^A_\si \nab^{A,\be}\psi+\nab^{A,\be}\nab^A_\si\psi)
	\big],
	\end{align*}
	\begin{align*}
	J_2=&\ \frac{3}{2}\Re(\la^\ga_\al\bar{\psi})\Im(\psi\bar{\la}^\be_\ga)
	-\frac{1}{2}\Re(\la^{\ga\be}\bar{\psi})\Im(\psi\bar{\la}_{\ga\al})
	+\Re(\bar{\la}^\be_\al \la^\ga_\si) \Im(\psi\bar{\la}^\si_\ga) \\
	&\ -\frac{3}{2}\Re(\la^\ga_\al\bar{\la}^\be_\si) \Im(\psi\bar{\la}^\si_\ga)
	+\frac{1}{2}\Re(\la^{\ga\si}\bar{\la}^\be_\si)\Im(\psi\bar{\la}_{\ga\al}) \\
	&\ -\frac{3}{2}\Re(\bar{\la}^\si_\al\la^\ga_\si) \Im(\psi\bar{\la}^\be_\ga)
	+\frac{1}{2}\Re(\bar{\la}^\si_\al\la^{\ga\be})\Im(\psi\bar{\la}_{\ga\si}) ,
	\end{align*}
	\begin{align*}
	J_3=& \ 
	\Re(\la^\ga_\al\bar{\psi}) \nab_\ga V^\be
	-\Re(\la^{\ga\be}\bar{\psi})\nab_\al V_\ga-V^\ga \Re( \nab^A_\ga \la^\be_\al  \bar{\psi})\\
	&\ -V^\ga \Re(\bar{\la}^\be_\al \nab^A_\ga \psi )  \\
	&\ -\Re(\la^\ga_\al \bar{\la}^\be_\si) \nab_\ga V^\si
	+\Re(\la^{\ga\si}\bar{\la}^\be_\si)\nab_\al V_\ga+V^\ga  \Re(\nab^A_\ga \la^\si_\al \bar{\la}^\be_\si)  \\
	&\ -\Re(\bar{\la}^\si_\al\la^\ga_\si)\nab_\ga V^\be
	+\Re(\bar{\la}^\si_\al\la^{\ga\be})\nab_\si V_\ga+V^\ga \Re(\bar{\la}^\si_\al \nab^A_\ga \la^\be_\si).  
	\end{align*}
	Then $I_1+J_1+J_2$ will vanish modulo $\{\psi\nab T^3,\la^2 T^1,\la^2 T^4\}$. Precisely, we have
	\begin{align*}
	I_1+J_1=&\ \nab\psi T^3+\psi\nab T^3-2\Ric_{\al\nu}\Im(\psi\bar{\la}^{\be\nu})\\
	&\ +\Im\big( \frac{1}{2}[\nab^{A,\mu},\nab^A_\al]\psi\bar{\la}^\be_\mu +\frac{1}{2}[\nab^{A,\mu},\nab^{A,\be}]\psi \bar{\la}_{\mu\al}+\psi\overline{[\nab^{A,\mu},\nab^{A,\be}]\la_{\mu\al}}\\
	&\ \quad +\psi \overline{\nab^{A,\be}T_{\mu\al,}^{3\ \mu}}+\frac{1}{2}\psi\overline{[\nab^{A,\be},\nab^A_\al]\psi}  \big)\\
	=&\ \nab\psi T^3+\psi\nab T^3-2\Ric_{\al\nu}\Im(\psi\bar{\la}^{\be\nu})
	+\frac{1}{2}{\bmF^\mu}_\al \Re(\psi\bar{\la}^\be_\mu) -\frac{1}{2}\bmF^{\mu\be}\Re(\psi \bar{\la}_{\mu\al})\\
	&\ 
	+ \Ric^{\mu\be} \Im(\psi\bar{\la}_{\mu\al})+ R_{\mu\be\al\de}\Im(\psi\bar{\la}^{\mu\de}) 
	-\frac{1}{2}|\psi|^2{\bmF^\be}_{\al}.
	\end{align*}
	We rewrite  $J_2$ as 
	\begin{align*}
	J_2
	=& \ \frac{3}{2}\Im(\psi \bar{\la}^{\be\ga})\tRic_{\al\ga}-\frac{1}{2}\Im(\psi\bar{\la}_{\ga\al})\tRic^{\ga\be}+\Im(\psi\bar{\la}^{\si\ga}){\tR^\be}_{\ \si\al\ga}.
	\end{align*}
	Then we obtain 
	\begin{align*}
	I_1+J_1+J_2=& \ \psi\nab T^3+\la^2 (T^1+T^4)
	+\frac{1}{2}\Im(\psi\bar{\la}_{\ga\al})\tRic^{\ga\be}-\frac{1}{2}\Im(\psi \bar{\la}^{\be\ga})\tRic_{\al\ga}\\
	& \ +\frac{1}{2}{\tbmF^\mu}_{\ \al} \Re(\psi\bar{\la}^\be_\mu) -\frac{1}{2}\tbmF^{\mu\be}\Re(\psi \bar{\la}_{\mu\al})-\frac{1}{2}|\psi|^2{\tbmF^\be}_{\ \al}\\
	=& \ \psi\nab T^3+\la^2 (T^1+T^4).
	\end{align*}
	
	We can also show that $I_2-(\nab^\mu V^\nu+\nab^\nu V^\mu){R_{\mu\al\nu}}^\be+J_3$ vanishes modulo $\{T^1\nab V,V\nab T^1\}$. This is because $J_3$ can be written as
	\begin{align*}
	J_3=& 
	\ \Re(\la^\ga_\al\bar{\psi}) \nab_\ga V^\be
	-\Re(\la^{\ga\be}\bar{\psi})\nab_\al V_\ga-V^\ga \Re( \nab^A_\ga \la^\be_\al  \bar{\psi})\\
	&\ +\Re(\bar{\la}^\be_\al \la_\si^\ga) ( 
	\nab_\ga V^\si-\nab^\si V_\ga)-V^\ga \Re(\bar{\la}^\be_\al \nab^A_\ga \psi )  \\
	&\ -\Re(\la^\ga_\al \bar{\la}^\be_\si) \nab_\ga V^\si
	+\Re(\la^{\ga\si}\bar{\la}^\be_\si)\nab_\al V_\ga+V^\ga  \Re(\nab^A_\ga \la^\si_\al \bar{\la}^\be_\si)  \\
	&\ -\Re(\bar{\la}^\si_\al\la^\ga_\si)\nab_\ga V^\be
	+\Re(\bar{\la}^\si_\al\la^{\ga\be})\nab_\si V_\ga+V^\ga \Re(\bar{\la}^\si_\al \nab^A_\ga \la^\be_\si) \\
	=&\ \tRic_{\al\ga}\nab^\ga V^\be -\tRic^{\ga\be}\nab_\al V_\ga-V^\ga \nab_\ga {\tRic_{\al}}^\be.
	\end{align*}
	Then we have
	\begin{align*}
	I_2-(\nab^\mu V^\nu+\nab^\nu V^\mu){R_{\mu\al\nu}}^\be+J_3=-T^1_{\al\ga}\nab^\ga V^\be+T^{1,\ga\be}\nab_\al V_\ga+V^\ga \nab_\ga T^{1,\ \be}_{\al}.
	\end{align*}
	This concludes the proof of the $T^1$-equations.
\end{proof}

\medskip

\begin{proof}[The equation for $T^2$]
	By the second Bianchi identities for the Riemannian curvature and the following equality
	\begin{align*}  
	&\nab_\de \tR_{\si\ga\al\be}+\nab_\si \tR_{\ga\de\al\be}+\nab_\ga \tR_{\de\si\al\be}\\
	=&\ \Re(T^3_{\de\ga,\be}\bar{\la}_{\al\si}+T^3_{\de\si,\al}\bar{\la}_{\be\ga}-T^3_{\de\ga,\al}\bar{\la}_{\be\si}-T^3_{\de\si,\be}\bar{\la}_{\al\ga}+T^3_{\si\ga,\al}\bar{\la}_{\be\de}-T^3_{\si\ga,\be}\bar{\la}_{\al\de}),
	\end{align*}
	we have the counterpart of the second Bianchi identities 
	\begin{equation*}
	\nab_\delta T^2_{\si\ga\al\be} + \nab_{\si} T^2_{\ga\de\al\be} 
	+ \nab_{\ga} T^2_{\de\si\al\be} = T^1\la,
	\end{equation*}
	which combine with the algebraic symmetries of the same tensor to yield 
	an elliptic system for $T^2$. Precisely, 
	using the above relation we have
	\begin{equation*}
	\nab^{\si}T^2_{\si\ga\al\be} = \nabla_{\al} T^1_{\ga\be} - \nabla_{\be} T^1_{\ga\al}
	+  T^1\la,
	\end{equation*}
	which combined with the previous one yields the desired elliptic system, 
	with $T^1$ viewed as a source term.
\end{proof}

\medskip

\begin{proof}[The equations for $T^3$]
	This has the form
	\begin{equation*}
	\begin{aligned}
	(i\d^B_t -\De^A_g)T^3_{\al\be,\ga}=&\ \la T^5+V\nab T^3+T^3(\nab V+\la^2+R)+(\nab^A\la+\la V)(T^1+T^2+T^4)\\
	&+\nab (T^2+T^4)\la.
	\end{aligned}
	\end{equation*}
	Recall the $\la$-equations 
	\begin{align*}
	i\d^B_t \la_{\be\ga}+\nab^A_\mu \nab^{A,\mu} \la_{\be\ga}-\frac{1}{2}(\tRic_{\be\de}\la^\de_\ga+\tRic_{\ga\de}\la^\de_\be)+\tR_{\be\si\ga\de}\la^{\si\de}+\frac{i}{2}(\tbmF_{\be\de}\la^\de_\ga+\tbmF_{\ga\de}\la^\de_\be)&\\
	-\frac{i}{2}\la^\de_\be[\Im(\psi\bar{\la}_{\de\ga})+2\nab_\ga V_\de]-\frac{i}{2}\la^\de_\ga[\Im(\psi\bar{\la}_{\de\be})+2\nab_\be V_\de]-iV^\de\nab^A_\de \la_{\be\ga}&=0.
	\end{align*}
	Applying $\nab^A_\al$ and $\nab^A_\be$ to the above $\la_{\be\ga}$ and $\la_{\al\ga}$-equations respectively, we obtain the difference
	\begin{align*}
	0=&\ \big[\nab^A_\al(i\d^B_t \la_{\be\ga}+\nab^A_\mu \nab^{A,\mu} \la_{\be\ga})-\nab^A_\be(i\d^B_t \la_{\al\ga}+\nab^A_\mu \nab^{A,\mu} \la_{\al\ga})\big]\\
	&\ +\big[\nab^A_\al [-\frac{1}{2}(\tRic_{\be\de}\la^\de_\ga+\tRic_{\ga\de}\la^\de_\be)+\tR_{\be\si\ga\de}\la^{\si\de}+\frac{i}{2}(\tbmF_{\be\de}\la^\de_\ga+\tbmF_{\ga\de}\la^\de_\be)]\\
	&\ \quad-\nab^A_\be[-\frac{1}{2}(\tRic_{\al\de}\la^\de_\ga+\tRic_{\ga\de}\la^\de_\al)+\tR_{\al\si\ga\de}\la^{\si\de}+\frac{i}{2}(\tbmF_{\al\de}\la^\de_\ga+\tbmF_{\ga\de}\la^\de_\al)]\big]\\
	&\ +\big[\nab^A_\al [-\frac{i}{2}\la^\de_\be[\Im(\psi\bar{\la}_{\de\ga})+2\nab_\ga V_\de]-\frac{i}{2}\la^\de_\ga[\Im(\psi\bar{\la}_{\de\be})+2\nab_\be V_\de]-iV^\de\nab^A_\de \la_{\be\ga}]\\
	&\ \quad -\nab^A_\be[-\frac{i}{2}\la^\de_\al[\Im(\psi\bar{\la}_{\de\ga})+2\nab_\ga V_\de]-\frac{i}{2}\la^\de_\ga[\Im(\psi\bar{\la}_{\de\al})+2\nab_\al V_\de]-iV^\de\nab^A_\de \la_{\al\ga}]\big]\\
	:=&\ I+II+III.
	\end{align*}
	
	We first compute the $I$. We commute $\nab^A_\al$ with $\d^B_t$ and $\nab^A_\mu \nab^{A,\mu}$ to give
	\begin{align} \nonumber
	I=&\ (i\d_t-\De^A_g) T^3_{\al\be,\ga}+i[\nab^A_\al,\d^B_t]\la_{\be\ga}+[\nab^A_\al,\nab^A_\mu\nab^{A,\mu}]\la_{\be\ga}\\  \nonumber
	&\ -i[\nab^A_\be,\d^B_t]\la_{\al\ga}-[\nab^A_\be,\nab^A_\mu\nab^{A,\mu}]\la_{\al\ga}\\\nonumber
	=&\ (i\d_t-\De^A_g) T^3_{\al\be,\ga}\\  \tag{$I_{1}$}
	&\ +i\d_t \Ga^\si_{\al\ga}\la_{\be\si}+\bmF_{0\al}\la_{\be\ga}-i\d_t \Ga^\si_{\be\ga}\la_{\al\si}-\bmF_{0\be}\la_{\al\ga}\\\tag{$I_{2}$}
	&\ +\big[[\nab_\al,\nab_\mu]\nab^{A,\mu}\la_{\be\ga}+i\bmF_{\al\mu}\nab^{A,\mu}\la_{\be\ga}
	+\nab^{A,\mu}([\nab_\al,\nab_\mu]+i\bmF_{\al\mu})\la_{\be\ga}\\\nonumber
	&\ \quad -[\nab_\be,\nab_\mu]\nab^{A,\mu}\la_{\al\ga}-i\bmF_{\be\mu}\nab^{A,\mu}\la_{\al\ga}
	-\nab^{A,\mu}([\nab_\be,\nab_\mu]+i\bmF_{\be\mu})\la_{\al\ga}\big].
	\end{align}
	For $I_1$, by the formulas for $\d_t \Ga$ in \eqref{dtGa}, for $G_{\mu\nu}$ in \eqref{dtg-formula} and for the commutators $[\nab_\al,\nab_\be]$ we have
	\begin{align} \nonumber
	I_1=&\  i(\nab_\al G_{\ga\de}+\nab_\ga G_{\al\de}-\nab_\de G_{\al\ga})\la^\de_\be-i(\nab_\be G_{\ga\de}+\nab_\ga G_{\be\de}-\nab_\de G_{\be\ga})\la^\de_\al+\nab T^1 \la\\\nonumber
	&\  +T^5_\al \la_{\be\ga}-T^5_\be \la_{\al\ga}+(\Re(\la^\si_\al \overline{\nab^A_\si \psi})-\tbmF_{\al\si}V^\si)\la_{\be\ga}-(\Re(\la^\si_\be \overline{\nab^A_\si \psi})-\tbmF_{\be\si}V^\si)\la_{\al\ga}\\\nonumber
	=&\   \nab T^1 \la+ T^5_\al \la_{\be\ga}-T^5_\be \la_{\al\ga}+I_{11}+I_{12},
	\end{align}
	where $I_{11}$, $I_{12}$ are the terms containing $\la\la\nab \la$ and $V\la$ respectively,
	\begin{align*}
    I_{11}:=&\  i(\nab_\al\Im(\psi\bar{\la}_{\ga\de})+\nab_\ga\Im(\psi\bar{\la}_{\al\de})-\nab_\de\Im(\psi\bar{\la}_{\al\ga}))\la^\de_\be\\\nonumber
	&\  \quad -i(\nab_\be\Im(\psi\bar{\la}_{\ga\de})+\nab_\ga\Im(\psi\bar{\la}_{\be\de})-\nab_\de\Im(\psi\bar{\la}_{\be\ga}))\la^\de_\al\\\nonumber
	&\  \quad+\Re(\la^\si_\al \overline{\nab^A_\si \psi})\la_{\be\ga}-\Re(\la^\si_\be \overline{\nab^A_\si \psi})\la_{\al\ga},\\
	I_{12}:=&\  \frac{i}{2}[(\nab_\al\nab_\ga +\nab_\ga\nab_\al)V_\de \la^\de_\be+R_{\al\si\ga\de}V^\de \la^\si_\be+R_{\ga\si\al\de}V^\de \la^\si_\be]\\\nonumber
	&\  \quad -\frac{i}{2}[(\nab_\be\nab_\ga +\nab_\ga\nab_\be)V_\de \la^\de_\al+R_{\be\si\ga\de}V^\de \la^\si_\al+R_{\ga\si\be\de}V^\de \la^\si_\al]\\\nonumber
	&\  \quad -\tbmF_{\al\si}V^\si\la_{\be\ga}+\tbmF_{\be\si}V^\si\la_{\al\ga}.
	\end{align*}
	Here, using the expressions for $\tbmF_{\al\be}$ and $\tR_{\be\al\ga\de}$, the expression $I_{11}$  can be rewritten as
	\begin{align*}
	I_{11}= i\nab_\al\Im(\psi \bar{\la}_{\ga\de})\la^\de_\be-i\nab_\be\Im(\psi \bar{\la}_{\ga\de})\la^\de_\al-i\nab^A_\ga\psi \tbmF_{\al\be}+\nab^A_\de \psi \tR_{\be\al\ga\de},
	\end{align*}
	Using commutators $[\nab_\ga,\nab_\al]$ and the Bianchi identities, the $I_{12}$ expression  can be rewritten as
	\begin{align*}
	    I_{12}=i\nab_\al\nab_\ga V_\de \la^\de_\be +i R_{\ga\si\al\de}V^\de \la^\si_\be-i\nab_\be\nab_\ga V_\de \la^\de_\al -i R_{\ga\si\be\de}V^\de \la^\si_\al-\tbmF_{\al\si}V^\si\la_{\be\ga}+\tbmF_{\be\si}V^\si\la_{\al\ga}
	\end{align*}

	For $I_2$, we use the Riemannian curvature tensor to write
	\begin{align}\nonumber
	I_2=&\   R_{\al\mu\mu\de}\nab^{A,\de}\la_{\be\ga}+R_{\al\mu\be\de}\nab^{A,\mu}\la^\de_\ga+R_{\al\mu\ga\de}\nab^{A,\mu}\la^\de_\be+i\bmF_{\al\mu}\nab^{A,\mu}\la_{\be\ga}\\\nonumber
	&\  +\nab^{A,\mu}(R_{\al\mu\be\de}\la^\de_\ga+R_{\al\mu\ga\de}\la^\de_\be+i\bmF_{\al\mu}\la_{\be\ga})\\\nonumber
	&\  -R_{\be\mu\mu\de}\nab^{A,\de}\la_{\al\ga}-R_{\be\mu\al\de}\nab^{A,\mu}\la^\de_\ga-R_{\be\mu\ga\de}\nab^{A,\mu}\la^\de_\al-i\bmF_{\be\mu}\nab^{A,\mu}\la_{\al\ga}\\\nonumber
	&\  -\nab^{A,\mu}(R_{\be\mu\al\de}\la^\de_\ga+R_{\be\mu\ga\de}\la^\de_\al+i\bmF_{\be\mu}\la_{\al\ga})\\\nonumber
	=&\   -\Ric_{\al\de}\nab^{A,\de}\la_{\be\ga}+2R_{\al\mu\be\de}\nab^{A,\mu}\la^\de_\ga+2R_{\al\mu\ga\de}\nab^{A,\mu}\la^\de_\be+2i\bmF_{\al\mu}\nab^{A,\mu}\la_{\be\ga}\\\nonumber
	&\  +\nab^\mu R_{\al\mu\be\de}\la^\de_\ga+\nab^\mu R_{\al\mu\ga\de}\la^\de_\be+i\nab^\mu \bmF_{\al\mu}\la_{\be\ga}\\\nonumber
	&\  +\Ric_{\be\de}\nab^{A,\de}\la_{\al\ga}-2R_{\be\mu\al\de}\nab^{A,\mu}\la^\de_\ga-2R_{\be\mu\ga\de}\nab^{A,\mu}\la^\de_\al-2i\bmF_{\be\mu}\nab^{A,\mu}\la_{\al\ga}\\\nonumber
	&\  -\nab^\mu R_{\be\mu\al\de}\la^\de_\ga-\nab^\mu R_{\be\mu\ga\de}\la^\de_\al-i\nab^\mu \bmF_{\be\mu}\la_{\al\ga}\\\nonumber
	=&\   2R_{\al\mu\be\de}{T^{3,\mu\de}}_{,\ga}+(T^1+T^2+T^4)\nab^A\la+\nab (T^2+T^4)\la+J_1,
	\end{align}
	where the terms in $J_1$ have the form $\la\la\nab\la$ as below 
	\begin{align*}
	J_1=&\   \nab^{A,\de}\la_{\be\ga}(-\tRic_{\al\de}+2i\tbmF_{\al\de})+\nab^{A,\de}\la_{\al\ga}(\tRic_{\be\de}-2i\tbmF_{\be\de}) +i\nab^\mu \tbmF_{\al\mu}\la_{\be\ga}-i\nab^\mu \tbmF_{\be\mu}\la_{\al\ga}\\
	&\  +2\tR_{\al\mu\ga\de}\nab^{A,\mu}\la^\de_\be+\nab^\mu \tR_{\al\mu\be\de}\la^\de_\ga+\nab^\mu \tR_{\al\mu\ga\de}\la^\de_\be\\
	&\  -2\tR_{\be\mu\ga\de}\nab^{A,\mu}\la^\de_\al-\nab^\mu \tR_{\be\mu\al\de}\la^\de_\ga-\nab^\mu \tR_{\be\mu\ga\de}\la^\de_\al.
	\end{align*}

	We next rewrite the $III$ expression as 
	\begin{align*} \nonumber
	III=&\  \nab^A_\al [-\frac{i}{2}\la^\de_\be[\Im(\psi\bar{\la}_{\de\ga})+2\nab_\ga V_\de]-\frac{i}{2}\la^\de_\ga[\Im(\psi\bar{\la}_{\de\be})+2\nab_\be V_\de]-iV^\de\nab^A_\de \la_{\be\ga}]\\ \nonumber
	&\  \quad -\nab^A_\be[-\frac{i}{2}\la^\de_\al[\Im(\psi\bar{\la}_{\de\ga})+2\nab_\ga V_\de]-\frac{i}{2}\la^\de_\ga[\Im(\psi\bar{\la}_{\de\al})+2\nab_\al V_\de]-iV^\de\nab^A_\de \la_{\al\ga}]\\ \nonumber
	=&\   -\frac{i}{2} T^3_{\al\be,\de}[\Im(\psi\bar{\la}^\de_{\ga})+2\nab_\ga V^\de]+III_1+III_2,
	\end{align*}
	where
	\begin{align*}
	    III_1:=&\  -\frac{i}{2}\la^\de_\be \nab_\al\Im(\psi\bar{\la}_{\de\ga})-\frac{i}{2}\nab^A_\al (\la^\de_\ga\Im(\psi\bar{\la}_{\de\be}))\\
	    &\  +\frac{i}{2}\la^\de_\al \nab_\be\Im(\psi\bar{\la}_{\de\ga})+\frac{i}{2}\nab^A_\be (\la^\de_\ga\Im(\psi\bar{\la}_{\de\al})),\\
	    III_2:=&\  -i\la^\de_\be \nab_\al\nab_\ga V_\de-i\nab^A_\al(\la^\de_\ga\nab_\be V_\de)-i\nab^A_\al(V^\de \nab^A_\de \la_{\be\ga})\\\nonumber
	&\  \quad +i\la^\de_\al \nab_\be\nab_\ga V_\de+i\nab^A_\be(\la^\de_\ga\nab_\al V_\de)+i\nab^A_\be(V^\de \nab^A_\de \la_{\al\ga}).
	\end{align*}

	The $I_{12}+III_2$ expression  vanishes modulo $\{V\nab T^3,T^3\nab V,\la V(T^2+T^4)\}$. Precisely, we can further write $III_2$ as
	\begin{align}\nonumber
	III_2=&\   -iT^3_{\al\de,\ga}\nab_\be V^\de+iT^3_{\be\de,\ga}\nab_\al V^\de-iV^\de \nab^A_\al T^3_{\de\be,\ga}+iV^\de \nab^A_\be T^3_{\de\al,\ga}\\\tag{$III_{21}$}
	&\  +i(-\la^\de_\be \nab_\al\nab_\ga V_\de+\la^\de_\al \nab_\be\nab_\ga V_\de-V^\de R_{\al\be\ga\si}\la^\si_\de)+V^\de\bmF_{\al\be}\la_{\de\ga}.
	\end{align}
Then replacing $R_{\al\be\ga\de}$, $\bmF_{\al\be}$ by $\tR_{\al\be\ga\de}$ and $\tbmF_{\al\be}$ respectively, we have
	\begin{align*}
	    I_{12}+III_{21}=&\  i(\la^\si_\be R_{\ga\si\al\de}V^\de-\la^\si_\al R_{\ga\si\be\de}V^\de-V^\de R_{\al\be\ga\si}\la^\si_\de)\\
	&\  -\tbmF_{\al\si}V^\si\la_{\be\ga}+\tbmF_{\be\si}V^\si\la_{\al\ga}+V^\de\bmF_{\al\be}\la_{\de\ga}\\
	=&\   i(\la^\si_\be T^2_{\ga\si\al\de}V^\de-\la^\si_\al T^2_{\ga\si\be\de}V^\de-V^\de T^2_{\al\be\ga\si}\la^\si_\de)+V^\de T^4_{\al\be}\la_{\de\ga}\\   \tag{$J_2$}
	&\  +\big[i(\la^\si_\be \tR_{\ga\si\al\de}V^\de-\la^\si_\al \tR_{\ga\si\be\de}V^\de-V^\de \tR_{\al\be\ga\si}\la^\si_\de)\\
	&\  \quad-\tbmF_{\al\si}V^\si\la_{\be\ga}+\tbmF_{\be\si}V^\si\la_{\al\ga}+V^\de\tbmF_{\al\be}\la_{\de\ga}\big]\\
	=&\   \la V T^2+\la VT^4,
	\end{align*}
	where the term $J_2$ vanishes due to the representations of $\tR_{\ga\si\al\de}$ and $\tbmF_{\al\si}$.

 	Next, we show that the terms $I_{11}+J_1+II+III_1$ vanish modulo $\la\la T^3$.  We have
	\begin{align*}
	I_{11}+III_1=&\  -i\nab^A_\ga\psi \tbmF_{\al\be}+\nab^A_\de \psi \tR_{\be\al\ga\de} 
	+\frac{i}{2}\la^\de_\be \nab_\al\Im(\psi\bar{\la}_{\de\ga})-\frac{i}{2}\la^\de_\al \nab_\be\Im(\psi\bar{\la}_{\de\ga})\\
	&\  -\frac{i}{2}\nab^A_\al (\la^\de_\ga\Im(\psi\bar{\la}_{\de\be}))+\frac{i}{2}\nab^A_\be (\la^\de_\ga\Im(\psi\bar{\la}_{\de\al})).
	\end{align*}
	We rewrite $II$ as 
	\begin{align}   \nonumber
	II=&\  \nab^A_\al [-\frac{1}{2}(\tRic_{\be\de}\la^\de_\ga+\tRic_{\ga\de}\la^\de_\be)+\tR_{\be\si\ga\de}\la^{\si\de}+\frac{i}{2}(\tbmF_{\be\de}\la^\de_\ga+\tbmF_{\ga\de}\la^\de_\be)]\\\nonumber
	&\  \quad-\nab^A_\be[-\frac{1}{2}(\tRic_{\al\de}\la^\de_\ga+\tRic_{\ga\de}\la^\de_\al)+\tR_{\al\si\ga\de}\la^{\si\de}+\frac{i}{2}(\tbmF_{\al\de}\la^\de_\ga+\tbmF_{\ga\de}\la^\de_\al)]\\\nonumber
	=&\   -\frac{1}{2}(\tRic_{\ga\de}-i\tbmF_{\ga\de}){T^3_{\al\be,}}^\de \\   \tag{$J_3$}
	&\  +\frac{1}{2}\nab^A_\al (\la^\de_\ga(-\tRic_{\be\de}+i\tbmF_{\be\de}))-\frac{1}{2}\nab^A_\be (\la^\de_\ga(-\tRic_{\al\de}+i\tbmF_{\al\de}))\\\tag{$J_4$}
	&\  +\frac{1}{2}\la^\de_\be \nab_\al (-\tRic_{\ga\de}+i\tbmF_{\ga\de})-\frac{1}{2}\la^\de_\al \nab_\be (-\tRic_{\ga\de}+i\tbmF_{\ga\de})\\\tag{$J_5$}
	&\  -\nab_\si \tR_{\al\be\ga\de}\la^{\si\de}+\tR_{\be\mu\ga\de}\nab^A_\al\la^{\mu\de}-\tR_{\al\mu\ga\de}\nab^A_\be\la^{\mu\de}.
	\end{align}
	And hence
	\begin{align*}
	&I_{11}+III_1+J_3+J_4\\
	=&-i\nab^A_\ga\psi \tbmF_{\al\be}+\nab^A_\de \psi \tR_{\be\al\ga\de} +\frac{1}{2}\la^\de_\be \nab_\al(-\la_{\ga\de}\bar{\psi}+\la_{\ga\si}\bar{\la}^\si_\de)\\
	&+\frac{1}{2}\la^\de_\al \nab_\be(\la_{\ga\de}\bar{\psi}-\la_{\ga\si}\bar{\la}^\si_\de)+\frac{1}{2}\nab^A_\al[\la^\de_\ga (-\bar{\la}_{\de\be}\psi+\la_{\be\si}\bar{\la}^\si_\de)]+\frac{1}{2}\nab^A_\be[\la^\de_\ga (\bar{\la}_{\de\al}\psi-\la_{\al\si}\bar{\la}^\si_\de)]\\
	=&-i\nab^A_\ga\psi \tbmF_{\al\be}+\nab^A_\de \psi \tR_{\be\al\ga\de}-\nab^A_\al \la^\de_\ga( \tRic_{\be\de}-i\tbmF_{\be\de})+\nab^A_\be \la^\de_\ga (\tRic_{\al\de}-i\tbmF_{\al\de})\\ &+\overline{\nab^A_\al \la^{\si\de}}\la_{\be\si}\la_{\ga\de}-\overline{\nab^A_\be \la^{\si\de}}\la_{\al\si}\la_{\ga\de}
	-\la^\de_\ga \Re(\nab^A_\al\psi \bar{\la}_{\be\de})+\la^\de_\ga \Re(\nab^A_\be\psi \bar{\la}_{\al\de})+\la^2 T^3.
	\end{align*}
	Since by $\tR$ and $\tbmF$ we also have
	\begin{align*}
	J_1+J_5=& -\nab^{A,\de}\la_{\be\ga}\tRic_{\al\de}+\nab^{A,\de}\la_{\al\ga}\tRic_{\be\de}+\la^\de_\ga \Re(\nab^A_\al\psi\bar{\la}_{\be\de}-\nab^A_\be\psi\bar{\la}_{\al\de})\\
	&-\overline{\nab^A_\al \la^{\mu\si}}\la_{\mu\be}\la_{\si\ga}+\overline{\nab^A_\be \la^{\mu\si}}\la_{\mu\al}\la_{\si\ga}-i\nab^{A,\de}\la_{\al\ga}\tbmF_{\be\de}+i\nab^{A,\de}\la_{\be\ga}\tbmF_{\al\de}\\
	&+\nab^{A,\si}\psi \tR_{\al\be\ga\si}+i\nab^A_\ga\psi \tbmF_{\al\be}+\la^2 T^3.
	\end{align*}
	Then in the above two formulas all terms cancel except for $\la\la T^3$. 
	Hence, we obtain that $I_{11}+J_1+II+III_1$ vanishes modulo $\la\la T^3$. 
	This concludes the proof of the $T^3$-equations.
\end{proof}

\medskip 

\begin{proof}[The equations for $T^4$]
	These have the form
	\begin{equation*}
	\begin{aligned}
	(\d_t -\De_g)T^4_{\al\be}=&-\Ric_{\alpha \delta} {T^{4,\delta}}_{\beta} + \Ric_{\beta\delta} {T^{4,\delta}}_{\al} - R_{\beta\alpha\sigma\delta} T^{2,\sigma\delta}\\
	&-\Re(\nab^{A,\si}\psi \overline{T^3_{\al\be,\si}})-V^\ga \Im(\la_{\ga}^\si\overline{T^3_{\al\be,\si}})+V^\ga \Im (T^3_{\ga\al,\si}\bar{\la}^\si_\be).
	\end{aligned}
	\end{equation*}
	By the $A$-equations we have
	\begin{align*}
	(\d_t -\De_g)T^4_{\al\be}=&
	-[\De_g,\nab_\al]A_\be+[\De_g,\nab_\be]A_\al
	-\nab_\al (\tRic_{\be\de}A^\de)+\nab_\be (\tRic_{\al\de}A^\de)\\
	&+\nab_\al\nab^\si \tbmF_{\be\si}-\nab_\be\nab^\si \tbmF_{\al\si}-\De_g \tbmF_{\al\be}\\
	&-\d_t \tbmF_{\al\be}
	+\nab_\al[\Re(\la^\ga_\be \overline{\nab^A_\ga \psi})-\tbmF_{\be\de}V^\de]-\nab_\be[\Re(\la^\ga_\al \overline{\nab^A_\ga \psi})-\tbmF_{\al\de}V^\de]\\
	:=&\ I_1+I_2+I_3.
	\end{align*}
	For the commutator we use the Bianchi identities to compute
	\begin{align*}
	&\ -[\nabla^\sigma \nabla_\sigma,\nabla_{\alpha}] A_\beta + [ \nabla^\sigma \nabla_\sigma,   \nabla_\beta] A_\alpha  \\
	= & \ -\nabla^\sigma (R_{\sigma \alpha \beta\delta} A^\delta - R_{\sigma \beta \alpha\delta} A^\delta)
	- (R_{\sigma \alpha \beta\delta} - R_{\sigma \beta \alpha\delta}) \nabla^\sigma A^\delta -
	{R^\sigma}_{\alpha \sigma \delta} \nabla^\delta A_\beta +{R^\sigma}_{\beta \sigma \delta} \nabla^\delta A_\alpha
	\\
	= & \ -\nabla^\sigma R_{\beta\alpha\sigma\delta} A^\delta - 2 R_{\beta\alpha\sigma\delta} \nabla^\sigma A^\delta
	- \Ric_{\alpha \delta} \nabla^{\delta} A_\beta + \Ric_{\beta \delta} \nabla^{\delta} A_\alpha
	\\
	= & \ -(\nabla_\beta \Ric_{\alpha\delta} - \nabla_\alpha R_{\beta\delta} ) A^\delta  
	- R_{\beta\alpha\sigma\delta} \bmF^{\sigma\delta}- \Ric_{\alpha \delta} ({\bmF^\delta}_\beta+\nabla_{\beta} A^\delta) +  \Ric_{\beta \delta} ({\bmF^\delta}_\alpha+ \nabla_{\alpha} A^\delta)
	\\
	= & \ -\nabla_\beta ( \Ric_{\alpha\delta} A^\delta) + \nabla_\alpha (\Ric_{\beta\delta}  A^\delta)  
	- R_{\beta\alpha\sigma\delta} \bmF^{\sigma\delta}- \Ric_{\alpha \delta} {\bmF^\delta}_\beta+  \Ric_{\beta \delta} {\bmF^\delta}_\alpha.
	\end{align*}
	We commute $\nab_\al,\ \nab_\be$ with $\nab^\si$ and use $\nab_\al \tbmF_{\be\si}+\nab_\be \tbmF_{\si\al}+\nab_\si \tbmF_{\al\be}=0$ to compute $I_2$ by
	\begin{align*}
	I_2 =& \  R_{\al\si\be\ga} {\tbmF^{\ga\si}} +
	R_{\al\si\si\ga} \tbmF^{\ \ga}_{\be} -  R_{\beta \si\al\ga} \tbmF^{\ga\si} -
	R_{\beta\si\si\ga} \tbmF^{\ \ga}_{\al}  \\
	= & \  \Ric_{\alpha \delta} {\tbmF^\delta}_{\ \beta} - \Ric_{\beta\delta} {\tbmF^\delta}_{\ \al} + R_{\beta\alpha\sigma\delta} \tbmF^{\sigma\delta}.
	\end{align*}
	Then we obtain
	\begin{equation*}
	I_1+I_2=-\Ric_{\alpha \delta} {T^{4,\delta}}_{\beta} + \Ric_{\beta\delta} {T^{4,\delta}}_{\al} - R_{\beta\alpha\sigma\delta} T^{4,\sigma\delta}
	\end{equation*}
	
	For $I_3$ we compute $\d_t \tbmF_{\al\be}$ first. 
	\begin{align*}
	\d_t \tbmF_{\al\be}=\Im(\d_t \la_{\al\si}\bar{\la}^\si_\be-\d_t \la_{\be\si}\bar{\la}^\si_\al)+\d_t g^{\si\mu}\Im(\la_{\al\si}\bar{\la}_{\be\mu})
	\end{align*}
	By the $g$-equations and
	\begin{equation*}
	i\d^{B}_t\la_{\al\be}+\nab^A_{\al}\nab^{A}_\be \psi -i\la^{\ga}_{\al}\Im(\psi\bar{\la}_{\ga\be})-i\la^{\ga}_{\al}\nab_{\be} V_{\ga}-i\la_\be^\ga\nab_\al V_\ga-iV^\ga\nab^A_\ga\la_{\al\be}=0,
	\end{equation*}
	we have
	\begin{align*}
	\Im(\d_t \la_{\al\si}\bar{\la}^\si_\be)=&-\Re(B\la_{\al\si}\bar{\la}^\si_\be-\nab^A_\al\nab^A_\si\psi \bar{\la}^\si_\be)+\Im(\la^\ga_\al\bar{\la}^\si_\be)(\Im(\psi\bar{\la}_{\ga\si})+\nab_\si V_\ga)\\
	&+\nab_\al (\tbmF^\ga_{\ \be}V_\ga)-V^\ga \Im(\la_{\ga\si}\overline{\nab^A_\al \la^\si_\be})+V^\ga \Im (T^1_{\ga\al,\si}\bar{\la}^\si_\be)
	\end{align*}
	Then we rewrite the expression $\d_t \tbmF_{\al\be}$ as 
	\begin{equation*}
	\begin{aligned}
	\d_t \tbmF_{\al\be}=&\ \nab_\al\Re(\nab^A_\si\psi \bar{\la}^\si_\be)
	-\nab_\be\Re(\nab^A_\si\psi \bar{\la}^\si_\al)-\Re(\nab^{A,\si}\psi \overline{T^3_{\al\be,\si}})\\
	&\ +\nab_\al (\tbmF^\ga_{\ \be}V_\ga)-\nab_\be (\tbmF^\ga_{\ \al}V_\ga)-V^\ga \Im(\la_{\ga}^\si\overline{T^3_{\al\be,\si}})+V^\ga \Im (T^3_{\ga\al,\si}\bar{\la}^\si_\be)\\
	=&\  \nab_\al \tbmF_{0\be}-\nab_\be \tbmF_{0\al}+\nab\psi T^3+\la V T^3.
	\end{aligned}
	\end{equation*}
	Hence, we have
	\begin{align*}
	I_3=&\ \nab_\al\Re(\nab^A_\si\psi \bar{\la}^\si_\be)
	-\nab_\be\Re(\nab^A_\si\psi \bar{\la}^\si_\al)-\Re(\nab^{A,\si}\psi \overline{T^1_{\al\be,\si}})\\
	&\ +\nab_\al (\tbmF^\ga_{\ \be}V_\ga)-\nab_\be (\tbmF^\ga_{\ \al}V_\ga)-V^\ga \Im(\la_{\ga}^\si\overline{T^1_{\al\be,\si}})+V^\ga \Im (T^3_{\ga\al,\si}\bar{\la}^\si_\be)\\
	&\ +\nab_\al[\Re(\la^\ga_\be \overline{\nab^A_\ga \psi})-\tbmF_{\be\de}V^\de]-\nab_\be[\Re(\la^\ga_\al \overline{\nab^A_\ga \psi})-\tbmF_{\al\de}V^\de]\\
	=&\ -\Re(\nab^{A,\si}\psi \overline{T^3_{\al\be,\si}})-V^\ga \Im(\la_{\ga}^\si\overline{T^3_{\al\be,\si}})+V^\ga \Im (T^3_{\ga\al,\si}\bar{\la}^\si_\be)
	\end{align*}
	This concludes the proof $T^4$-equations.
\end{proof}
\end{proof}

\bigskip

\section{The reconstruction of the flow}

In this last section we close the circle of ideas in this paper, and prove that one can start
from the good gauge solution given by Theorem~\ref{LWP-MSS-thm}, and reconstruct the flow at the level of
$d$-dimensional embedded submanifolds. For completeness, we 
provide here another, more complete statement of our 
main theorem:

\begin{thm}[Small data local well-posedness]   \label{LWP-thm-full}
	Let  $d\geq 2$ and $s>\frac{d}{2}$. Consider the skew mean curvature flow 	\eqref{Main-Sys}  for maps $F$ from $\R^d$ to the Euclidean space $(\R^{d+2},g_{\R^{d+2}})$ with  initial data $\Sigma_0$ which, in some 
	coordinates, has a metric $g_0$ satisfying $\||D|^{\si_d}(g_0-I_d)\|_{H^{s+1-\si_d}}\leq \ep_0$ and mean curvature $\lV \mathbf{H}_0  \rV_{H^s(\Sigma_0)}\leq \ep_0$. In addition, we assume that $\|g_0-I_d\|_{Y^{lo}_0}\lesssim \ep_0$ in dimension $d= 2$.
	
If $\ep_0>0$ is sufficiently small, then there exists a unique solution 
\[
F:  \R^d \times [0,1] \to (\R^{d+2},g_{\R^{d+2}})
\]
which, when represented in harmonic coordinates at the initial time and heat coordinates dynamically,
 has regularity
\[
\d_x^2 F,\ \d_t F \in C([0,1]; H^{s}(\R^d)).
\]
and induced metric and mean curvature 
\[
|D|^{\si_d}(g-I_d) \in C([0,1]; H^{s+2-\si_d}(\R^d)), \quad \mathbf{H} \in C([0,1]; H^{s}(\R^d)).
\]
In addition the mean curvature satisfies
the bounds
	\begin{equation*}
	\lV \la\rV_{l^2 X^s} +  \lV (h,A)\rV_{\bEE^s}\lesssim \lV \la_0\rV_{H^s}+\|h_0\|_{\bY_0^{s+2}}.
	\end{equation*}
where $\lambda$ and $A$ are expressed using 
the Coulomb gauge initially and the heat gauge 
dynamically in the normal bundle $N \Sigma_t$.	
\end{thm}

We prove the theorem in several steps:

\subsection{ The moving frame}
Once we have the initial data $(h_0,A_0,\la_0)$ which is small in $\HH^s\times H^s$ by \propref{Global-harmonic} and \ref{Global-harmonic-d2}, Theorem~\ref{LWP-MSS-thm} yields the good gauge local solution $\la$, along with the associated
derived variables $(h,A)$. But this does not yet give us the actual maps $F$.
 
 Here we undertake the task of reconstructing the frame $(F_\alpha, m)$. For this we use the system consisting of \eqref{strsys-cpf} and \eqref{mo-frame}, viewed as a linear ode. We recall these equations here:
 \begin{equation}               \label{strsys-cpf-re}
 \left\{\begin{aligned}
 &\d_{\al}F_{\be}=\Gamma^{\ga}_{\al\be}F_{\ga}+\Re(\lambda_{\al\be}\bar{m}),\\
 &\d_{\al}^A m=-\lambda^{\ga}_{\al} F_{\ga},
 \end{aligned}\right.
 \end{equation}
 respectively
 \begin{equation}              \label{mo-frame-re}
 \left\{\begin{aligned}
 &\d_t F_{\al}=-\Im (\d^A_{\al} \psi \bar{m}-i\la_{\al\ga}V^{\ga} \bar{m})+[\Im(\psi\bar{\la}^{\ga}_{\al})+\nab_{\al} V^{\ga}]F_{\ga},\\
 &\d^{B}_t m=-i(\d^{A,\al} \psi -i\la^{\al}_{\ga}V^{\ga} )F_{\al}.
 \end{aligned}\right.
 \end{equation}

 We start with the frame at time $t=0$, which 
 already is known to solve \eqref{strsys-cpf-re}, and 
 has the following properties:
 \begin{enumerate}[label=(\roman*)]
     \item \emph{Orthogonality}, $F_\alpha \perp m$, $\<m,m\>=2$, $\<m,\bar{m}\>=0$ and consistency with the metric $g_{\al\be} = \< F_\al,F_\be \>$.  

     \item \emph{Integrability}, $\partial_\beta F_\alpha = \partial_\alpha F_\beta$.
     \item \emph{Consistency}  with the second fundamental form and the connection $A$:
 \[
  \d_\al F_\be\cdot m=\la_{\al\be}, \qquad \< \partial_\alpha m,m\>= -2 i A_\alpha.
 \]
    
 \end{enumerate}

 Next we extend this frame to times $t > 0$ by simultaneously solving the pair of equations \eqref{strsys-cpf-re} and \eqref{mo-frame-re}.

\subsubsection{The solvability of \eqref{strsys-cpf-re} and \eqref{mo-frame-re}}
 The system consisting of \eqref{strsys-cpf-re} and \eqref{mo-frame-re}
 is overdetermined, and the necessary and sufficient condition 
 for existence of solutions is provided by Frobenius' theorem.
 We now verify these compatibility conditions in two steps:

 \medskip

 a) Compatibility conditions for the system \eqref{strsys-cpf-re} at fixed time. Here, by $T^2_{\al\be\mu\nu}=0$, $T^3_{\al\be,\ga}=0$, $T^4_{\al\be}=0 $ and  we have
 \begin{align*}
     &\d_{\al}(\Ga^\si_{\be\ga}F_\si+\Re(\la_{\be\ga}\bar{m}))-\d_{\be}(\Ga^\si_{\al\ga}F_\si+\Re(\la_{\al\ga}\bar{m}))
     =0,
 \end{align*}
 and 
 \[
 \d_\al(iA_\be m+\la^\si_\be F_\si)-\d_\be(iA_\al m+\la^\si_\al F_\si)=0,
 \]
 as needed.

 \medskip

 b) Compatibility conditions between the system \eqref{strsys-cpf-re} and \eqref{mo-frame-re}.
 By \eqref{strsys-cpf-re}, \eqref{mo-frame-re} and \eqref{dtla-rep} we have
 \begin{align*}
     \d_t(iA_\al m+\la^\si_\al F_\si)-\d_\al (iB m+i(\d^{A,\si}\psi-i\la^\si_\ga V^\ga)F_\si)
     =0
 \end{align*}
 and
 \begin{equation*}     
 \begin{aligned}
     \d_\be[-\Im (\d^A_{\al} \psi \bar{m}-i\la_{\al\ga}V^{\ga} \bar{m})+[\Im(\psi\bar{\la}^{\ga}_{\al})+\nab_{\al} V^{\ga}]F_{\ga}]-\d_t[\Ga^\ga_{\be\al}F_\ga+\Re(\la_{\be\al}\bar{m})]
     = 0.
 \end{aligned}
 \end{equation*}

 \medskip

\subsubsection{Solving the system \eqref{strsys-cpf-re}-\eqref{mo-frame-re} locally}
Starting from the existing frame at time $t=0$, we want to extend
it forward in time by solving \eqref{mo-frame-re}, while insuring that \eqref{strsys-cpf-re} remains valid. 
The difficulty is that we lack the uniform integrability in time for the coefficients in \eqref{mo-frame-re}. However, in view of the local energy decay bounds for $\lambda$ and $\psi$, we do know
that locally we have $\lambda \in L^2_t H^{s+\frac12}$.
We choose a distinguished coordinate, say $x_d$, and denote the remaining coordinates by $x'$. Then in
view of Sobolev embeddings we have the local regularity
\[
\partial \lambda \in C_{x_d} L^2_t H^{s-1}_{x'} \cap L^2_{x_d}
L^2_t C_{x'}
\]
Thus on a ``good" $x_d$ slice we have  $\partial \lambda \in L^2_t C_{x'}$ and we can extend our frame forward in time as a continuous function, with $L^2_t L^\infty_{x'}$ time derivatives and bounded
spatial derivatives. 

At fixed time all the coefficients are continuous so we can start from the above $x_d$ slice and solve the system \eqref{strsys-cpf-re} globally in $x$, obtaining a global frame 
$(F_\alpha,m)$ which is locally Lipschitz in $x$ and continuous
in $t$. By Frobenius' theorem, this solution must also satisfy 
\eqref{mo-frame-re} on any good $x_d$ slice, which is a.e.
Thus we have obtained the desired global frame $(F_\alpha,m)$
for $t \in [0,1]$.

\subsubsection{Propagating the properties (i)-(iii)}
Here we show that the properties (i)-(iii) above also extend to all $t \in [0,1]$. 
The properties (ii) and (iii) follow directly from the equations \eqref{strsys-cpf-re} and \eqref{mo-frame-re} once 
the orthogonality conditions in (i) are verified.
We  denote 
 \begin{equation*}  
     \tg_{00}=\<m,m\>,\quad \tg_{\al 0}=\<F_\al, m\>,\quad \tg_{\al\be}=\<F_\al,F_\be\>.
 \end{equation*}

The first step is to propagate (i) forward in time on a good $x_d$
slice. Indeed, by \eqref{mo-frame-re} and \eqref{dtg-formula}
we have
\begin{align*}
 &\begin{aligned}
     \d_t \tg_{\al 0}=&-\frac{i}{2}(\overline{\d^A_\al \psi}+i \bar{\la}_{\al\ga}V^\ga)(\tg_{00}-2)-i(\overline{\d^{A,\si} \psi}+i \bar{\la}^\si_{\ga}V^\ga)(g_{\al\si}-\tg_{\al\si})\\
     &+\frac{i}{2}(\d^A_\al \psi+i \la_{\al\ga}V^\ga)\<\bar{m},m\>+(\Im(\psi\bar{\la}_\al^\ga)+\nab_\al V^\ga)\tg_{\ga 0}+iB \tg_{\al 0},
 \end{aligned}
     \\
     &\d_t (\tg_{00}-2)=2\Im [(\d^{A,\al}\psi-i \la^\al_\ga V^\ga) \tg_{\al 0}],\\
     &\d_t \<m,\bar{m}\>= -2iB\<m,\bar{m}\>-2i (\d^{A,\al}\psi-i \la^\al_\ga V^\ga) \bar{\tg}_{\al 0},\\
     &\begin{aligned}
     \d_t(g_{\al\be}-\tg_{\al\be})=&\ (\Im(\psi\bar{\la}_\al^\ga)+\nab_\al V^\ga)(g_{\be\ga}-\tg_{\be\ga})+(\Im(\psi\bar{\la}_\be^\ga)+\nab_\be V^\ga)(g_{\al\ga}-\tg_{\al\ga})\\
     &\ +\Im (\d^A_\al \psi \tg_{\be 0}-i\la_{\al\ga}V^\ga \tg_{\be 0})-\Im (\d^A_\be \psi \bar{\tg}_{\al 0}-i\la_{\be\ga}V^\ga \bar{\tg}_{\al 0}).
     \end{aligned}
 \end{align*}
 Viewed as a linear system of ode's in time, these equations allow us to propagate (i) in time, given that it is satisfied at $t = 0$.

It remains to propagate (i) spatially.
Using \eqref{strsys-cpf-re} we compute
\begin{align*}
&\d_\al \tg_{\be 0}=\Ga^\ga_{\al\be} \tg_{\ga 0}+\frac{1}{2}\la_{\al\be}\<\bar{m},m\>+\frac{1}{2}\bar{\la}_{\al\be}(\tg_{00}-2)+\bar{\la}^\ga_\al (g_{\be\ga}-\tg_{\be\ga})+iA_\al \tg_{\be 0},\\
&\d_{\al}(\tg_{00}-2)=-2\Re(\la^\ga_\al \tg_{\ga 0}),\\
&\d_\al \<m,\bar{m}\>=-2i A_\al \<m,\bar{m}\>-2\Re\la^\ga_\al \bar{\tg}_{\ga 0},\\
&\d_\al(g_{\be\ga}- \tg_{\be\ga})=\Ga^\si_{\al\be}(g_{\si\ga}-\tg_{\si\ga})+\Ga^\si_{\al\ga}(g_{\si\be}-\tg_{\si\be})+\Re(\bar{\la}_{\be\al}\tg_{\ga 0}+\bar{\la}_{\ga\al}\tg_{\be 0}).
\end{align*}
By ode uniqueness and the choice of the initial data, the desired properties (i) for the frame are indeed propagated spatially.

\subsubsection{ The Sobolev regularity of the frame }
Here we show that our frame has the global regularity
\begin{equation*}
 \partial_x(F_\alpha, m) \in L^\infty H^{s}, \qquad \partial_t     (F_\alpha, m)
 \in L^\infty H^{s-1}. 
\end{equation*}

As a consequence of the property (i), we directly see that 
$(F_\alpha,m) \in L^\infty$. From \eqref{strsys-cpf-re}
it then follows that $\partial_x (F_\alpha,m) \in L^\infty$.
This allows us to differentiate further in  \eqref{strsys-cpf-re}
and bound higher derivatives of the frame, up to the 
$H^{s}$ regularity for $\partial_x(F_\alpha,m)$, which is imposed by $\lambda$. We can directly estimate this last norm.
Precisely, by \eqref{strsys-cpf-re}, \eqref{psi-full-reg} and Sobolev embeddings we have
\begin{align*}
    \|\d_x F_\al\|_{H^s}\lesssim &\ \|\Ga F_\ga+\la m\|_{H^s}\\
    \lesssim &\ \|\Ga\|_{H^s}\|F_\ga\|_{L^\infty\cap \dot{H}^s}+\|\la\|_{H^s}\|m\|_{L^\infty\cap \dot{H}^s}\\
    \lesssim &\ \ep_0 (\|g\|^{1/2}_{L^\infty}+\|\d_x F_\al\|_{H^s})+\ep_0(1+\|\d_x m\|_{H^s})\\
    \lesssim & \ \ep_0(1+\|\d_x F_\al\|_{H^s}+\|\d_x m\|_{H^s})
\end{align*}
and 
\begin{align*}
    \|\d_\al m\|_{H^{s}}\lesssim &\  \|Am+\la F_\ga\|_{H^{s}}\\
    \lesssim &\ \| A\|_{H^{s}}\|m\|_{L^\infty\cap \dot{H}^s}+\|\la\|_{H^s}\|F_\ga\|_{L^\infty\cap \dot{H}^s}\\
    \lesssim &\ \ep_0(1+\|\d_x F_\al\|_{H^s}+\|\d_\al m\|_{H^{s}}).
\end{align*}
These imply the uniform bound
\begin{equation*}
    \|\d_x F_\al\|_{H^s}+\|\d_x m\|_{H^{s}}\lesssim \ep_0.
\end{equation*}

 \subsection{ The moving manifold \texorpdfstring{$\Sigma_t$}{}} Here we propagate the full map $F$ 
 by simply integrating \eqref{sys-cpf}, i.e.
 \begin{align*}
     F(t)=F(0)+\int_0^t -\Im(\psi \bar{m})+V^\ga F_\ga ds.
 \end{align*}
 Then by \eqref{strsys-cpf-re}, we have
 \begin{align*}
     \d_\al F(t)=\d_\al F(0)+\int_0^t -\Im (\d^A_{\al} \psi \bar{m}-i\la_{\al\ga}V^{\ga} \bar{m})+[\Im(\psi\bar{\la}^{\ga}_{\al})+\nab_{\al} V^{\ga}]F_{\ga} ds,
 \end{align*}
 which is consistent with above 
 definition of $F_\alpha$.

 \subsection{ The (SMCF) equation for \texorpdfstring{$F$}{}}
 Here we establish that $F$ solves \eqref{Main-Sys}. Using the relation  
 $\la_{\al\be}=\d_\al \d_\be F \cdot m$
 we have
 \begin{align*} 
 -\Im (\psi\bar{m})=&\ -\Im (g^{\al\be}\d_\al\d_\be F\cdot (\nu_1+i\nu_2)\ (\nu_1-i\nu_2))\\
 =&\ (\De_g F\cdot \nu_1) \nu_2-(\De_g F\cdot \nu_2) \nu_1\\
 =&\ J (\De_g F)^{\perp}=J\mathbf{H}(F).
 \end{align*}
 This implies that the $F$ solves \eqref{Main-Sys}.

\end{document}